\documentclass[reqno,10pt,centertags]{amsart}
\usepackage{amsmath,amsthm,amscd,amssymb,latexsym,esint,upref,stmaryrd,
enumerate,color,verbatim,yfonts}
\usepackage{color}
\usepackage{hyperref} 
\newcommand*{\mailto}[1]{\href{mailto:#1}{\nolinkurl{#1}}}
\newcommand{\arxiv}[1]{\href{http://arxiv.org/abs/#1}{arXiv:#1}}

\allowdisplaybreaks 
\numberwithin{equation}{section}

\newtheorem{theorem}{Theorem}[section]
\newtheorem{proposition}[theorem]{Proposition}
\newtheorem{lemma}[theorem]{Lemma}
\newtheorem{corollary}[theorem]{Corollary}
\newtheorem{definition}[theorem]{Definition}
\newtheorem{hypothesis}[theorem]{Hypothesis}

\theoremstyle{remark}
\newtheorem{remark}[theorem]{Remark}

\newcommand{\beq}{\begin{equation}}
\newcommand{\enq}{\end{equation}}

\newcommand{\C}{{\bbC}}

\newcommand{\bbC}{{\mathbb{C}}}

\newcommand{\bbN}{{\mathbb{N}}}

\newcommand{\bbR}{{\mathbb{R}}}
\newcommand{\bbT}{{\mathbb{T}}}

\newcommand{\bbZ}{{\mathbb{Z}}}

\newcommand{\bsA}{{\boldsymbol{A}}}
\newcommand{\bsB}{{\boldsymbol{B}}}
\newcommand{\bsC}{{\boldsymbol{C}}}
\newcommand{\bsD}{{\boldsymbol{D}}}

\newcommand{\bsH}{{\boldsymbol{H}}}
\newcommand{\bsI}{{\boldsymbol{I}}}

\newcommand{\bsP}{{\boldsymbol{P}}}

\newcommand{\bsT}{{\boldsymbol{T}}}

\newcommand{\cB}{{\mathcal B}}

\newcommand{\cE}{{\mathcal E}}
\newcommand{\cF}{{\mathcal F}}

\newcommand{\cH}{{\mathcal H}}

\newcommand{\cK}{{\mathcal K}}

\newcommand{\cN}{{\mathcal N}}

\DeclareMathOperator{\dom}{dom}
\DeclareMathOperator{\tr}{tr}

\DeclareMathOperator*{\nlim}{n-lim}
\DeclareMathOperator*{\slim}{s-lim}

\DeclareMathOperator*{\sgn}{sgn}

\renewcommand{\Re}{\text{\rm Re}}
\renewcommand{\Im}{\text{\rm Im}}
\renewcommand{\ln}{\text{\rm ln}}
\newcommand{\pv}{\text{\rm p.v.}}

\newcommand{\loc}{\operatorname{loc}}

\newcommand{\ind}{\operatorname{index}}
\newcommand{\no}{\notag}
\newcommand{\lb}{\label}
\newcommand{\f}{\frac}

\newcommand{\ol}{\overline}

\newcommand{\wti}{\widetilde}

\newcommand{\bi}{\bibitem}

\newcommand{\mnn}[1]{M^{m\times m}\left(#1\right)}

\newcommand{\bca}{\left(\begin{array}{c}}

\newcommand{\eca}{\end{array}\right)}

\begin{document}
\title{Trace Formulas for a Class of non-Fredholm Operators: A Review}

\author[Carey]{Alan Carey}  
\address{Mathematical Sciences Institute, Australian National University, 
Kingsley St., Canberra, ACT 0200, Australia
and School of Mathematics and Applied Statistics, University of Wollongong, NSW, Australia,  2522}  
\email{\mailto{acarey@maths.anu.edu.au}}
\urladdr{\url{http://maths.anu.edu.au/~acarey/}}
  
\author[F.\ Gesztesy]{Fritz Gesztesy}
\address{Department of Mathematics,
University of Missouri, Columbia, MO 65211, USA}
\address{Present address: Department of Mathematics
Baylor University, One Bear Place \#97328,
Waco, TX 76798-7328, USA}
\email{\mailto{Fritz\_Gesztesy@baylor.edu}}
\urladdr{\url{http://www.baylor.edu/math/index.php?id=935340}}

\author [Grosse]{Harald Grosse}
\address{Faculty of Physics, University of Vienna, Boltzmanngasse 5, A-1090 Vienna, Austria}
\email{harald.grosse@univie.ac.at}

\author[Levitina]{Galina Levitina} 
\address{School of Mathematics and Statistics, UNSW, Kensington, NSW 2052,
Australia} 
\email{\mailto{g.levitina@student.unsw.edu.au}}

\author[Potapov]{Denis Potapov} 
\address{School of Mathematics and Statistics, UNSW, Kensington, NSW 2052,
Australia} 
\email{\mailto{d.potapov@unsw.edu.au}}

\author[Sukochev]{Fedor Sukochev}
\address{School of Mathematics and Statistics, UNSW, Kensington, NSW 2052,
Australia} 
\email{\mailto{f.sukochev@unsw.edu.au}}

\author[Zanin]{Dmitriy Zanin} 
\address{School of Mathematics and Statistics, UNSW, Kensington, NSW 2052,
Australia} 
\email{\mailto{d.zanin@unsw.edu.au}}
\thanks{Submitted to {\it Rev. Math. Phys.}}


\date{\today}

\begin{abstract} 
Take a one-parameter family of self-adjoint Fredholm operators $\{A(t)\}_{t\in \bbR}$ 
on a Hilbert space 
$\mathcal H$, joining endpoints $A_\pm$.  There is a long history of work on the question of whether the spectral flow along this path
is given by the index of the operator $\bsD_{\bsA}^{}= (d/d t) +\bsA$ acting in 
$L^2(\bbR; \mathcal H)$, where $\bsA$ denotes the multiplication operator
$(\bsA f)(t) = A(t)f(t)$ for $f\in \dom(\bsA)$. Most results are about the case
where the operators $A(\cdot)$ have compact resolvent. In this article we review what is known
when these operators have some essential spectrum and describe some new results.

Using the operators
$\bsH_1=\bsD_\bsA^*\bsD^{}_\bsA$, $\bsH_2=\bsD_\bsA^{}\bsD_\bsA^*$, 
 an abstract trace formula for Fredholm  operators with  essential spectrum was proved in \cite{GLMST11}, extending a result of Pushnitski \cite{Pu08}, although, still under strong 
 hypotheses on $A(\cdot)$:
$$ 
 \tr_{L^2(\bbR; \mathcal H)}
\big((\bsH_2 - z \, \bsI)^{-1}-(\bsH_1 - z \, \bsI)^{-1}\big) 
 = \frac{1}{2z}\tr_{L^2(\mathcal H)} (g_z(A_+)-g_z(A_-)),  
$$
where $g_z(x)=x(x^2-z)^{-1/2}$, $x \in \bbR$, $z \in \bbC \backslash [0,\infty)$. Associated to the pairs $(\bsH_2, \bsH_1)$ and 
$(A_+,A_-)$ are Krein spectral shift functions
$\xi(\, \cdot \, ; \bsH_2, \bsH_1)$ and $\xi(\, \cdot \, ; A_+,A_-)$ respectively.
From the trace formula it was shown that there is a second, 
Pushnitski-type, formula: 
$$
\xi(\lambda; \bsH_2, \bsH_1)=\frac{1}{\pi}\int_{-\lambda^{1/2}}^{\lambda^{1/2}}
\frac{\xi(\nu; A_+,A_-)\, d\nu}{(\lambda-\nu^2)^{1/2}} \, 
\text{  for a.e.~$\lambda>0$.} 
$$
This can be employed to establish the desired equality, 
$$
{\textit Fredholm \; index = \xi(0; A_+,A_-) = spectral flow} 
$$  
This equality was generalized to non-Fredholm operators in \cite{CGPST15} in the form 
$$
{\textit Witten \; index = [\xi_R(0; A_+,A_-) + \xi_L(0; A_+,A_-)]/2},  
$$
replacing the Fredholm index on the LHS by the Witten index of $\bsD_{\bsA}$ and 
$\xi(0; A_+,A_-)$ on the RHS by an appropriate arithmetic mean (assuming $0$ is a right and  left Lebesgue point for $\xi(\, \cdot \, ; A_+,A_-)$ denoted by $\xi_R(0; A_+,A_-)$ and 
$\xi_L(0; A_+,A_-)$, respectively).  But this applies only under
the restrictive assumption that the endpoint $A_+$ is a relatively trace class perturbation 
of $A_-$ (ruling out general differential operators).

In addition to reviewing this previous work we  describe in this article some extensions  
using  a $(1+1)$-dimensional setup, where $A_\pm$ are non-Fredholm differential operators. 
By a  careful analysis  we prove,  for a class of examples, that the preceding trace formula still 
holds in this more general situation. Then we prove that the  Pushnitski-type formula for 
spectral shift functions also holds and this then gives the equality of spectral shift functions 
in the form
$$
\xi(\lambda; \bsH_2, \bsH_1) = \xi(\nu; A_+,A_-) \, 
\text{  for a.e.~$\lambda>0$ and a.e.~$\nu \in \bbR$,} 
$$
for the $(1+1)$-dimensional model operator at hand. This shows that neither the relatively trace class perturbation assumption nor the Fredholm assumption are required if one works with spectral shift functions. The results support the view that the spectral shift function should be a replacement for the spectral flow in certain non-Fredholm situations and  also point the way to the study of higher-dimensional cases.
We discuss the connection with summability questions in Fredholm modules in an appendix. 

\end{abstract}

\date{\today}
\maketitle

\newpage 

{\scriptsize{\tableofcontents}}

\section{Introduction and Review} \lb{s1} 

The issue of the relationship between the spectral flow and the Fredholm index was first raised in the work of Atiyah--Patodi--Singer \cite{APS76} and settled in the most definitive fashion for certain self-adjoint differential operators with compact resolvent in a paper of Robbin--Salamon, \cite{RS95}. For differential operators on noncompact manifolds it is typically the case that they possess some essential spectrum. An extension of the result of \cite{RS95} to this situation and its relationship to scattering theory was initiated in \cite{GLMST11} following \cite{Pu08}. However, the key assumption in \cite{GLMST11} is that one considers the spectral flow between self-adjoint operators that differ by a relatively trace class perturbation.  This latter assumption is violated in general for differential operators (although, not necessarily for certain classes of pseudo-differential operators). Indeed, as the bulk of the available literature focuses on systems with purely discrete spectra, there is relatively little work available in the way of index formulas for operators with essential spectrum except for \cite{BGGSS87}, \cite{BMS88} and previous work by the present authors. Motivation for this study stems, for example, from the fact that the spectral flow is a useful tool in condensed matter theory
where the operators that arise do have some essential spectrum \cite{St96}.

In the first two sections of this article we review previous work and also initiate our main objective of explaining extensions of previous efforts
(in particular, \cite{GLMST11}) so as to apply to differential operators in higher dimensions.
We will focus on examples of the non-Fredholm case motivated by recent progress in 
\cite{{CGLPSZ14}}--\cite{CGPST15}.
  
\begin{remark} \lb{Dirac} The critical fact in connection with partial differential operators is the relative Schatten--von Neumann class constraint.
To describe this, suppose for example that we have the flat space Dirac-type operator $A_-$ acting in $L^2(\bbR^n) \otimes \bbC^m$ and a smooth, $m \times m$ matrix-valued bounded function $F:\bbR^n\to M^{m \times m}(L^{\infty}(\bbR^n) \cap C^{\infty}(\bbR^n))$, $m \in \bbN$.  Then $F$ acts as a bounded $m \times m$ matrix-valued multiplication operator on $L^2(\bbR^n) \otimes \bbC^m$. 
Under suitable decay conditions at infinity for $F$, the product $F (1+A_-^2)^{-s/2}$ is trace class for $s>n$ and no smaller value of $s$ (see \cite[Remark~4.3]{Si05}). \hfill $\diamond$
\end{remark}

Thus, in differential operator terms, \cite{CGPST15, GLMST11}
considered the zero-dimensional case
 because,  if we write in the notation of the remark, $A_+=A_-+F$, then 
 $F(1+A_-^2)^{-1/2}=(A_+-A_-)(1+A_-^2)^{-1/2}$ 
which is trace class only if $n=0$. In this article, following the ideas in \cite{CGLS14}, we consider
the situation when $(A_+-A_-)(1+A_-^2)^{-1/2}$ is Hilbert--Schmidt and show how this allows some general one-dimensional
examples. We note that our results point the way to an attack on the problem of partial differential operators in higher dimensions.

\subsection{The Model Operator.} 
To make the discussion precise we start by introducing the model operators
that form the basis of study in later sections. In this instance  $A_-$,  acting in the Hilbert 
space $L^2(\bbR)\otimes \bbC^m$, $m\in\bbN$, is the self-adjoint ``chiral Dirac operator'' 
\begin{equation} 
A_-=- i \frac{d}{dx} \otimes I_m, \quad \dom(A_-) = W^{1,2}(\bbR) \otimes \bbC^m.
\end{equation} 
For a matrix potential
$\Phi \in \mnn{L^\infty(\bbR)}$ with essentially bounded entries we will also use the 
abbreviation $\Phi$ for the associated bounded operator acting by multiplication on 
$L^2(\bbR)\otimes\bbC^m$. 

Under certain assumptions on a bounded real-valued function $\theta$ on $\bbR$ and a 
self-adjoint $m \times m$ matrix-valued function $\Phi$ on $\bbR$ we consider the family of operators
\begin{equation} 
A(t)=A_-+\theta(t) \Phi, \quad \dom(A(t)) = W^{1,2}(\bbR) \otimes \bbC^m, \; t \in \bbR,  
\end{equation} 
and the associated operator $\bsA$ on the Hilbert space 
$L^2(\bbR; dt; L^2(\bbR;dx))\otimes\bbC^m$, which we will identify with $L^2(\bbR^2; dtdx) \otimes \bbC^m$ (in short, with $L^2(\bbR^2) \otimes \bbC^m$), defined by 
\begin{align}\lb{def_bsA}
&(\bsA f)(t) = A(t) f(t) \, \text{ for a.e.~$t\in\bbR$,}   \no \\
& \, f \in \dom(\bsA) = \big\{g \in L^2(\bbR^2) \otimes \bbC^m\,\big|\,
g(t)\in W^{1,2}(\bbR) \otimes \bbC^m \text{ for a.e.~$t\in\bbR$,}     \\
& \quad\;\,\, t \mapsto A(t) g(t) \text{ is (weakly) measurable,} \,  
\int_{\bbR} \|A(t) g(t)\|_{L^2(\bbR) \otimes \bbC^m}^2 \, dt < \infty\big\}.   \no 
\end{align}
Our hypothesis ensures the existence of the asymptote 
\begin{equation} 
A_+=A_-+ \Phi, \quad \dom(A_+) = W^{1,2}(\bbR) \otimes \bbC^m,
\end{equation} 
as a norm resolvent limit of $A(t)$ as $ t \to \infty$. (Similarly, $A(t)$ converges in the norm resolvent sense to $A_-$ as $t \to - \infty$.) {\it We will show later that the operators 
$A_+$ and $A_-$ are unitarily equivalent and thus both have, as continuous spectrum, the whole real line. In particular, $A_{\pm}$ are not Fredholm.}

Next, we introduce the densely defined, closed operator $\bsD_\bsA^{}$ acting in 
$L^2(\bbR^2)\otimes\bbC^m$ by setting
\begin{equation} \lb{model1}
\bsD^{}_\bsA=\frac{d}{dt} \otimes I_m + \bsA, \quad 
\dom(\bsD^{}_{\bsA}) = W^{1,2}(\bbR^2) \otimes \bbC^m, 
\end{equation} 
with $I_m$ the identity operator in $\bbC^m$. 
We also define self-adjoint operators $\bsH_1$ and $\bsH_2$ acting in 
$L^2(\bbR^2) \otimes \bbC^m$, by 
$
\bsH_1=\bsD_\bsA^*\bsD_\bsA^{},\quad \bsH_2=\bsD_\bsA^{}\bsD_\bsA^*. 
$ 
Clearly  $A_\pm$ are one-dimensional differential operators and $\bsH_1,\bsH_2$ are 
two-dimensional, hence our terminology: this situation describes the $(1+1)$-dimensional case.

It should be noted that the family of bounded operators 
$\{B(t)=\theta(t) \Phi\}_{t\in\bbR}$ on $L^2(\bbR) \otimes \bbC^m$ do not satisfy the 
assumptions in \cite{Pu08}. 

Next, we review previous work starting with \cite{Pu08}.

{\it{\bf  The Pushnitski Assumptions.}
Let $\cH$ be a complex, separable Hilbert space. \\
$(i)$ Assume $A_- \in \cB(\cH)$ is self-adjoint in $\cH$. \\
$(ii)$ Suppose there exists a family of bounded self-adjoint operators $\{B(t)\}_{t\in\bbR}$ 
$($the allowed perturbations of $A_-$$)$ 
in $\cH$ with $B(\cdot)$ weakly locally absolutely continuous on $\bbR$, implying the 
existence of a family of bounded self-adjoint operators $\{B'(t)\}_{t\in\bbR}$ in $\cH$ such that for 
a.e.\ $t\in\bbR$,  
\begin{equation} 
\frac{d}{dt} (g,B(t) h)_{\cH} = (g,B'(t) h)_{\cH}, \quad g, h\in\cH.     \lb{2.1}
\end{equation} 
$(iii)$ Assume that $B'(t) \in \cB_1(\cH)$, $t \in \bbR$ $($cf.\ our choice of notation for trace ideals described in Subsection \ref{ss1.5}$)$, and 
\begin{equation}  \lb{2.2}
\int_\bbR \big\|B'(t)\big\|_{\cB_1(\cH)} \, dt < \infty.
\end{equation}}

Using these assumptions the trace formula and the Pushnitski-type formula stated in the abstract 
are proved.  These results motivated the paper \cite{GLMST11} which sought to prove analogous
results under weaker (i.e., relative trace class) conditions. For comparison we state the key assumption of \cite{GLMST11} on these perturbations that replaces item $(iii)$ of the Pushnitski assumptions.

{\it {\bf The GLMST assumptions.}~$(iii^\prime)$ Assume the relatively trace class perturbation assumption
\begin{equation}  
\big\|B'(t)(A_-^2 + I_{\cH})^{-1/2}\big\|_{\cB_1(\cH)} < \infty \mbox{ and }
\int_\bbR \big\|B'(t)(A_-^2 + I_{\cH})^{-1/2}\big\|_{\cB_1(\cH)}\, dt < \infty.
\end{equation} 
}

We emphasize that the operators $\{B(t)=\theta(t) \Phi\}_{t\in\bbR}$, though deceptively simple, 
satisfy neither $(iii)$ nor $(iii^\prime)$ (see Remarks \ref{Dirac} and \ref{not_rtc}).
This motivated the paper \cite{CGLS14} which seeks to provide an abstract framework for
generalizations of the Pushnitski results.

However \cite{CGLS14} alone is not enough to establish a 
trace formula of the kind stated in our abstract for the model operator (\ref{model1}).  Nor can we use 
\cite{CGPST15}, where non-Fredholm operators were studied, as it needs assumption 
$(iii^\prime)$ as well and so cannot be applied to our present context.

Remark \ref{Dirac} shows that 
one has, for the examples we study in this paper, a relatively Hilbert--Schmidt perturbation condition.
This Hilbert--Schmidt constraint also appears in  \cite{CGLS14} in an abstract setting where it is used to obtain a Pushnitski-type formula. 

While this progression, going from the $(0+1)$-dimensional case in \cite{GLMST11} to the $(1+1)$-dimensional case both here and in \cite{CGLS14} may appear incremental at first sight, approximation methods in \cite{CGLS14} and \cite{CGLNPS16} to make progress on the general case of model operators in higher dimensions are in preparation. Hence our interest in giving here an accessible exposition via a class of models  of this new  approach.

In the following we will explain further the relevance of previous papers to the current investigation as well as discuss in more detail the model operator that forms the main
focus of this article.

\subsection{The Witten Index.}
In this subsection we briefly review the notion of a (resolvent regularized) Witten index following 
\cite{CGLS14} and \cite{CGPST15}. 

We start by recalling the hypotheses used in  \cite{CGLS14}. 

\begin{hypothesis} \lb{h3.4} 
Suppose $\cH$ is a complex, separable Hilbert space with  $A_-$  self-adjoint on 
$\dom(A_-) \subseteq \cH$. \\
$(i)$ Suppose we have a family of bounded operators $A(t)=A_-+B(t)$, $t \in \bbR$, satisfying 
$\{B(t)\}_{t \in \bbR} \subset \cB(\cH)$, which is continuously differentiable in norm on 
$\bbR$ and such that 
\begin{equation}
\|B'(\cdot)\|_{\cB(\cH)} \in L^1(\bbR; dt).     \lb{intB'a}
\end{equation} 
$(ii)$ Suppose that $|B_+|^{1/2}(A_- -z_0 I)^{-1} \in \cB_2(\cH)$ for some $($and hence 
for all\,$)$ $z_0 \in \rho(A_-)$. $($Here $B_+ = \nlim_{t \to +\infty} B(t)$, and we again refer to our choice of notation for trace ideals described in Subsection \ref{ss1.5}.$)$ \\
$(iii)$ Assume that $\sup_{t \in \bbR} \|B'(t)\|_{\cB(\cH)} < \infty$. \\  
$(iv)$ Acting in the space $L^2(\bbR; \cH)$ we have the operators
$\bsD_{\bsA_-}^{}= \frac{d}{dt} + \bsA_-^{}$, $\bsH_0=\bsD_{\bsA_-}^*\bsD_{\bsA_-}^{}$,
$\bsD_{\bsA}= \frac{d}{dt} + \bsA$,
and $\bsB$, the operator of multiplication by the family $\{B(t)\}_{t\in\bbR}$. Suppose 
that $\bsA_- \bsB$ is bounded with respect to $\bsH_0$ with bound strictly 
less than one, that is, there exists $0 \leq a < 1$ and $b \in (0,\infty)$ such that 
\begin{equation}
\|\bsA_- \bsB f\|_{L^2(\bbR; \cH)} \leq a \|\bsH_0 f\|_{L^2(\bbR; \cH)} 
+ b \|f\|_{L^2(\bbR; \cH)}, \quad f \in \dom(\bsH_0).
\end{equation} 
$(v)$ Suppose that
\begin{align} \lb{rel_H-S}
\begin{split} 
& |B'(t)|^{1/2}(|A_-|+I)^{-1}\in\cB_2(\cH), \quad t \in \bbR,  \\ 
& \big\||B'(\cdot)|^{1/2}(|A_-|+I)^{-1}\big\|_{\cB_2(\cH)}\in L^2(\bbR; dt).
\end{split} 
\end{align} 
\end{hypothesis}

The definition of the Witten index used in the next result is taken from 
\cite{CGLS14} and is reviewed in the present article  
 in Section \ref{WI_section}. The following result also uses the Krein spectral shift 
 function for the pairs $(A_+, A_-)$ and $(\bsH_2, \bsH_1)$, 
and we  previously reviewed the relevant background on this in 
our conference proceedings article \cite{CGLS15}. For more detailed information 
we refer the reader to \cite{BY93}, \cite[Ch.~8]{Ya92}, \cite[Sect.~0.9]{Ya10}, here we just mention the following facts on the spectral shift function 
$\xi(\, \cdot \, ; A, A_0)$  corresponding to a pair of self-adjoint operators $(A,A_0)$ in 
some separable, complex Hilbert space $\cK$, under the assumption that for some (and hence for all\,) $z_0 \in \rho(A) \cap \rho(A_0)$, 
\begin{equation}
\big[(A - z_0 I_{\cK})^{-1} - (A_0 - z_0 I_{\cK})^{-1}\big] \in \cB_1(\cK).    \lb{B.26a} 
\end{equation}
Introduces the modified perturbation determinant,
\begin{align}
\begin{split} 
\wti D_{A/A_0}(z;z_0) = {\det}_{\cK} 
\big((A - z I_{\cK})(A - \ol{z_0} I_{\cK})^{-1} (A_0 - \ol{z_0} I_{\cK})(A_0 - z I_{\cK})^{-1}\big),&  \\     
z \in \rho(A) \cap \rho(A_0), \; \Im(z_0) > 0,&     \lb{B.27a}
\end{split} 
\end{align}
and notes that (cf.\ \cite[p.270]{Ya92}) 
\begin{equation}
\ol{\wti D_{A/A_0} (z; z_0)} = \wti D_{A/A_0} (\ol z; z_0)/ \wti D_{A/A_0} (z_0; z_0), \quad 
 \wti D_{A/A_0} (\ol{z_0}; z_0) =1, 
\end{equation}
and 
\begin{align}
\begin{split} 
{\tr}_{\cK} \big[(A - z I_{\cK})^{-1} - (A_0 - z I_{\cK})^{-1}\big] 
= - \f{d}{dz} \ln\big( \wti D_{A/A_0} (z; z_0)\big),&  \\ 
z \in \rho(A) \cap \rho(A_0), \; \Im(z_0) > 0.& 
\end{split}
\end{align}
In addition,
\begin{equation}
\f{\wti D_{A/A_0} (z; z_0)}{\wti D_{A/A_0} (\ol{z}; z_0)} 
= \f{\wti D_{A/A_0} (z; z_1)}{\wti D_{A/A_0} (\ol{z}; z_1)}, \quad   
z \in \rho(A) \cap \rho(A_0), \; \Im(z_0) > 0, \, \Im(z_1) > 0. 
\end{equation}
Then, defining
\begin{align}
\begin{split}
& \xi(\lambda; A,A_0; z_0) = (2\pi)^{-1} \lim_{\varepsilon \downarrow 0} 
\big[\Im\big(\ln\big(\wti D_{A/A_0} (\lambda + i \varepsilon; z_0)\big)\big)  \lb{B.31a} \\
& \hspace*{4.2cm} - \Im\big(\ln\big(\wti D_{A/A_0} (\lambda - i \varepsilon; z_0)\big)\big)\big]
\, \text{ for a.e.~$\lambda \in \bbR$,} 
\end{split}
\end{align}
one obtains for $z \in \rho(A) \cap \rho(A_0)$, $\Im(z_0) > 0$, $\Im(z_1) > 0$, 
\begin{align}
&  \xi(\, \cdot \,; A,A_0; z_0) \in L^1\big(\bbR; (\lambda^2 + 1)^{-1} d\lambda\big),    \\
& \ln\big(\wti D_{A/A_0} (z; z_0)\big) = \int_{\bbR} \xi(\lambda; A,A_0; z_0) d\lambda 
\big[(\lambda -z)^{-1} - (\lambda - \ol{z_0})^{-1}\big],    \\
& \xi(\lambda; A,A_0; z_0) = \xi(\lambda; A,A_0; z_1) + n(z_0,z_1) \, 
\text{ for some $n(z_0,z_1) \in \bbZ$,}     \lb{B.34a} \\
& {\tr}_{\cK} \big[(A - z I_{\cK})^{-1} - (A_0 - z I_{\cK})^{-1}\big] 
= - \int_{\bbR} \f{\xi(\lambda; A,A_0; z_0) d \lambda}{(\lambda - z)^2},    \\
& [f(A) - f(A_0)] \in \cB_1(\cK), \quad f \in C_0^{\infty}(\bbR),    \\
& {\tr}_{\cH} (f(A) - f(A_0)) = 
\int_{\bbR} \xi(\lambda; A,A_0; z_0) d\lambda \, f'(\lambda), 
\quad f \in C_0^{\infty}(\bbR)  
\end{align}
(the final two assertions can be greatly improved). In this context $\xi(\, \cdot \,; A,A_0; z_0)$ is defined up to an integer. The latter can be fixed giving rise to $\xi(\, \cdot \,; A,A_0)$ as discussed in Section \ref{s5}. 

\begin{theorem} \lb{t8.3} 
Assume Hypothesis \ref{h3.4} and assume that $0$ is a right 
and a left Lebesgue point of $\xi(\,\cdot\,\, ; A_+, A_-)$ $($denoted by $\xi_R(0; A_+,A_-)$ 
and $\xi_L(0; A_+, A_-)$, respectively\,$)$. In addition, consider 
\begin{equation} 
\bsH_1=\bsD_\bsA^*\bsD^{}_\bsA,\quad \bsH_2=\bsD_\bsA^{}\bsD_\bsA^*. 
\end{equation} 
Then $0$ is a right Lebesgue point of 
$\xi(\,\cdot\,\, ; \bsH_2, \bsH_1)$ $($denoted by $\xi_R(0; \bsH_2, \bsH_1)$$)$ 
and $W_r(\bsD_\bsA^{})$ exists and equals 
\begin{equation}
W_r(\bsD_\bsA^{}) = \xi_R(0; \bsH_2, \bsH_1) 
= [\xi_R(0; A_+,A_-) + \xi_L(0; A_+, A_-)]/2.     \lb{8.5a}
\end{equation}
In particular, if $0 \in \rho(A_+) \cap \rho(A_-)$, then $\bsD_\bsA^{}$ is Fredholm and
\begin{equation} 
\ind(\bsD_\bsA^{}) = W_r(\bsD_\bsA^{}) = \xi(0; A_+, A_-).     \lb{8.6} 
\end{equation} 
\end{theorem}

\subsection{The Principle Trace Formula.} We summarize the main result of 
Sections \ref{s3} and \ref{s4}.
The principle trace formula was obtained in \cite{GLMST11} and \cite{Pu08}.  
There is a generalization in \cite{CGK} that applies to all space dimensions, that is, it handles 
Dirac-type operators in dimensions $n \in \bbN$ under certain technical assumptions. 
However, in this article we will not follow \cite{CGK} due to the complexity of the argument
given there and the fact that still further effort is needed to establish the results described here. 
In fact, for the case of the examples under discussion, a direct proof is somewhat
more instructive.

\begin{theorem}\lb{main1} Assume Hypothesis \ref{h3.1}, let 
$z\in\C\backslash [0,\infty)$ and $g_z(x)=x(x^2-z)^{-1/2}$, $x \in \bbR$. Then the following assertions hold:  
\begin{align}
& [g_z(A_+)-g_z(A_-)] \in\cB_1(L^2(\bbR)\otimes\bbC^m),\lb{incl_rhs}\\
& \big[(\bsH_2 - z \, \bsI)^{-1}-(\bsH_1 - z \, \bsI)^{-1}\big] 
\in \cB_1\big(L^2\big(\bbR^2) \otimes \bbC^m\big),\lb{incl_lhs}\\
& \tr_{L^2(\bbR^2) \otimes \bbC^m}
\big((\bsH_2 - z \, \bsI)^{-1}-(\bsH_1 - z \, \bsI)^{-1}\big)  \no \\
& \quad = \frac{1}{2z}\tr_{L^2(\bbR) \otimes \bbC^m} (g_z(A_+)-g_z(A_-)).  \lb{principle} 
\end{align}
\end{theorem}

\noindent 
The inclusion \eqref{incl_rhs} is an interesting result on its own; we feel that this rather strong result  \eqref{incl_rhs} is somewhat surprising. 

\subsection{The Generalized Pushnitski Formula.} We summarize the main result of 
Section \ref{s5}.
For an exposition of the theory of Krein's spectral shift function we refer to 
\cite{BY93}, \cite[Ch.~8]{Ya92}, \cite[Sect.~0.9]{Ya10}. For those aspects relevant to 
this article we refer to the review paper \cite{CGLS15}. 

Given the pairs $(\bsH_2, \bsH_1)$ and $(A_+, A_-)$, the corresponding Krein spectral shift functions (to be introduced in detail in Section \ref{s5}) are denoted by 
$\xi(\, \cdot \, ; \bsH_2, \bsH_1)$ and $\xi(\, \cdot \, ; A_+,A_-)$, respectively. These functions are only determined a.e. in general. They give formulas for both sides of \eqref{principle}.
The main application of the trace formula (\ref{principle}) that we make in this paper is to 
prove the following result:

\begin{theorem}\lb{main2}
Assume Hypothesis \ref{h3.1}. Then for a.e.~$\lambda>0$ 
and a.e.~$\nu \in \bbR$, 
\begin{equation}
\xi(\lambda; \bsH_2,\bsH_1)=\xi(\nu; A_+,A_-) 
= \frac{1}{2\pi}\int_\bbR\tr_{\bbC^m}(\Phi(x))\,dx. \lb{Push}
\end{equation}
\end{theorem}

Relation (\ref{Push}) will follow directly from the Pushnitski-type  formula (cf.\ 
\cite[Theorem~1.1]{Pu08}), 
\begin{equation} 
\xi(\lambda; \bsH_2, \bsH_1)=\frac{1}{\pi}\int_{-\lambda^{1/2}}^{\lambda^{1/2}}
\frac{\xi(\nu; A_+,A_-)\, d\nu}{(\lambda-\nu^2)^{1/2}} \, 
\text{  for a.e.~$\lambda>0$.} 
\end{equation} 
Moreover, employing some classical harmonic analysis, we are able to compute the actual  pointwise
value of the spectral shift function for the pair $A_+,A_-$. 

To complete this circle of ideas we still need to understand in more detail how the spectral flow enters the picture.
 Our view at this time is that the results of this paper support the idea that in the non-Fredholm setting the spectral shift function may provide information analogous to the spectral flow \cite{ACS07}. 

\begin{remark} 
We note that the inclusion, obtained in \eqref{incl_rhs}, has an interesting connection with the theory of Hankel operators. One of the fundamental results of Peller \cite{Pe80} describes the class of functions $\psi$ on $\bbR$ for which the commutator $[\sgn(D),\psi]$ is in the Schatten--von Neumann class 
$\cB_p\big(L^2(\bbR)\big)$. Here, $D$ in $L^2(\bbR)$ denotes the operator 
$D= -id/dx$, $\dom(D) = W^{1,2}(\bbR)$. Our result (discussed in Appendix \ref{sA}) shows that if one takes instead of the function $\sgn(\cdot)$ the ``smooth'' sign  function $g_{-1}(\cdot)$,  then the class of functions $\psi$ on $\bbR$, for which the commutator $[g_{-1}(D),\psi]$ is trace class, becomes much larger.   \hfill $\diamond$
\end{remark}

\begin{remark}  We have also studied this class of examples from the viewpoint of scattering 
theory in \cite{CGLPSZ14} (see also \cite{CGLS14})
to give an alternative approach to equation (\ref{Push}) and the explicit computation of the spectral shift function for the pair $(A_+,A_-)$.
However, the methods of these two papers are completely different. Most importantly \cite{CGLPSZ14} 
does not have the trace formula \eqref{principle}, it relies on an approximation result employing 
pseudo-differential operators, instead.     \hfill $\diamond$
\end{remark}

\subsection{Notation.} \lb{ss1.5}
We briefly summarize some of the notation used throughout this paper.
Let $\cH$ be a separable complex Hilbert space, $(\cdot,\cdot)_{\cH}$ the scalar product in $\cH$
(linear in the second argument), and $I_{\cH}$ the identity operator in $\cH$.
If $T$ is a linear operator mapping (a subspace of) a Hilbert space into another, then 
$\dom(T)$ and $\ker(T)$ denote the domain and kernel (i.e., null space) of $T$. 
The closure of a closable operator $S$ is denoted by $\ol S$. 
The convergence of bounded operators in the strong operator topology (i.e., pointwise limits) will be denoted by $\slim$, similarly, norm limits of bounded operators are denoted by $\nlim$. 

The Banach spaces of bounded and compact linear operators on a separable complex Hilbert space $\cH$ are denoted by $\cB(\cH)$ and $\cB_\infty(\cH)$, respectively; the corresponding $\ell^p$-based Schatten--von Neumann ideals are denoted by $\cB_p (\cH)$, with associated norm abbreviated by 
$\|\cdot\|_{\cB_p(\cH)}$, $p \geq 1$. Moreover, 
$\tr_{\cH}(A)$ denotes 
the trace of a trace class operator $A\in\cB_1(\cH)$. 

Linear operators in the Hilbert space $L^2(\bbR; dt; \cH)$, in short, $L^2(\bbR; \cH)$, will be denoted by calligraphic boldface symbols of the type $\bsT$, to distinguish them from 
operators $T$ in $\cH$. In particular, operators denoted by 
$\bsT$ in the Hilbert space $L^2(\bbR;\cH)$ typically represent operators associated with a 
family of operators $\{T(t)\}_{t\in\bbR}$ in $\cH$, defined by
\begin{align}
&(\bsT f)(t) = T(t) f(t) \, \text{ for a.e.\ $t\in\bbR$,}    \no \\
& f \in \dom(\bsT) = \big\{g \in L^2(\bbR;\cH) \,\big|\,
g(t)\in \dom(T(t)) \text{ for a.e.\ } t\in\bbR;    \lb{1.1}  \\
& \quad t \mapsto T(t)g(t) \text{ is (weakly) measurable;} \, 
\int_{\bbR} \|T(t) g(t)\|_{\cH}^2 \, dt <  \infty\bigg\}.   \no
\end{align}
In the special case, where $\{T(t)\}$ is a family of bounded operators on $\cH$ with 
$\sup_{t\in\bbR}\|T(t)\|_{\cB(\cH)}<\infty$, the associated operator $\bsT$ is a bounded operator on $L^2(\bbR;\cH)$ with $\|\bsT\|_{\cB(L^2(\bbR;\cH))} = \sup_{t\in\bbR}\|T(t)\|_{\cB(\cH)}$.

Capital letters $\Phi$, $\Psi$, etc. stand for $m\times m$ matrix-valued 
functions, while $f,\psi$ are typically real or complex-valued functions. 
By $L^p(\bbR)$ we denote the classical $L^p$-space of complex-valued measurable $p$-integrable functions on $\bbR$, 
employing Lebesgue measure if no measure is indicated, with associated norm 
denoted by $\|\cdot\|_p$, $p\geq 1$, and by $W^{k,p}(\bbR)$, $1\leq p<\infty$, $k\in\bbN$, the Sobolev space  consisting of all real-valued measurable functions $f$ on $\bbR$ such that
$
\|f\|_{k,p}:= \sum_{j=0}^k\|f^{(j)}\|_p<\infty. 
$
By $S(\bbR)$ we denote the test function space of all Schwartz functions on $\bbR$ and by $S'(\bbR)$ its dual consisting of tempered distributions (i.e., continuous, linear functionals on $S(\bbR)$). The bounded continuous functions 
on $\bbR$ are denoted by $C_b(\bbR)$. The symbol $AC_{\loc}(\bbR)$ represents locally absolutely continuous functions on $\bbR$. 

To simplify notation, we will frequently omit Lebesgue measure whenever possible 
and simply use $L^2(\bbR)$ instead of 
$L^2(\bbR; dx)$, and $L^2(\bbR^2)$ instead of $L^2(\bbR^2; dtdx)$, etc. 
If no confusion can arise, the identity operator in $L^2(\bbR)$ and $L^2(\bbR) \otimes \bbC^m$  
is simply denoted by $I$, the identity operator in $L^2(\bbR^2$ ) and 
$L^2(\bbR^2) \otimes \bbC^m$ by $\bsI$, and finally, $I_m$ represents the identity operator in $\bbC^m$.  

For a space $X$, $\mnn{X}$, denotes the space of all $m\times m$ matrices 
with entries in $X$. 
Moreover, the (maximally defined) operator of multiplication by the function 
$\phi$ in $L^2(\bbR)$, respectively, 
by the $m \times m$ matrix $\Phi \in M^{m \times m}(\bbR)$ in $L^2(\bbR) \otimes \bbC^m$, 
is simply denoted by $\phi$, 
respectively, $\Phi$  (instead of the more elaborate notation $M_{\phi}$, respectively, 
$M_\Phi$).   

Finally, we employ the abbreviations
\begin{equation}\lb{dfngz} 
g_z(x) = x(x^2-z)^{-1/2}, \; z\in\C\backslash [0,\infty), \quad 
g(x) = g_{-1}(x), \quad  x\in\bbR. 
\end{equation}

\section{Preliminaries} \lb{s2}

We start this section with the main hypothesis used in the statements of Theorems \ref{main1} and  \ref{main2}. We then provide a complete introduction to the operators $A_-,A_+, \bsH_1$, and $\bsH_2$ and describe their simplest properties (noting that we have deliberately kept our notation consistent with the usage in previous papers to ease comparisons). 

\begin{hypothesis} \lb{h3.1}
$(i)$ Introduce a self-adjoint matrix-valued function $\Phi$ satisfying
\begin{equation} 
\Phi \in \mnn{W^{1,1}(\bbR)\cap C_{b}(\bbR)},\,\, \Phi'\in\mnn{L^\infty(\bbR;dx)}. 
\end{equation} 
$(ii)$ Let $\theta\in L^{\infty}(\bbR; dt)$, $0 < \theta$,  satisfy 
$\theta' \in L^{\infty}(\bbR; dt) \cap L^1(\bbR; dt)$ and
\begin{equation}   
\lim_{t\rightarrow-\infty}\theta(t)=0, \quad  \lim_{t\rightarrow+\infty}\theta(t)=1.
\end{equation} 
\end{hypothesis}

\subsection{The Setup.}
The operator
\begin{equation}\lb{def_A_-}
A_-=D\otimes I_m, \quad \dom(A_-)=W^{1,2}(\bbR)\otimes \bbC^m, 
\end{equation}
is self-adjoint in $L^2(\bbR; dx)\otimes \bbC^m$, where 
\begin{equation} 
D=-i\frac{d}{dx}, \quad \dom(D) = W^{1,2} (\bbR).     \lb{defD}
\end{equation}  
The families of bounded self-adjoint operators $\{B(t)\}_{t \in \bbR}$ and 
self-adjoint operators $\{A(t)\}_{t \in \bbR}$ acting in 
$L^2(\bbR)\otimes \bbC^m$ are defined by
\begin{align}\lb{def_B(t)A(t)} 
\begin{split} 
& B(t)=\theta(t) \Phi , \quad \dom(B(t)) = L^2(\bbR; dx)\otimes \bbC^m,  \\ 
& A(t)=A_-+B(t), \quad
\dom(A(t)) = W^{1,2}(\bbR)\otimes \bbC^m, \; t \in \bbR.
\end{split} 
\end{align}
In particular, 
\begin{equation}\lb{def_Bprime}
B'(t)=\theta'(t) \Phi, \quad \dom(B(t)) = L^2(\bbR; dx)\otimes \bbC^m, \; t \in \bbR.
\end{equation}

From  Hypothesis \ref{h3.1} 
one concludes that there  exist limits
\begin{equation} 
\nlim_{t\rightarrow -\infty}B(t)=\nlim_{t\rightarrow -\infty}\theta(t) \Phi  =0, 
\mbox{
and 
}
\nlim_{t\rightarrow +\infty}B(t)=\nlim_{t\rightarrow +\infty}\theta(t) \Phi  = \Phi. 
\end{equation} 
Therefore, setting
$A_+=A_-+ \Phi ,$ with $\dom(A_+)= W^{1,2}(\bbR)\otimes \bbC^m, $
and using the standard resolvent identity one obtains   
\begin{align}
\begin{split} 
& \|(A(t)-z I)^{-1}-(A_--z I)^{-1}\|_{\cB(L^2(\bbR) \otimes \bbC^m)}     \\
& \quad =\|(A(t)-z I)^{-1}\theta(t) \Phi  (A_--z I)^{-1}\|_{\cB(L^2(\bbR) \otimes \bbC^m)}   
  \leq \ C \, |\theta(t)|\| \Phi  \|_\infty, 
  \end{split} 
\end{align}
and similarly 
\begin{align} 
& \|(A(t)-z I)^{-1}-(A_+-z I)^{-1}\|_{\cB(L^2(\bbR) \otimes \bbC^m)}    \\ 
& \quad =\|(A(t)-z I)^{-1}(1-\theta(t)) \Phi  (A_+-z I)^{-1}\|_{\cB(L^2(\bbR) \otimes \bbC^m)}     \leq \ C \, |1-\theta(t)|\| \Phi  \|_\infty.   \no 
\end{align} 
That is,  the limits
$
\nlim_{t\rightarrow \pm\infty}(A(t)-z I)^{-1}=(A_\pm-z I)^{-1} 
$ exist.

Subsequently, we will exploit the unitary equivalence of the operators $A_-$ and 
$A_+=A_-+ \Phi$. The following lemma establishes this fact (it corresponds to the 
well-known possibility of ``gauging away'' magnetic fields in one dimension).

\begin{lemma}\lb{unit_equiv} Assume Hypothesis \ref{h3.1}\,$(i)$ and let $x_0 \in \bbR$. 
Then there exists a 
unitary $m\times m$-matrix-valued function $\Psi(\,\cdot\,,x_0)$ on $\bbR$ such that 
\begin{equation} 
\Psi(\,\cdot\,,x_0)^* A_+ \Psi(\,\cdot\,,x_0) = A_-.   \lb{unitary}
\end{equation} 
\end{lemma}
\begin{proof}
We consider the first-order system of differential equations
\begin{equation}
\partial_x \Psi(x,x_0) = - i \Phi(x) \Psi(x,x_0), \quad \Psi(x_0, x_0) = I_m,      \lb{5.10} 
\end{equation} 
and take a fundamental $m \times m$ matrix of solutions, 
$\Psi(\, \cdot \, , \, \cdot \,)$ of \eqref{5.10}. Standard ODE theory (taking into account 
that $\Phi(x)$ is self-adjoint for all $x \in \bbR$), see, for instance, \cite[Sect.~IV.1]{Ha82}, 
then yields the properties, 
\begin{align}
& \Psi(x,x) = I_m, \quad x \in \bbR,    \lb{5.11} \\
& \Psi(x,x')^* = \Psi(x,x')^{-1} = \Psi(x',x), \quad x, x'  \in \bbR,     \lb{5.12} \\
& \Psi(x, x') \Psi(x', x'') = \Psi(x,x''), \quad x, x', x'' \in \bbR,      \lb{5.13} \\
& \ln({\det}_{\bbC^m} (\Psi(x,x'))) = - i \int_{x'}^x {\tr}_{\bbC^m} (\Phi(x''))\, dx'', \quad 
x, x' \in \bbR.    \lb{5.14} 
\end{align}

Properties \eqref{5.11}--\eqref{5.13} of $\Psi(\,\cdot\, , \, \cdot \,)$ can also be proved by a 
(norm-convergent) iteration (i.e., a Dyson expansion, cf.\ \cite[Theorem~X.69]{RS80}) of 
\begin{equation}
\Psi(x,x') = I_m - i \int_{x'}^x \ \Phi(x'') \Psi(x'', x')\, dx'', 
\quad x, x' \in \bbR.      \lb{5.14a} 
\end{equation}
In particular, because of the $L^1$-assumption made on $\Phi$ in Hypothesis \ref{h3.1}, 
iterating \eqref{5.14a} also permits one to take the limits of $\Psi(x,x')$ as $x$ and/or $x'$ 
tend to $\pm\infty$, that is, 
\begin{equation}
\Psi(\infty,x'), \; \Psi(x, - \infty), \; \Psi(\infty, -\infty), \, \text{etc., } \, x, x' \in \bbR,  
\end{equation} 
all exist.

Employing \eqref{5.10}, one verifies that $A_+$ is unitarily equivalent to $A_-$,
\begin{equation}
\Psi(\,\cdot\,,x_0)^{-1} A_+ \Psi(\,\cdot\,,x_0) = A_-    \lb{5.15}
\end{equation}
since 
\begin{align}
& \big([-i (d/dx) I_m + \Phi] \Psi(\, \cdot \, , x_0) f\big)(x)    \no \\
& \quad = -i [\Psi'(x,x_0)f(x) + \Psi(x,x_0) f'(x)] + \Phi(x) f(x)  \no \\
& \quad = -i \Psi(x,x_0) f'(x), \quad f \in C_0^{\infty}(\bbR)\otimes \bbC^m.
\end{align}
\end{proof}

\begin{remark}\lb{Psi} $(i)$ Given $N \in \bbN \cup\{\infty\}$, 
equality \eqref{5.10} yields inductively upon $N$ that 
$\Psi(\,\cdot\,,x_0) \in M^{m \times m}\big(C^{N}(\bbR)\big)$ whenever 
$\Phi\in M^{m \times m}\big(C^{N-1}(\bbR)\big)$. (For the first induction step,  
$N=1$, see, e.g., \cite[Lemma~IV.1.1]{Ha82}.) \\
$(ii)$ In the scalar case $m=1$, the function $\psi(x,x_0)=\exp{\big(-i\int_{x_0}^x \phi(y)\,dy\big)}$, 
$x \in \bbR$, yields the unitary equivalence in \eqref{unitary}.  \hfill $\diamond$
\end{remark}

\begin{corollary} 
Assume Hypothesis \ref{h3.1}\,$(i)$. 
Since the operators $A_-$ and $A_+$ are self-adjoint, Lemma \ref{unit_equiv} and the  functional calculus imply that
\begin{equation}\lb{fc_D,DMf}
h(A_+)= \Psi(\,\cdot\,,x_0) h(A_-) \Psi(\,\cdot\,,x_0)^*
\end{equation}
for any locally bounded Borel function $h\colon\bbR\to\bbC$.
\end{corollary}

Note, that for any $\Phi \in \mnn{L^{\infty}(\bbR) \cap AC_{\loc}(\bbR)}$ with $\Phi' \in \mnn{L^\infty(\bbR)}$, 
the equality
\begin{equation}\lb{comut_D}
[A_-, \Phi  ] = -i \Phi'
\end{equation}
holds.

\begin{remark}The operators $A_-$ and the family $B(t)$ do not satisfy 
\cite[Hypothesis~2.1]{GLMST11}. Indeed, Hypothesis 2.1 in \cite{GLMST11}, in particular, 
requires that the family  $B(t)$ is relative trace class, that is,  
$B'(t)(|A_-|+I)^{-1}\in\cB_1(L^2(\bbR)).$
In our setting, $B'(t)(|A_-|+I)^{-1}=\theta'(t) \Phi  (|A_-|+I)^{-1}$ and  by \cite[Theorem~4.1, Remark~4(a)]{Si05} a necessary condition for the operator $ \Phi  (|A_-|+I)^{-1}$ 
to be trace class is that the function $t\mapsto (|t|+1)^{-1}$ is integrable, which is 
obviously not the case. \hfill $\diamond$
\end{remark}

Next, we introduce the operator $(d/dt) \otimes I_m$ on 
$ L^2(\bbR^2) \otimes \bbC^m$  defined on its domain $ W^{1,2} \big(\bbR; L^2(\bbR)\big)\otimes \bbC^m$
in the obvious way and we make the identifications 
\begin{equation}
L^2\big(\bbR; dt; L^2(\bbR; dx)\big) \otimes \bbC^m = L^2(\bbR^2;dt dx) \otimes \bbC^m 
= L^2(\bbR^2) \otimes \bbC^m.  
\end{equation}
In addition, we define the  operator
\begin{equation}
\bsD_{\bsA_-}^{} = \f{d}{dt} \otimes I_m + \bsA_-,
\quad \dom(\bsD_{\bsA^{}_-})= W^{1,2}(\bbR^2)\otimes \bbC^m,  
\end{equation}
where $\bsA_-$ is the operator acting in $L^2(\bbR^2)\otimes \bbC^m$ 
given by 
\begin{align}\lb{def_bsA_-}
& (\bsA_- f)(t)=A_- f(t), \\
& \, f \in \dom(\bsA_-) = \big\{g \in L^2(\bbR^2) \otimes \bbC^m \,\big|\,
g(t)\in W^{1,2}(\bbR) \otimes \bbC^m \text{ for a.e.~$t\in\bbR$,}    \no \\
& \quad t \mapsto A_- g(t) \text{ is (weakly) measurable,} \,  
\int_{\bbR} \|A_- g(t)\|_{L^2(\bbR) \otimes \bbC^m}^2 \, dt < \infty\bigg\}. \no  
\end{align} 
By \cite[Lemma~4.2]{GLMST11}  the operator $\bsD_{\bsA_-}^{}$ is closed and densely 
defined with adjoint 
\begin{equation} 
\bsD_{\bsA_-}^*=-\frac{d}{dt} \otimes I_m + \bsA_-, \quad  
\dom(\bsD_{\bsA^*_-}) = \dom(\bsD_{\bsA^{}_-}) 
= W^{1,2}(\bbR^2) \otimes \bbC^m. 
\end{equation} 
Introduce the operator $\bsA$ in $L^2(\bbR^2) \otimes \bbC^m$ associated 
with the family $\{A(t)\}_{t \in \bbR}$ by  equation(\ref{def_bsA}),  
Introducing also the bounded operator $\bsB$ on $L^2(\bbR^2) \otimes \bbC^m$ 
by setting
\begin{equation}\lb{def_bsB}
(\bsB f)(t)=B(t) f(t) = \theta(t) (\Phi f)(t), \quad f\in L^2(\bbR^2)\otimes \bbC^m,
\end{equation}
one obtains  
$\bsA=\bsA_-+\bsB.$
Finally we come to the definition of  the model operator $\bsD_\bsA^{}$ acting in $L^2(\bbR^2) \otimes \bbC^m$.
\begin{align}\lb{def_D_A}
\begin{split} 
& \, \bsD_\bsA^{}=\frac{d}{dt} \otimes I_m + \bsA=\frac{d}{dt} \otimes I_m + \bsA_-+\bsB
=\bsD_{\bsA_-}^{}+\bsB,\\
& \dom(\bsD_\bsA^{})=\dom(\bsD_{\bsA_-}^{})=W^{1,2}(\bbR^2) \otimes \bbC^m. 
\end{split} 
\end{align}
Since $\bsD_{\bsA_-}^{}$ is a closed densely defined operator and $\bsB$ is a bounded operator on $L^2(\bbR^2) \otimes \bbC^m$, it follows that $\bsD_\bsA^{}$ is also closed and densely defined in  $L^2(\bbR^2) \otimes \bbC^m$ with adjoint given by
\begin{align}\lb{def_D_A^*}
& \, \bsD_\bsA^*=\bsD_{\bsA_-}^*+\bsB^*=-\frac{d}{dt} \otimes I_m + \bsA_-+\bsB
= - \frac{d}{dt} \otimes I_m + \bsA, \\
& \dom(\bsD_\bsA^*) =\dom(\bsD_\bsA^{}) = W^{1,2}(\bbR^2)\otimes \bbC^m.  \no
\end{align}
 
The second order operator $\bsH_0$ in $L^2(\bbR^2)\otimes \bbC^m$ is now constructed by 
\begin{align}\lb{def_H0}
\begin{split}
& \bsH_0=\bsD_{\bsA_-}^*\bsD^{}_{\bsA_-} 
= \big(-\frac{d^2}{dt^2} - \frac{d^2}{dx^2}\big) \otimes I_m 
= - \Delta \otimes I_m,\\ 
& \dom(\bsH_0)=W^{2,2}(\bbR^2) \otimes \bbC^m,  
\end{split} 
\end{align}
and the bounded operator $\bsB^\prime$ on $L^2(\bbR^2) \otimes \bbC^m$ associated with the family $\{B'(t)\}_{t \in \bbR}$ defined in \eqref{def_Bprime} is,  
\begin{equation}\lb{def_bsBprime}
(\bsB^\prime f)(t)=B'(t) f(t)=\theta'(t) (\Phi f)(t), \quad f \in L^2(\bbR^2) \otimes \bbC^m.
\end{equation}
An application of \cite[Theorem~VIII.33]{RS80} shows that the operator $\bsH_0$ is self-adjoint with 
$\dom(\bsH_0)\subset\dom(\bsA_-)$. 

Now we come to the operators $\bsH_1$ and $\bsH_2$ acting  in  $L^2(\bbR^2) \otimes \bbC^m$. Set
\begin{equation}\lb{def_H_j}
\bsH_1=\bsD_\bsA^*\bsD_\bsA^{}, \quad \bsH_2=\bsD_\bsA^{} \bsD_\bsA^*.
\end{equation}
By \eqref{def_D_A} and \eqref{def_D_A^*} one can write 
\begin{align}\lb{def_dom_H_j}
\begin{split}
& \, \bsH_j = \bsH_0 + \bsB \bsA_- + \bsA_- \bsB + \bsB^2 + (-1)^j \bsB^{\prime},   \\
& \dom(\bsH_j) = \dom(\bsH_0) = W^{1,2}(\bbR^2)\otimes \bbC^m, \quad j =1,2. 
\end{split}
\end{align} 
Hence one concludes that $\dom(\bsH_1)=\dom(\bsH_2)=\dom(\bsH_0)$ and that the 
operators $\bsH_j$, $j=1,2$, are well-defined  
since $\bsB$ leaves the domain of $\bsA_-$ invariant. In addition, one can write 
\begin{equation} 
\bsB\bsA_-+\bsA_-\bsB = [\bsA_-,\bsB]+2\bsB\bsA_-
= - [\bsA_-,\bsB]+2\bsA_-\bsB, 
\end{equation} 
and 
\begin{align}
([\bsA_-,\bsB] f)(t)&=(\bsA_-\bsB f)(t)-(\bsB\bsA_- f)(t) 
= A_-\theta(t) \Phi f(t) - \theta(t) \Phi  A_- f(t)   \no \\
&=\theta(t)[A_-, \Phi  ] f(t)= i^{-1} \theta(t) \Phi' f(t), 
\quad f \in W^{2,2}(\bbR^2)\otimes \bbC^m.  
\end{align}
Employing the fact that $\Phi'\in\mnn{L^\infty(\bbR)}$, $\theta \in L^\infty(\bbR)$, one obtains 
that the commutator $[\bsA_-,\bsB]$ has a bounded closure.
For subsequent purposes we denote 
\begin{equation}\lb{commutBA}
\bsC:= \ol{[\bsA_-,\bsB]}, 
\end{equation}
and write
\begin{align}\lb{HwithC}
\begin{split} 
\bsH_j &= \bsH_0 + 2\bsB\bsA_- +\bsC+ \bsB^2 + (-1)^j \bsB^{\prime}   \\
&= \bsH_0 + 2\bsA_- \bsB-\bsC+ \bsB^2 + (-1)^j \bsB^{\prime}, \quad j =1,2.
\end{split} 
\end{align} 

\begin{lemma}\lb{boundedness} 
Assume Hypothesis \ref{h3.1} and let $z\in\bbC\backslash [0,\infty)$. Then the operators 
$\bsA_-$, $\bsH_0$, and $\bsH_j$, $j=1,2$, defined above satisfy the following properties: \\
$(i)$ $\bsA_-(\bsH_0-z\bsI)^{-1/2} \in \cB\big(L^2(\bbR^2) \otimes \bbC^m\big)$. \\
$(ii)$ $(\bsH_0-z\bsI)(\bsH_j-z\bsI)^{-1} \in \cB\big(L^2(\bbR^2) \otimes \bbC^m\big)$, 
$j=1,2$. 
\end{lemma}
\begin{proof}
$(i)$ Via the two-dimensional Fourier transform the operator  $\bsA_-(\bsH_0-z)^{-1/2}$
is isometric  to the multiplication operator on $L^2(\bbR^2)\otimes \bbC^m$ given by the function
$(s,p)\mapsto p\bsI(s^2\bsI+p^2\bsI-z\bsI)^{-1/2}$ and this is clearly bounded.\\
(ii) By \eqref{def_dom_H_j} the operators $\bsH_j,$\, $j=1,2$ and $\bsH_0$ have the same domain. Therefore, the operator $(\bsH_0-z\bsI)({\bsH}_j-z\bsI)^{-1}$, $j=1,2,$ is everywhere defined and closed as the product of a bounded and a closed operator. Hence, by the closed graph theorem, the operator  $(\bsH_0-z\bsI)({\bsH}_j-z\bsI)^{-1}$, $j=1,2$, is bounded. 
\end{proof}

\subsection{The Approximation Technique.}
The principal method exploited in \cite{GLMST11} to pass from the Pushnitski 
assumptions to the weaker ones 
involving relatively trace class perturbations was an approximation scheme. Here we illustrate 
a  modified version of this scheme that enables us (albeit with considerable effort) to handle relatively 
Hilbert--Schmidt perturbations. 

The idea is that we  approximate the operators $A_-$ and $A_+$, and hence, $\bsH_1,\bsH_2$
in such a fashion that, for the approximants, the relatively trace class perturbation property
of \cite{GLMST11}  is restored. To this end let $P_n = E_{A_-}((-n,n))$ be the spectral projection of $A_-$ corresponding to the interval $(-n,n)$. We set
\begin{equation}\lb{def_B_n}
B_n(t)=P_nB(t)P_n,\quad A_n(t)=A_-+B_n(t), \quad t \in \bbR.
\end{equation}
Therefore, 
\begin{equation}\lb{A_+,n}
A_{-,n}=A_-,\quad A_{+,n} = A_- + P_n \Phi  P_n, \quad n \in \bbN.
\end{equation}

\begin{remark}\lb{not_rtc}
By the estimate \eqref{trclass} below, the operator $A_-$ and the family of bounded operators $B_n(t)=P_nB(t)P_n$, $t \in \bbR$, $n \in \bbN$, satisfy \cite[Hypothesis~2.1]{GLMST11} (and even 
\cite[Hypothesis~(1.3)]{Pu08}). \hfill $\diamond$
\end{remark}

For the projection $\bsP_n$ defined by 
\begin{equation} 
(\bsP_n f)(t) = P_n f(t), \quad f \in L^2(\bbR^2)\otimes \bbC^m, \; t \in \bbR, \; n \in \bbN, 
\end{equation} 
one infers that $\bsP_n = E_{\bsA_-}((-n,n))$ is the spectral projection of $\bsA_-$ corresponding to the interval $(-n,n)$; moreover, $\bsP_n$ commutes with $\bsH_0$, $n \in \bbN$.
In addition, for the operators $\bsH_{j,n}$, $j=1,2$, defined in terms of the family 
$\{A_n(t)\}_{t \in \bbR}$, one  obtains decompositions similar to \eqref{HwithC}, 
\begin{align} \lb{H_nwithC}
\begin{split} 
\bsH_{j,n} &= \bsH_0 + 2\bsB_n \bsA_- +\bsC_n + \bsB_n^2 + (-1)^j \bsB_n^{\prime}
\\
&= \bsH_0 + 2\bsA_-\bsB_n - \bsC_n + \bsB_n^2 + (-1)^j \bsB_n^{\prime},  
\quad n \in \bbN, \; j =1,2, 
\end{split} 
\end{align} 
with
\begin{equation}\lb{def_bsB_n}
\bsB_n =\bsP_n \bsB \bsP_n, \quad 
\bsB_n^{\prime} = \bsP_n\bsB^{\prime}\bsP_n, \quad \bsC_n=\bsP_n\bsC \bsP_n, \quad n \in \bbN. 
\end{equation} 

\begin{remark}\lb{common_core}
Since $\dom(\bsH_0)\subset\dom(\bsA_-)$, it follows from \eqref{HwithC} and \eqref{H_nwithC} that the operators $\bsH_j$ and $\bsH_{j,n}$, $j=1,2$, have the common core $\dom(\bsH_j)=\dom(\bsH_0) = W^{2,2}(\bbR) \otimes \bbC^m$, $j=1,2$. \hfill $\diamond$
\end{remark}


\section{The Right-Hand Side of the Trace Formula \eqref{principle}} \lb{s3} 

In the first part of this section we prove the inclusion 
\begin{equation} 
[g_z(A_+)-g_z(A_-)] \in \cB_1\big(L^2(\bbR)\big), \quad z<0. 
\end{equation} 
Later, in Theorem \ref{thm_PTF}, we extend this result to the first inclusion  in Theorem \ref{main1}, that is, 
\begin{equation} 
[g_z(A_+)-g_z(A_-)] \in \cB_1\big(L^2(\bbR)\big), \quad z\in\bbC \backslash [0,\infty). 
\end{equation} 
This result is the main advance that we make in the proof of the trace formula over 
the approach in \cite{CGK}. 
In the second part of this section we prove that the difference $[g_z(A_+)-g_z(A_-)]$, $z<0$, 
can be approximated in $\cB_1\big(L^2(\bbR)\otimes \bbC^m\big)$-norm by the operators 
$[g_z(A_{+,n})-g_z(A_-)]$ as $ n \to \infty$.

\subsection{The Trace Result.}
For brevity, we introduce the notations $R_{+,\lambda}(z)$, $R_{-,\lambda}(z)$ 
for appropriate resolvents of the operators $A_-$ and $A_+$, respectively, that is,  
\begin{equation}\lb{def_Alambda}
 R_{+,\lambda}(z)= \big(A_++i(\lambda-z)^{1/2}I\big)^{-1}, \ 
 R_{-,\lambda}(z)= \big(A_-+i(\lambda-z)^{1/2} I\big)^{-1}, \ \lambda>0. 
\end{equation}

\begin{lemma}\lb{general integral lemma} 
Let $z<0$, then,  
\begin{equation} 
g_z(A_+)-g_z(A_-)={\pi}^{-1}\Re\bigg(\int_0^{\infty} \lambda^{-1/2} 
[R_{+,\lambda}(z) - R_{-,\lambda}(z)]\,d\lambda\bigg).
\end{equation} 
\end{lemma}
\begin{proof} We recall the fact that for any self-adjoint operator $T$ in $\cH$, 
\begin{equation} 
\big(T^2-z I_{\cH}\big)^{-1/2} = {\pi}^{-1}\int_0^{\infty} \lambda^{-1/2} 
\big(T^2 - (z - \lambda) I_{\cH}\big)^{-1}\,d\lambda,    \quad z <0,
\end{equation}  
with a norm convergent Bochner integral (see, e.g., \cite[p.~282]{Ka80} for a 
more general result). Thus,
\begin{equation} 
g_z(A_+)-g_z(A_-)=\frac{1}{\pi}\int_0^{\infty}{\lambda^{1/2}}
\Big[A_+ \big(A_+^2 -(z - \lambda) I_{\cH}\big)^{-1} 
- A_- \big(A_-^2 - (z - \lambda) I_{\cH}\big)^{-1}\Big] \, d\lambda.
\end{equation} 
Taking into account the equality 
\begin{equation} 
A_+ \big(A_+^2 - (z - \lambda) I_{\cH}\big)^{-1} - A_- \big(A_-^2 - (z - \lambda) I_{\cH}\big)^{-1} 
=\Re(R_{+,\lambda}(z) - R_{-,\lambda}(z)),   
\end{equation}  
one concludes the proof.
\end{proof}

\begin{remark}
One observes that Lemma \ref{general integral lemma} holds for arbitrary self-adjoint operators. 
${}$  \hfill$\diamond$
\end{remark}

In what follows, we denote for brevity 
$U_{\lambda}(z)=(A_+-A_-)R_{-,\lambda}(z) = \Phi  R_{-,\lambda}(z),$
frequently suppressing the $z$-dependence of $U_{\lambda}$ and $A_{\pm,\lambda}$ in the following. 
The next result yields the first claim in Theorem \ref{main1}. In our present setting we do not resort to the double operator integration technique as in \cite{GLMST11}, but instead apply more elementary means.

\begin{proposition}\lb{g_complex}
Suppose that $\Phi\in\mnn{W^{1,1}(\bbR)\cap C_b(\bbR)}$ and 
$z <0$.  
For the operators $A_-=D\otimes I_m$, $A_+=A_-+ \Phi$ one obtains 
\begin{align} 
& g_z(A_+)-g_z(A_-) = \Phi  {(A_-^2 - z I)^{-3/2}}      \lb{gir} \\
& \quad 
+\pi^{-1}\Re\bigg(\int_0^\infty \big[\lambda^{-1/2}R_{+,\lambda}(z)U_\lambda(z)^2
- \lambda^{-1/2}[R_{-,\lambda}(z), \Phi  ]R_{-,\lambda}(z)\big]\,d\lambda\bigg),    \no 
\end{align} 
and each term on the right-hand side of \eqref{gir} lies in $\cB_1\big(L^2(\bbR)\big)$. 
In addition, 
\begin{equation} 
\|g_z(A_+)-g_z(A_-)\|_{\cB_1(L^2(\bbR) \otimes \bbC^m)} \leq\|\Phi\|_{1,1}.
\end{equation}  
\end{proposition}
\begin{proof}
Using the resolvent identity twice one can write 
\begin{align}
R_{+,\lambda}-R_{-,\lambda}&=R_{+,\lambda} \Phi  R_{-,\lambda}
=-R_{-,\lambda} \Phi  +R_{+,\lambda} \Phi  R_{-,\lambda} \Phi  R_{-,\lambda}=-R_{-,\lambda}U_\lambda+R_{+,\lambda}U_\lambda^2\no\\
&=- \Phi  R_{-,\lambda}^2-[R_{-,\lambda}, \Phi  ]R_{-,\lambda}+R_{+,\lambda}U_\lambda^2.\lb{resol_dif}
\end{align}
Next, we separately treat the three preceding terms.

First, we show  that 
\begin{equation} 
\int_0^\infty {\lambda^{-1/2}}R_{+,\lambda}U_\lambda^2\,d\lambda 
\in\cB_1\big(L^2(\bbR)\otimes \bbC^m\big).    \lb{XX}
\end{equation} 
Employing the noncommutative H\"older inequality (see, e.g., \cite[Ch.~2]{Si05}), 
\begin{align}\lb{A1lU^2}
& \bigg\|\int_0^\infty{\lambda^{-1/2}}R_{+,\lambda}U_\lambda^2 \, d\lambda 
\bigg\|_{\cB_1(L^2(\bbR) \otimes \bbC^m)}    
\leq \int_0^\infty{\lambda^{-1/2}} 
\big\|R_{+,\lambda}U_\lambda^2\big\|_{\cB_1(L^2(\bbR) \otimes \bbC^m)} \, d\lambda   \no \\
& \quad \leq \int_0^\infty {\lambda^{-1/2}} \|R_{+,\lambda}\|_{\cB(L^2(\bbR) \otimes \bbC^m)}
\|U_\lambda\|_{\cB_2(L^2(\bbR) \otimes \bbC^m)}^2 \, d\lambda.
\end{align}

Thus, applying \cite[Theorem~4.1]{Si05} and Remark \ref{viamax}, 
\begin{align}
\|U_\lambda\|_{\cB_2(L^2(\bbR) \otimes \bbC^m)} 
&=\big\| \Phi  \big(A_-+i(\lambda-z)^{1/2} I\big)^{-1}\big\|_{\cB_2(L^2(\bbR) \otimes \bbC^m)}   \no \\
&\leq\ C \, \max_{j,k=1,\dots, m}\big\|\Phi_{j,k} 
\big(D+i(1+\lambda)^{1/2} I \big)^{-1}\big\|_{\cB_2(L^2(\bbR))}  \no \\ 
& \leq \ C \, \max_{j,k}\|\Phi_{j,k}\|_2\|h_1\|_2    \no \\
&=\ C \, \|\Phi\|_2\|h_1\|_2\leq\ C \, 
\|\Phi\|_1\|\Phi\|_\infty\|h_1\|_2,
\end{align}
where 
$h_1(t)=({t+i(\lambda-z)^{1/2}})^{-1}$ and $C>0$ represents a constant that may well differ from line to line. Since $\|h_1\|_2=C (\lambda-z)^{-1/4}$, one infers that 
\begin{equation}\lb{norm_U_lambda}
\|U_\lambda\|_{\cB_2(L^2(\bbR) \otimes \bbC^m)}
\leq\ C \, \|\Phi\|_1\|\Phi\|_\infty (\lambda-z)^{-1/4}.
\end{equation}
In addition, (see \eqref{def_Alambda}) one has  
\begin{equation} 
\|R_{+,\lambda}\|_{\cB(L^2(\bbR) \otimes \bbC^m)} 
\leq\sup_{t\in\bbR}({|t+i(\lambda-z)^{1/2}|})^{-1}=(\lambda-z)^{-1/2}, 
\end{equation}  
hence, combining this estimate with \eqref{A1lU^2} and \eqref{norm_U_lambda}, one obtains 
\begin{equation}\lb{norm_firstint}
\bigg\|\int_0^\infty{\lambda^{-1/2}}R_{+,\lambda}U_\lambda^2 \, 
d\lambda\bigg\|_{\cB_1(L^2(\bbR) \otimes \bbC^m)} 
\leq\ C \, \|\Phi\|_1\|\Phi\|_\infty\int_0^\infty{\lambda^{-1/2}(\lambda-z)^{-1}}\,d\lambda.
\end{equation}
Since the integral on the right-hand side converges, the claim \eqref{XX} follows. 

Next, we show that 
\begin{equation} 
\int_0^\infty{\lambda^{-1/2}}[R_{-,\lambda}, \Phi  ]R_{-,\lambda}
\, d\lambda \in\cB_1\big(L^2(\bbR)\otimes \bbC^m\big).    \lb{YY} 
\end{equation} 
Using the formula $[C^{-1},B]=-C^{-1}[C,B]C^{-1}$ and equality \eqref{comut_D} one gets  
\begin{equation} 
[R_{-,\lambda}, \Phi  ]=-R_{-,\lambda}[A_-, \Phi  ]R_{-,\lambda}=iR_{-,\lambda} \Phi' R_{-,\lambda}. 
\end{equation} 
Hence, one infers 
\begin{align}\lb{norm_sec}
& \bigg\|\int_0^\infty{\lambda^{-1/2}}[R_{-,\lambda}, \Phi  ]R_{-,\lambda}\,
d\lambda\bigg\|_{\cB_1(L^2(\bbR) \otimes \bbC^m)}    \no \\ 
& 
\quad 
\leq\int_0^\infty{\lambda^{-1/2}} \big\|R_{-,\lambda} 
\Phi' R_{-,\lambda}^2\big\|_{\cB_1(L^2(\bbR) \otimes \bbC^m)}\, d\lambda    \no \\
& \quad \leq \int_0^\infty{\lambda^{-1/2}} 
\big\|R_{-,\lambda} |\Phi'|^{1/2} \big\|_{\cB_2(L^2(\bbR) \otimes \bbC^m)} \big\| |\Phi'|^{1/2} 
R_{-,\lambda}^2\big\|_{\cB_2(L^2(\bbR) \otimes \bbC^m)} \, d\lambda.     
\end{align}
Since by hypothesis, $\Phi'\in\mnn{L^1(\bbR)}$, Corollary \ref{square_root} implies that the matrix 
$|\Phi'|^{1/2}=\Big\{\widetilde{\Phi'}_{j,k}\Big\}_{j,k=1}^m$ belongs to $\mnn{L^2(\bbR)}$. Thus, using once more \cite[Theorem~4.1]{Si05}, one concludes that  
\begin{align}
 & \big\|R_{-,\lambda}|\Phi'|^{1/2}\big\|_{\cB_2(L^2(\bbR) \otimes \bbC^m)}   
 \leq\ C \, \max_{1 \leq j,k \leq m}
\Big\|\big(D+i(1+\lambda)^{1/2} I\big)^{-1}\widetilde{\Phi'}_{j,k} 
\Big\|_{\cB_2(L^2(\bbR))}   \no \\
& \quad 
 \leq\ C \, \max_{1 \leq j,k \leq m} \Big\|\widetilde{\Phi'}_{j,k}\Big\|_2\|h_1\|_2
\leq\ C \, \big\||\Phi'|^{1/2}\big\|_2\|h_1\|_2.
\end{align}
Arguing similarly, one obtains that  
\begin{equation} 
\big\||\Phi'|^{1/2}R_{-,\lambda}^2\big\|_{\cB_2(L^2(\bbR) \otimes \bbC^m)} 
\leq\ C \, \big\||\Phi'|^{1/2}\big\|_2\|h_2\|_2,
\end{equation}  
where
$ h_2(t)={(t+i(\lambda-z)^{1/2})^{-2}}.$
It is easy to check that 
$\|h_2\|_2=C(\lambda-z)^{-3/4}$. Appealing to the estimate 
$\big\||\Phi'|^{1/2}\big\|_2^2\leq C \, \|\Phi'\|_1$ (see Corollary \ref{square_root}), 
 \eqref{YY} is proved by estimating the RHS of \eqref{norm_sec} 
 as follows
\begin{align}\lb{norm_commut} 
\bigg\|\int_0^\infty{\lambda^{-1/2}}[R_{-,\lambda}, 
\Phi  ]R_{-,\lambda}\,d\lambda\bigg\|_{\cB_1(L^2(\bbR) \otimes \bbC^m)}   \no
& \leq\ C \, \|\Phi'\|_1\int_0^\infty{\lambda^{-1/2}(\lambda-z)^{-1}}\,d\lambda \\ 
&
<\infty. 
\end{align}

Finally, we prove that 
\begin{equation} 
\Re\bigg(\int_0^\infty{\lambda^{-1/2}} \Phi  R_{-,\lambda}^2\,d\lambda\big) \in 
\cB_1\big(L^2(\bbR) \otimes \bbC^m\bigg).   \lb{ZZ} 
\end{equation}  
First we note, 
\begin{align} 
\begin{split}  
& \Re\bigg(\int_0^\infty{\lambda^{-1/2}} \Phi  R_{-,\lambda}^2\,d\lambda\bigg)
= \Phi  \int_0^\infty{\lambda^{-1/2}}\Re\big(R_{-,\lambda}^2\big)\,d\lambda      \\
& \quad = \Phi  \int_0^\infty{\lambda^{-1/2}} \big(A_-^2+(z-\lambda) I\big)
\big(A_-^2- (z-\lambda) I\big)^{-2} \, d\lambda, 
\end{split} 
\end{align} 
and also 
\begin{equation} 
\int_0^\infty \lambda^{-1/2} \big(A_-^2+(z-\lambda) I\big) \big(A_-^2 -(z -\lambda) I\big)^{-2}
\, d\lambda =-{\pi}^{-1} \big(A_-^2-z I\big)^{-3/2}, 
\end{equation}  
so one obtains 
$
\Re\big(\int_0^\infty{\lambda^{-1/2}} 
\Phi  A_{0,\lambda}^2\,d\lambda\big) = - \pi{ \Phi  } \big(A_-^2-z I\big)^{-3/2}.
$
Furthermore, 
\cite[Theorem~4.5]{Si05} and Lemma \ref{Cwikel_par} imply 
\begin{align}\lb{first_inL1}
& \big\| \Phi \big(A_-^2-z I\big)^{-3/2}\big\|_{\cB_1(L^2(\bbR) \otimes \bbC^m)} 
\leq\ C \, \max_{j,k} \big\|\Phi_{j,k}(D^2-z I)^{-3/2}\big\|_{\cB_1(L^2(\bbR))}   \no\\
& \quad \leq\ C \, \max_{j,k=1,\dots,m}\|\Phi_{j,k}\|_{\ell^1(L^2(\bbR))}
\big\|((\cdot)^2-z)^{-3/2}\big\|_{\ell^1(L^2(\bbR))}     \no \\
& \quad \leq \ C \, \|\Phi'\|_{1,1},
\end{align}
and hence \eqref{ZZ}.

Thus, combining equality \eqref{resol_dif} with Lemma \ref{general integral lemma} and the estimates obtained in \eqref{norm_firstint}, \eqref{norm_commut} and \eqref{first_inL1} imply 
\eqref{gir}. 
In addition, the same estimates yield 
\begin{equation} 
\|g_z(A_+)-g_z(A_-)\|_{\cB_1(L^2(\bbR) \otimes \bbC^m)} 
\leq\ C \, [\|\Phi\|_1\|\Phi\|_\infty+\|\Phi\|_{1,1}].
\end{equation} 
Next, for fixed $n\in\bbN$,  
\begin{align}\lb{estimate_normg}
& \|g_z(A_+)-g_z(A_-)\|_{\cB_1(L^2(\bbR) \otimes \bbC^m)}     \no \\ 
& \quad =\big\|\sum_{k=0}^{n-1}\big(g_z\big(A_-+n^{-1}({k+1}){} \Phi \big)
-g_z\big(A_-+n^{-1}({k+1}) \Phi\big) \big)\big\|_{\cB_1(L^2(\bbR) \otimes \bbC^m)}  \no \\
& \quad \leq \sum_{k=1}^{n-1} \big\|g_z\big(A_-+n^{-1}({k+1}) \Phi \big)
-g_z\big(A_-+n^{-1}k \Phi  \big)\big\|_{\cB_1(L^2(\bbR) \otimes \bbC^m)}.
\end{align} 
Applying Lemma \ref{unit_equiv} one obtains for fixed $k \in \bbN$ the existence of a  
sequence of unitary matrices   
$\Psi_{(k,n)}$ such that $A_-+\frac{k}{n} \Phi  = \Psi_{(k,n)} A_- \Psi_{(k,n)}^*$. 
(We use the notation $\Psi_{(k,n)}$ to avoid any confusion with the matrix elements 
$\Psi_{k,\ell}$, $k,\ell = 1,\dots,m$, of $\Psi$.) Hence we have 
\begin{align} 
\begin{split} 
A_-+n^{-1}({k+1}) \Phi  
&= \Psi_{(k,n)} A_- 
\Psi_{(k,n)}^* + n^{-1} \Phi  
\\&
= \Psi_{(k,n)} \big(A_-+n^{-1}\Psi_{(k,n)}^* \Phi  \Psi_{(k,n)}\big) \Psi_{(k,n)}^*. 
\end{split} 
\end{align}
And thus, 
\begin{align}
\begin{split}  
&\|g_z(A_-+n^{-1}({k+1})\Phi ) 
- g_z(A_-+n^{-1}k \Phi )\|_{\cB_1(L^2(\bbR) \otimes \bbC^m)}  \\
& \quad = \|g_z(A_-)-g_z(A_-+n^{-1}
\Psi_{(k,n)}^* \Phi \Psi_{(k,n)})\|_{\cB_1(L^2(\bbR) \otimes \bbC^m)}.
\end{split} 
\end{align} 
Combining this with \eqref{estimate_normg} and using that every $\Psi_{(k,n)}$ 
is a unitary matrix yields 
\begin{align}
& \|g_z(A_+)-g_z(A_-)\|_{\cB_1(L^2(\bbR) \otimes \bbC^m)}   \no \\
&  \quad \leq \sum_{k=0}^{n-1}
 \big\|g_z(A_-)-g_z
\big(A_-+n^{-1} 
\Psi_{(k,n)}^* \Phi \Psi_{(k,n)}\big)\big\|_{\cB_1(L^2(\bbR) \otimes \bbC^m)}  \no \\
& \quad \leq\ C \, \sum_{k=0}^{n-1}
\|n^{-1} \Psi_{(k,n)}^* \Phi 
\Psi_{(k,n)}
\|_\infty
\|n^{-1} \Psi_{(k,n)}^* \Phi \Psi_{(k,n)}
\|_1   \no \\ 
& \qquad  + 
\|n^{-1} \Psi_{(k,n)}^* \Phi \Psi_{(k,n)}    
\|_{1,1}    \no \\
& \quad \leq n \ C [\|n^{-1}\Phi \|_\infty \| n^{-1} \Phi\|_1 
+ \|n^{-1} \Phi \|_{1,1}].
\end{align}
Hence,
\begin{align}
& \|g_z(A_+)-g_z(A_-)\|_{\cB_1(L^2(\bbR) \otimes \bbC^m)}   \no \\
& \quad \leq \ C \, \lim_{n\rightarrow \infty} 
n[\|n^{-1}\Phi\|_\infty 
\|n^{-1}\Phi\|_1+ \|n^{-1} \Phi\|_{1,1}]   \no \\
& \quad \leq C\lim_{n\rightarrow \infty} n[{n^{-2}}\|\Phi\|_\infty \|\Phi\|_1 
+n^{-1} \|\Phi\|_{1,1}]= C \|\Phi\|_{1,1}.
\end{align}
\end{proof}

\subsection{The Approximation Argument.}
In this subsection we explain the key idea of our approach.
We turn to the operators 
$A_{+,n}=A_-+P_n \Phi  P_n$, $P_n = E_{A_-}((-n,n))$, $n \in \bbN$, but first, we 
represent the difference $g_z(A_{+,n})-g_z(A_-)$ in close analogy to our expression for 
$g_z(A_{+})-g_z(A_{-})$ obtained in Proposition \ref{g_complex}.

\begin{proposition}\lb{g_n_complex}
Suppose that $\Phi\in\mnn{W^{1,1}(\bbR)\cap C_b(\bbR)}$ and  
$z <0$. Let 
 $A_{+,n}=A_-+P_n \Phi  P_n$, $n \in \bbN$, be as in \eqref{A_+,n}, and 
 \begin{equation}  
 R_{+,\lambda}^{(n)}=\big(A_{+,n}+i(\lambda-z)^{1/2} I\big)^{-1}, \quad 
 U_\lambda^{(n)}=P_nU_\lambda P_n,  \quad n \in \bbN. 
 \end{equation} 
 Then we have that $g_z(A_{+,n}) - g_z(A_-)$ is the following sum of trace class operators:
\begin{align}
\begin{split} 
& g_z(A_{+,n}) - g_z(A_-)=P_n \Phi  {(A_-^2-z)^{-3/2}}P_n    \lb{3.10} \\
& \quad +\pi^{-1}\Re\bigg(\int_0^\infty(R_{+,\lambda}^{(n)}(U^{(n)}_\lambda)^2-P_n[R_{-,\lambda}, \Phi  ]R_{-,\lambda}P_n)\lambda^{-1/2}\,d\lambda\bigg). 
\end{split} 
\end{align}
\end{proposition}
\begin{proof}
One computes, 
\begin{align}
& R_{+,\lambda}^{(n)}-R_{-,\lambda} =-R_{+,\lambda}^{(n)}(P_n \Phi  P_n)R_{-,\lambda}  \no \\
& \quad =-R_{-,\lambda}P_n \Phi P_nR_{-,\lambda}+R_{+,\lambda}^{(n)}P_n \Phi  P_nR_{-,\lambda}P_n \Phi  P_nR_{-,\lambda}  \no \\
& \quad =-P_nR_{-,\lambda} \Phi  R_{-,\lambda}P_n+R_{+,\lambda}^{(n)}P_n \Phi  R_{-,\lambda}P_n \Phi  R_{-,\lambda}P_n  \no \\
& \quad =-P_n[R_{-,\lambda}, \Phi  ]R_{-,\lambda}P_n-P_n \Phi  R_{-,\lambda}^2P_n+R_{+,\lambda}^{(n)}P_n \Phi  R_{-,\lambda}P_n \Phi  R_{-,\lambda}P_n.
\end{align}
Arguing as in the proof of Proposition \ref{g_complex} yields the claimed assertions.
\end{proof}

The following theorem is the main result of this section; it yields a trace norm approximation 
of the operator $[g_z(A_+)-g_z(A_-)]$, $z \in \bbC \backslash [0,\infty)$.

\begin{theorem}\lb{conv_rhs} 
Suppose that $\Phi\in\mnn{W^{1,1}(\bbR)\cap C_b(\bbR)}$ and 
$z <0$. Then, 
$
\lim_{n\rightarrow\infty} \big\|[g_z(A_{+,n})-g_z(A_{-})] 
- [g_z(A_+)-g_z(A_-)]\big\|_{\cB_1(L^2(\bbR)\otimes \bbC^m)}= 0. 
$
\end{theorem}
\begin{proof}
For simplicity, we assume $z<0$ throughout this proof. 
By Proposition \ref{g_complex} and Proposition \ref{g_n_complex} it suffices to show that we have trace norm convergence:
\begin{align}
& P_n \Phi \big(A_-^2-z I\big)^{-3/2}P_n \underset{n\to\infty}{\longrightarrow} 
\Phi \big(A_-^2-z I\big)^{-3/2},    \no \\
& P_n\int_0^\infty{\lambda^{-1/2}}[R_{-,\lambda}, \Phi  ]R_{-,\lambda}\,d\lambda P_n
 \underset{n\to\infty}{\longrightarrow}  \int_0^\infty{\lambda^{-1/2}}[R_{-,\lambda}, \Phi  ]R_{-,\lambda} \, d\lambda,   \lb{3terms} \\
& \int_0^\infty{\lambda^{-1/2}}R_{+,\lambda}^{(n)}(U^{(n)}_\lambda)^2\,d\lambda 
 \underset{n\to\infty}{\longrightarrow}  \int_0^\infty{\lambda^{-1/2}}R_{+,\lambda}
 U_\lambda^2\,d\lambda. \no 
\end{align}
By \eqref{norm_commut} and \eqref{first_inL1} the operators $ \Phi  {(A_-^2-z)^{-3/2}}$ and 
$ \int_0^\infty{\lambda^{-1/2}}[R_{-,\lambda}, \Phi  ]R_{-,\lambda} \, d\lambda $ are trace class. As $P_n \underset{n\to\infty}{\longrightarrow} I$ in the strong operator topology, 
Lemma \ref{so_L_p}  implies convergence of the first two terms in \eqref{3terms}. For the third term one obtains 
\begin{align}\lb{conv_third}
\begin{split} 
& \bigg\|\int_0^\infty{\lambda^{-1/2}}R_{+,\lambda}^{(n)}(U^{(n)}_\lambda)^2\,d\lambda 
-\int_0^\infty{\lambda^{-1/2}}R_{+,\lambda} 
U_\lambda^2\,d\lambda\bigg\|_{\cB_1(L^2(\bbR) \otimes \bbC^m)}   \\
& \quad \leq \int_0^\infty{\lambda^{-1/2}}
\big\|R_{+,\lambda}^{(n)}(U^{(n)}_\lambda)^2 
- R_{+,\lambda}U_\lambda^2\big\|_{\cB_1(L^2(\bbR) \otimes \bbC^m)} \, d\lambda, 
\end{split}
\end{align}

We start by showing that $A_{+,n}\underset{n\to\infty} \longrightarrow A_+$ in the 
strong resolvent sense. Since $A_-$ is closed and densely defined, the operators 
$A_{+,n}$ and $A_+$ have the common core $\dom(A_-)$. For all $f \in \dom(A_-)$, 
$n \in \bbN$
\begin{align}
& \|A_{+,n}f - A_+ f\|_{L^2(\bbR) \otimes \bbC^m} 
= \|P_nA_- f +P_n \Phi  P_n f - A_- f - \Phi f\|_{L^2(\bbR) \otimes \bbC^m}   \no \\
& \quad \leq \|P_nA_- f - A_- f\|_{L^2(\bbR) \otimes \bbC^m}\! 
+ \|P_n \Phi  (1-P_n) f\|_{L^2(\bbR) \otimes \bbC^m} \!    \no \\
& \qquad+ \|(1-P_n) \Phi f\|_{L^2(\bbR) \otimes \bbC^m}\!.
\end{align}
The first and the last term converges to zero, since 
$\slim_{n\to\infty} P_n = I$, while the second term converges 
to zero since 
\begin{equation} 
\|P_n \Phi  (1-P_n) f\|_{L^2(\bbR) \otimes \bbC^m} 
\leq\| \Phi  \|_\infty\|(1-P_n) f\|_{L^2(\bbR) \otimes \bbC^m} 
\underset{n\to\infty}{\longrightarrow} 0. 
\end{equation} 
Thus, \cite[Theorem~VIII.25]{RS80} (see also \cite[Theorem~9.16]{We80}) implies that 
$A_{+,n}$ converges to $A_+$ in the strong resolvent sense. 

Next, we claim that 
$\lim_{n\rightarrow\infty} \big\|(U^{(n)}_\lambda)^2-U_\lambda^2
\big\|_{\cB_1(L^2(\bbR) \otimes \bbC^m)} = 0.$ We first note that 
\begin{align} 
& \big\|U_\lambda^2-P_nU_\lambda P_nU_\lambda P_n\big\|_{\cB_1(L^2(\bbR) \otimes \bbC^m)} 
\no \\
& \quad
 \leq \big\|(1-P_n)U_\lambda^2\big\|_{\cB_1(L^2(\bbR) \otimes \bbC^m)} 
+ \|U_\lambda\|_{\cB_2(L^2(\bbR) \otimes \bbC^m)}   \no \\
& \qquad \times \|U_\lambda(1-P_n) + (1-P_n)U_\lambda P_n\|_{\cB_2(L^2(\bbR) \otimes \bbC^m)}. 
\end{align} 
Since $U_\lambda\in \cB_2\big(L^2(\bbR)\big)$ (see \eqref{norm_U_lambda}), 
Lemma \ref{so_L_p} implies that 
\begin{equation} 
\big\|U_\lambda^2-P_nU_\lambda P_nU_\lambda P_n\big\|_{\cB_1(L^2(\bbR) \otimes \bbC^m)} \underset{n\to\infty}{\longrightarrow} 0.
\end{equation} 
Combining strong resolvent convergence of $A_{+,n}$ to $A_+$ as $n\to\infty$, and 
trace norm convergence of $(U^{(n)}_\lambda)^2$ to $U_\lambda^2$ as $n\to\infty$, 
Lemma \ref{so_L_p} implies 
that the integrands on the right-hand side of \eqref{conv_third} converge to zero.
In addition, one infers that  
\begin{align}
& \big\|R_{+,\lambda}^{(n)}(U^{(n)}_\lambda)^2
-R_{+,\lambda}U_\lambda^2\big\|_{\cB_1(L^2(\bbR) \otimes \bbC^m)} \leq 
\big\|R_{+,\lambda}^{(n)}\big\|_{\cB(L^2(\bbR) \otimes \bbC^m)}  
\|U_\lambda\|_{\cB_2(L^2(\bbR) \otimes \bbC^m)}^2   \no \\
& \quad + \|R_{+,\lambda}\|_{\cB(L^2(\bbR) \otimes \bbC^m)} 
\|U_\lambda\|_{\cB_2(L^2(\bbR) \otimes \bbC^m)}^2. 
\end{align} 
For the norms of resolvents one estimates (we recall that $z < 0$), 
\begin{equation} 
\big\|R_{+,\lambda}^{(n)}\big\|_{\cB(L^2(\bbR) \otimes \bbC^m)}, \quad 
\|R_{+,\lambda}\|_{\cB(L^2(\bbR) \otimes \bbC^m)} \leq (\lambda-z)^{-1/2},
\end{equation}  
and 
\begin{equation} 
\|U_\lambda\|_{\cB_2(L^2(\bbR) \otimes \bbC^m)}^2\leq (\lambda-z)^{-1/2}, 
\end{equation}  
and thus one obtains that 
$
\lambda^{-1/2}\big\|R_{+,\lambda}^{(n)}(U^{(n)}_\lambda)^2 
- R_{+,\lambda}U_\lambda^2\big\|_{\cB_1(L^2(\bbR) \otimes \bbC^m)}
$ 
is dominated by the integrable function $\lambda^{-1/2}(\lambda-z)^{-1}$. Thus, the dominated convergence theorem implies that 
\begin{equation} 
\int_0^\infty{\lambda^{-1/2}} \big\|R_{+,\lambda}^{(n)}(U^{(n)}_\lambda)^2
-R_{+,\lambda}U_\lambda^2\big\|_{\cB_1(L^2(\bbR) \otimes \bbC^m)}\, 
d\lambda \underset{n\to\infty} \longrightarrow 0.
\end{equation} 
\end{proof}

\section{The Left-Hand Side of the Trace Formula \eqref{principle}} \lb{s4} 

Our main objective in this section is to provide the second inclusion of Theorem \ref{main1} (see \eqref{incl_lhs}) and prove the trace norm convergence on the left-hand side of the trace formula \eqref{principle}.
First, we state a lemma which collects some properties of the operators 
$\bsH_{j,n}$, $n \in \bbN$, and $\bsH_j$, $j=1,2$.

\begin{lemma} \lb{l3.5}
Assume Hypothesis \ref{h3.1} and let $z\in\bbC\backslash [0,\infty)$.
Then the following assertions hold: \\
$(i)$ The operators $\bsH_{j,n}$ converge to $\bsH_j$, $j=1,2$, 
in the strong resolvent sense, 
\begin{equation}
\slim_{n \to \infty} (\bsH_{j,n}-z\, \bsI)^{-1} 
= ( \bsH_{j}-z\, \bsI)^{-1}, \quad j=1,2.     \lb{limR}
\end{equation}
$(ii)$ The operators 
$\ol{( \bsH_{1,n}-z \, \bsI)^{-1}(\bsH_{0}-z \, \bsI)}$ and
$(\bsH_{0}- z \, \bsI) ( \bsH_{1,n}- z \, \bsI)^{-1}$, $n \in \bbN$, 
are uniformly bounded with respect to $n \in \bbN$. In addition, 
\begin{align} 
& \slim_{n \to \infty} (\bsH_{0}-z \, \bsI) (\bsH_{j,n}-z\, \bsI)^{-1} 
= (\bsH_{0}-z \, \bsI)( \bsH_{j}-z\, \bsI)^{-1}, \quad j=1,2,    \lb{conv2}\\
& \slim_{n \to \infty} \ol{(\bsH_{j,n}-z \, \bsI)^{-1}(\bsH_{0} - z \, \bsI)} 
= \ol{( \bsH_{j}-z \, \bsI)^{-1}(\bsH_{0}-z \, \bsI)}, \quad j=1,2.   \lb{conv1} 
\end{align}  
\end{lemma}
\begin{proof} Since the proof for the operators $\bsH_{2,n},\bsH_2$ is a verbatim repetition of the proof for $\bsH_{1,n},\bsH_1$, we exclusively focus on the latter. 

\noindent 
$(i)$  By Remark \ref{common_core} the  operators $\bsH_1$ and $\bsH_{1,n}$ have the 
common core 
$\dom(\bsH_1)=\dom(\bsH_0)$. Since $\bsH_{1,n}$ and $\bsH_1$ are self-adjoint operators with a common core, by \cite[Theorem~VIII.25]{RS80} (see also 
\cite[Theorem~9.16]{We80}) it is sufficient to show that 
\begin{equation} 
\bsH_{1,n}f \underset{n\to\infty}{\longrightarrow} \bsH_1 f, \quad f \in\dom(\bsH_0). 
\end{equation} 
Equalities \eqref{HwithC} and \eqref{H_nwithC} imply convergence of every term separately. 
First, rewriting 
\begin{equation} 
\bsB^\prime-\bsB^\prime_n=\bsB^\prime-\bsP_n \bsB^\prime \bsP_n=(\bsI-\bsP_n)\bsB^\prime+\bsP_n\bsB^\prime(\bsI-\bsP_n), 
\end{equation} 
the convergence 
\begin{equation}
\slim_{n \to \infty} \bsB^\prime_n = \bsB^\prime    \lb{limBprime}
\end{equation}
follows 
since the operator $\bsB^\prime$, defined by \eqref{def_bsBprime}, is a bounded operator,  
and $\bsP_n \underset{n\to\infty}{\longrightarrow} \bsI$ in the strong operator topology. 
Arguing similarly, one infers that 
\begin{equation} 
\slim_{n \to \infty} \bsB_n = \bsB, \quad 
\slim_{n \to \infty} \bsC_n = \bsC.     \lb{limBC}
\end{equation} 
Next, one notes that 
\begin{align} 
\begin{split} 
\bsB^2-\bsB^2_n &=\bsB^2-\bsP_n\bsB\bsP_n\bsB\bsP_n      \\
&=(\bsI-\bsP_n)\bsB^2+\bsP_n\bsB\big(\bsB(\bsI-\bsP_n)+(\bsI-\bsP_n)\bsB\bsP_n\big),   
\end{split} 
\end{align}
implying, 
\begin{equation} 
\slim_{n \to \infty} \bsB^2_n = \bsB^2.    \lb{limB-2}
\end{equation} 
Thus, appealing to \eqref{HwithC} and \eqref{H_nwithC}, it remains to 
show that $\slim_{n \to \infty} \bsB_n\bsA_- f = \bsB\bsA_- f$ for all 
$f \in \dom(\bsH_0)$. The fact,  
\begin{align}
\begin{split} 
\bsB\bsA_--\bsB_n\bsA_-&=\bsB\bsA_--\bsP_n\bsB\bsA_-\bsP_n   \\
&=(\bsI-\bsP_n)\bsB\bsA_-+\bsP_n\bsB(\bsI-\bsP_n)\bsA_-,  
\end{split} 
\end{align} 
implies the required convergence. Consequently, 
\begin{equation} 
\slim_{n \to \infty} \bsH_{1,n} f = \bsH_1 f, \quad f \in \dom(\bsH_0), 
\end{equation} 
completes the proof of item $(i)$.

\smallskip 
\noindent 
$(ii)$ Fix $z \in \bbC \backslash [0,\infty)$. First we prove the uniform boundedness with 
respect to $n\in\bbN$ of the operators $\ol{( \bsH_{1,n}-z \, \bsI)^{-1}(\bsH_{0}-z \, \bsI)}$ and
$(\bsH_{0}- z \, \bsI) ( \bsH_{1,n}- z \, \bsI)^{-1}$, $n \in \bbN$.
Since the operator $( \bsH_{1,n}-z \, \bsI)^{-1}(\bsH_{0}-z \, \bsI)$ is closable, one concludes  
(see, e.g., \cite[Theorem~VIII.1]{RS80}) that 
\begin{equation}
\ol{( \bsH_{1,n}-z \, \bsI)^{-1}(\bsH_{0}-z \, \bsI)} = 
\big[(\bsH_{0}- {\ol z} \, \bsI) ( \bsH_{1,n}- {\ol z} \, \bsI)^{-1}\big]^*. 
\end{equation}
Thus, it suffices to show  that $(\bsH_{0}- z \, \bsI) ( \bsH_{1,n}- z \, \bsI)^{-1}$ is uniformly 
bounded with respect to $n \in \bbN$. 

Using the standard resolvent identity one obtains
\begin{equation}\lb{ssss2}
({\bsH}_{1,n}-z \, \bsI)^{-1}-(\bsH_0-z \, \bsI)^{-1}= 
-({\bsH}_{1,n}-z \, \bsI)^{-1}(\bsH_{1,n}-\bsH_0)(\bsH_0-z \, \bsI)^{-1}, 
\end{equation}
and hence employing \eqref{H_nwithC} one arrives at
\begin{align}
\begin{split} 
& (\bsH_0 - z \, \bsI)(\bsH_{1,n} - z \, \bsI)^{-1} 
= \bsI - \big[(\bsH_{1,n} - \bsH_0)(\bsH_{1,n} - z \, \bsI)^{-1}\big]   \\
& \quad = \bsI - \big[2 \bsB_n \bsA_- + \bsC_n + \bsB_n^2 - \bsB_n^{\prime}\big]
(\bsH_{1,n} - z \, \bsI)^{-1}, \quad n \in \bbN.   \lb{seq} 
\end{split}
\end{align}
The sequence of bounded operators $(\bsH_{1,n} - z \, \bsI)^{-1}$, $n \in \bbN$, is uniformly bounded, in addition, since the operators $\bsB,\bsB^\prime,\bsC$ are bounded, it follows from \eqref{def_bsB_n} that the sequences $\bsB_n$, $\bsB^\prime_n$, and $\bsC_n$, $n \in \bbN$, are also uniformly bounded with respect to $n\in\bbN$. Thus, by \eqref{seq} it is sufficient to prove the uniform boundedness of the sequence  $\bsA_- (\bsH_{1,n} - z \, \bsI)^{-1}$, $n \in \bbN$,  
which we focus on next. 

Again, appealing to the standard resolvent identity one obtains for each $ n \in \bbN$, 
\begin{align}\lb{seq1} 
 & \bsA_- (\bsH_{1,n} - z \, \bsI)^{-1}    \no \\
 & \quad = \bsA_- (\bsH_0 - z \, \bsI)^{-1}   
- \bsA_- (\bsH_0 - z \, \bsI)^{-1}\big[(\bsH_{1,n} - \bsH_0)
(\bsH_{1,n} - z \, \bsI)^{-1}\big]   \no \\
&  \quad= \bsA_- (\bsH_0 - z \, \bsI)^{-1} 
\no   \\ & \qquad 
- \bsA_- (\bsH_0 - z \, \bsI)^{-1}
\big[2 \bsA_- \bsB_n - \bsC_n + \bsB_n^2 - \bsB_n^{\prime}\big]
(\bsH_{1,n} - z \, \bsI)^{-1}.    
\end{align}
Arguing as above, it is sufficient to show that the operators $\bsA_- (\bsH_0 - z \, \bsI)^{-1}$ and $\ol{\bsA_- (\bsH_0 - z \, \bsI)^{-1}\bsA_-}$ are bounded. The first operator is bounded  by Lemma~\ref{boundedness}\,$(i)$ while the second operator is bounded since 
\begin{equation}
\ol{\bsA_- (\bsH_0 + \bsI)^{-1} \bsA_-} = 
\big[\bsA_- (\bsH_0 + \bsI)^{-1/2}\big] \big[\bsA_- (\bsH_0 + \bsI)^{-1/2}\big]^*,  
\end{equation}
and by Lemma~\ref{boundedness}\,$(i)$ the operator $\bsA_- (\bsH_0 +\, \bsI)^{-1/2}$ is bounded.
Thus, the sequence operators $(\bsH_0 - z \, \bsI)(\bsH_{1,n} - z \, \bsI)^{-1}$,  is uniformly bounded in $n\in\bbN$.

Next, gathering all terms from \eqref{seq} and \eqref{seq1} one arrives at 
\begin{align}
& (\bsH_0 - z \, \bsI)(\bsH_{1,n} - z \, \bsI)^{-1} 
= \bsI - \big[2 \bsB_n \bsA_- + \bsC_n + \bsB_n^2 - \bsB_n^{\prime}\big]
(\bsH_{1,n} - z \, \bsI)^{-1}    \no \\ 
& \quad = \bsI - \big[\bsC_n + \bsB_n^2 - \bsB_n^{\prime}\big]
(\bsH_{1,n} - z \, \bsI)^{-1} - 2 \bsB_n \bsA_- (\bsH_0 - z \, \bsI)^{-1}    \no \\
& \qquad + 2 \bsB_n \bsA_- (\bsH_0 - z \, \bsI)^{-1}  
\big[2 \bsA_- \bsB_n - \bsC_n + \bsB_n^2 - \bsB_n^{\prime}\big] 
(\bsH_{1,n} - z \, \bsI)^{-1}   \no \\
& \quad = \bsI - \big[\bsC_n + \bsB_n^2 - \bsB_n^{\prime}\big]
(\bsH_{1,n} - z \, \bsI)^{-1} - 2 \bsB_n \big[\bsA_- (\bsH_0 - z \, \bsI)^{-1}\big]   \no \\
& \qquad + 2 \bsB_n \big[\bsA_- (\bsH_0 - z \, \bsI)^{-1}\big]   
\big[- \bsC_n + \bsB_n^2 - \bsB_n^{\prime}\big] 
(\bsH_{1,n} - z \, \bsI)^{-1}    \no \\
& \qquad + 4 \bsB_n \big[\ol{\bsA_- (\bsH_0 - z \, \bsI)^{-1} \bsA_-}\big] \bsB_n 
(\bsH_{1,n} - z \, \bsI)^{-1}, \quad n \in \bbN.    \lb{seq3}
\end{align}
Similarly, one rewrites the right-hand side of \eqref{conv2} as 
\begin{align}
& (\bsH_0 - z \, \bsI)(\bsH_{1} - z \, \bsI)^{-1} 
 = \bsI - \big[\bsC + \bsB^2 - \bsB^{\prime}\big](\bsH_{1} - z \, \bsI)^{-1}  \no \\
& \quad - 2 \bsB \big[\bsA_- (\bsH_0 - z \, \bsI)^{-1}\big]   \no \\
& \quad + 2 \bsB \big[\bsA_- (\bsH_0 - z \, \bsI)^{-1}\big]   
\big[- \bsC + \bsB^2 - \bsB^{\prime}\big] 
(\bsH_{1} - z \, \bsI)^{-1}    \no \\
& \quad + 4 \bsB \big[\ol{\bsA_- (\bsH_0 - z \, \bsI)^{-1} \bsA_-}\big] \bsB 
(\bsH_{1} - z \, \bsI)^{-1}, \quad n \in \bbN.    \lb{seq4}
\end{align}
Thus, the strong resolvent convergence in \eqref{conv2}, \eqref{limBprime}, 
\eqref{limBC}, and \eqref{limB-2} implies strong convergence in \eqref{conv2}.   
Finally, to prove \eqref{conv1}, we first note that by 
 the strong resolvent convergence in \eqref{limR},   
\begin{equation} 
(\bsH_{1,n}-z \, \bsI)^{-1}(\bsH_{0} - z \, \bsI) f  \underset{n\to\infty}{\longrightarrow} ( \bsH_1 - z \, \bsI)^{-1}(\bsH_{0}-z \, \bsI)f, \quad f \in \dom(\bsH_0).
\end{equation} 
Since, in addition, the operators $\ol{(\bsH_{1,n}-z \, \bsI)^{-1}(\bsH_{0} - z \, \bsI)}$ are uniformly bounded with respect to $n\in\bbN$, and $\dom(\bsH_0)$ is dense in 
$L^2(\bbR^2)\otimes\bbC^m$, one infers the strong operator convergence  \eqref{conv1}.
\end{proof}

Next, we prove the second inclusion of our main result, Theorem \ref{main1}. 

\begin{proposition}\lb{fritz2} 
Assume Hypothesis \ref{h3.1} and let 
$z\in\bbC\backslash [0,\infty)$. Then, 
\begin{equation} 
\big[({\bsH}_1-z\bsI)^{-1}-({\bsH}_2-z\bsI)^{-1}\big] 
\in\cB_1\big(L^2(\bbR^2)\otimes \bbC^m\big). 
\end{equation}  
In addition,  
$
\big[({\bsH}_{1,n}-z\bsI)^{-1}-({\bsH}_{2,n}-z\bsI)^{-1}\big] 
\in\cB_1\big(L^2(\bbR^2)\otimes \bbC^m\big), \quad n \in \bbN. 
$
\end{proposition}
\begin{proof} The standard resolvent identity and \eqref{HwithC} imply 
\begin{align}
& ({\bsH}_1-z\bsI)^{-1}-({\bsH}_2-z\bsI)^{-1}=2({\bsH}_1-z\bsI)^{-1} \theta' \Phi  ({\bsH}_2-z\bsI)^{-1} 
\no \\
& \quad =2\overline{({\bsH}_1-z\bsI)^{-1}(\bsH_0-z\bsI)}(\bsH_0-z\bsI)^{-1} \theta' \Phi  (\bsH_0-z\bsI)^{-1}   \no \\ 
& \qquad \times (\bsH_0-z\bsI)({\bsH}_2-z\bsI)^{-1}.
\end{align}

By Lemma \ref{l3.5}\,$(ii)$, 
\begin{align}
& \ol{(\bsH_1 - z \, \bsI)^{-1} (\bsH_0 - z  \bsI)} 
= \big[(\bsH_0 - {\ol z} \, \bsI) (\bsH_1 - {\ol z} \, \bsI)^{-1}\big]^* \in 
\cB\big(L^2(\bbR^2)\otimes\bbC^m\big),   \no \\ 
& (\bsH_0 - z \, \bsI) (\bsH_2 - z \, \bsI)^{-1} \in 
\cB\big(L^2(\bbR^2)\otimes\bbC^m\big), 
\end{align} 
and hence it suffices to show that
 \begin{equation}\lb{middleInL1}
 (\bsH_0-z\bsI)^{-1} \theta' \Phi  (\bsH_0-z\bsI)^{-1}\in\cB_1(L^2(\bbR^2) \otimes \bbC^m).
 \end{equation}
Since $\Phi\in \mnn{L^1(\bbR)}$, Corollary \ref{square_root} implies that $|\Phi|^{1/2}\in \mnn{L^2(\bbR)}$. In  addition, $\theta'\in L^1(\bbR)$, and hence, 
$|\theta'|^{1/2}|\Phi|^{1/2}\in \mnn{L^2(\bbR^2)}$. Thus, 
\cite[Theorem~4.1]{Si05} implies  
\begin{equation} 
(\bsH_0-z\bsI)^{-1} |\theta'|^{1/2} |\Phi|^{1/2}, \; |\theta'|^{1/2} |\Phi|^{1/2}
(\bsH_0-z\bsI)^{-1}\in\cB_2(L^2(\bbR^2) \otimes \bbC^m), 
\end{equation} 
and hence 
$(\bsH_0-z\bsI)^{-1} \theta' \Phi  (\bsH_0-z\bsI)^{-1}\in\cB_1(L^2(\bbR^2) \otimes \bbC^m).$
The inclusion $\big[({\bsH}_{1,n}-z\, \bsI)^{-1}-({\bsH}_{2,n}-z\, \bsI)^{-1}\big] 
\in\cB_1(L^2(\bbR^2) \otimes \bbC^m)$, $n \in \bbN$, is proved similarly.
\end{proof}

The following result shows that 
$\big[{({\bsH}_1-z\bsI)^{-1}-({\bsH}_2-z\bsI)^{-1}}\big]$ can be approximated in trace 
norm by 
$\big[({\bsH}_{1,n}-z\bsI)^{-1}-({\bsH}_{2,n}-z\bsI)^{-1}\big]$ 
as $n \to \infty$. 

\begin{proposition} \lb{t3.7}
Assume Hypothesis \ref{h3.1} and let $z \in \bbC \backslash \bbR$. Then 
\begin{align}
\begin{split} 
& \lim_{n\to\infty} \big\|\big[(\bsH_{2,n} - z \, \bsI)^{-1} - (\bsH_{1,n} - z \, \bsI)^{-1}\big]  \\
& \hspace*{1cm} - [(\bsH_2 - z \, \bsI)^{-1} - (\bsH_1 - z \, \bsI)^{-1}\big]
\big\|_{\cB_1(L^2(\bbR^2)\otimes\bbC^m)} = 0.    \lb{2.66}
\end{split} 
\end{align}
\end{proposition}
\begin{proof}
An application of \eqref{HwithC}, \eqref{H_nwithC} and the resolvent equation for the difference of 
resolvents in \eqref{2.66} yields  
\begin{align}
& \big[(\bsH_{2,n} - z \, \bsI)^{-1} - (\bsH_{1,n} - z \, \bsI)^{-1}\big]
 - [(\bsH_2 - z \, \bsI)^{-1} - (\bsH_1 - z \, \bsI)^{-1}\big]    \no \\
 & \quad = - 2 (\bsH_{2,n} - z \, \bsI)^{-1} \bsB_n^{\prime} (\bsH_{1,n} - z \, \bsI)^{-1} 
 + 2 (\bsH_2 - z \, \bsI)^{-1} \bsB^{\prime} (\bsH_1 - z \, \bsI)^{-1}   \no \\
 & \quad = - 2\ol{\big[(\bsH_{2,n} - z \, \bsI)^{-1} (\bsH_0 - z \, \bsI)\big]}   
\big\{\bsP_n (\bsH_0 - z \, \bsI)^{-1} 
\bsB^{\prime} (\bsH_0 - z \, \bsI)^{-1} 
 \bsP_n\big\}   
 \no \\
& \qquad 
\quad \times \big[(\bsH_0 - z \, \bsI) (\bsH_{1,n} - z \, \bsI)^{-1}\big]      \\
& \qquad + 2\ol{\big[(\bsH_2 - z \, \bsI)^{-1} (\bsH_0 - z \, \bsI)\big]}   
 \big\{(\bsH_0 - z \, \bsI)^{-1} \bsB^{\prime} 
(\bsH_0 - z \, \bsI)^{-1}\big\}    \no \\
& \qquad \quad \times \big[(\bsH_0 - z \, \bsI) (\bsH_1 - z \, \bsI)^{-1}\big], \quad 
z \in \bbC \backslash [0,\infty).  \no
\end{align}
Since $\slim_{n\to\infty} \bsP_n = \bsI$, 
inclusion \eqref{middleInL1} and Lemma \ref{so_L_p} imply that the sequence  
$\bsP_n (\bsH_0 - z \, \bsI)^{-1} \bsB^{\prime} (\bsH_0 - z \, \bsI)^{-1} \bsP_n$ 
converges to $(\bsH_0 - z \, \bsI)^{-1} \bsB^{\prime} 
(\bsH_0 - z \, \bsI)^{-1}$ in $\cB_1\big(L^2(\bbR^2)\otimes\bbC^m\big)$-norm as 
$n\to \infty$. Another application of Lemma \ref{so_L_p} proves \eqref{2.66} since by 
Lemma \ref{l3.5}\,$(ii)$,  
\begin{align}
\slim_{n\to\infty} \big[(\bsH_0 - z \, \bsI) (\bsH_{1,n} - z \, \bsI)^{-1}\big] = 
\big[(\bsH_0 - z \, \bsI) (\bsH_1 - z \, \bsI)^{-1}\big],&     \\
\slim_{n\to\infty} \ol{\big[(\bsH_{2,n} - z \, \bsI)^{-1} (\bsH_0 - z \, \bsI)\big]} =  
\ol{\big[(\bsH_2 - z \, \bsI)^{-1} (\bsH_0 - z \, \bsI)\big]},&    \\
z \in \bbC \backslash [0,\infty),& 
\end{align} 
completing the proof. 
\end{proof}

\section{Proof of the  Trace Formula \eqref{principle} and its Implications for Relations 
Between Spectral Shift Functions} \lb{s5} 

It follows from the results in Sections \ref{s3} and \ref{s4} that both inclusion \eqref{incl_rhs} 
and  \eqref{incl_lhs} of Theorem \ref{main1} hold. In addition, one can approximate the left and right-hand sides of the principle trace formula \eqref{principle} in their respective trace norms. In this section, we finally prove the equality \eqref{principle} for the operators $A_-$, $A_+$ and 
$\bsH_j$, $j=1,2$, thereby extending the result of \cite[Theorem~2.2]{GLMST11} to the  family of operators $\{A(t)\}_{t \in \bbR}$ given by \eqref{def_B(t)A(t)}, which do not satisfy 
\cite[Hypothesis~2.1]{GLMST11}

We start by stating the principal trace formula for the operators $A_{\pm,n}$ and 
$\bsH_{j,n}$, $j=1,2$, $n\in\bbN$, defined by \eqref{A_+,n} and \eqref{H_nwithC}.

\begin{proposition}\lb{propAppr}
Assume Hypothesis \ref{h3.1} and let $z \in \bbC \backslash [0,\infty)$. 
For the operators $A_n(t)$, $A_{\pm,n}$ on $\cH$ and the operators $\bsH_{1,n}$, 
 $\bsH_{2,n}$ on $L^2(\bbR^2) \otimes \bbC^m$, obtained by replacing $A(t)$ by $A_n(t)$ in 
\eqref{def_H_j}, $n \in \bbN$, one obtains, 
\begin{align} \lb{trfn}
\begin{split} 
& \tr_{L^2(\bbR^2) \otimes \bbC^m}\big((\bsH_{2,n} - z \, \bsI)^{-1} 
- (\bsH_{1,n} - z \, \bsI)^{-1}\big)    \\
& \quad =\frac{1}{2z}\tr_{L^2(\bbR) \otimes \bbC^m}\big(g_z(A_{+,n})-g_z(A_-)\big), 
\quad n \in \bbN.      
\end{split} 
\end{align}
\end{proposition}
\begin{proof}
By \cite[Proposition~1.3]{Pu08} it suffices to show that 
\beq\lb{trclass}
\int_\bbR\|B'_n(t)\|_{\cB_1(L^2(\bbR) \otimes \bbC^m)} \, dt<\infty.
\enq
It follows from \eqref{def_B_n} that in our case it is sufficient to prove that
\beq
\int_\bbR\|P_n\theta'(t) \Phi  P_n\|_{\cB_1(L^2(\bbR) \otimes \bbC^m)}\, dt < \infty.
\enq
Appealing to \cite[Theorem~4.5]{Si05} yields  
\begin{align}
\|P_n\theta'(t) \Phi  P_n\|_{\cB_1(L^2(\bbR) \otimes \bbC^m)} 
& \leq |\theta'(t)|\|P_n |\Phi|^{1/2} \|_{\cB_2(L^2(\bbR) \otimes \bbC^m)}
\||\Phi|^{1/2}P_n\|_{\cB_2(L^2(\bbR) \otimes \bbC^m)}   \no \\
& \leq \ C \, |\theta'(t)|, 
\end{align}
completing the proof since $\theta' \in L^1(\bbR)$ by Hypothesis \ref{h3.1}.
\end{proof}

By Proposition \ref{g_complex}, 
$[g(A_+)-g(A_-)] \in \cB_1\big(L^2(\bbR)\otimes \bbC^m\big)$, and so 
by \cite[Section~8.11]{Ya10} we define the spectral shift function 
for the pair $A_-,A_+$ by setting
\begin{equation}\lb{def_xiA}
\xi(\nu; A_+,A_-):=\xi(g(\nu); g(A_+),g(A_-)),
\end{equation}
where $\xi(\,\cdot\,;g(A_+),g(A_-))$ is the spectral shift function for the pair $g(A_+),g(A_-)$ uniquely defined by the requirement $\xi(\,\cdot\,;g(A_+),g(A_-))\in L^1(\bbR)$ 
(cf., \cite[Sections~9.1, 9.2]{Ya10}). It follows from definition \eqref{def_xiA} that the 
function $\xi(\,\cdot\,;A_+,A_-)$ is locally integrable on $\bbR$.
The Krein--Lifshitz trace formula in its simplest form implies that 
$
{\tr}_{L^2(\bbR) \otimes \bbC^m}(g(A_+)-g(A_-))=\int_{[-1,1]}\xi(s;g(A_+),g(A_-))\,ds.
$ 
 Hence, changing variables in the integral and using \eqref{def_xiA} results in 
\begin{equation}\lb{tr for g}
{\tr}_{L^2(\bbR) \otimes \bbC^m}(g(A_+)-g(A_-))=\int_{\bbR}\xi(\nu;A_+,A_-)g'(\nu)\,d\nu.
\end{equation}

For convenience of the reader we now recall the  Besov space 
$B_{\infty 1}^1(\bbR)$. There are several equivalent definitions of this space 
(see, e.g.,  \cite{Pe76}). Using the difference operator $\Delta_t$ defined by 
$
(\Delta_th)(s)=h(s+t)-h(t), \quad h \in C(\bbR),
$ a convenient definition for our purposes reads as follows, 
\begin{equation}
B_{\infty, 1}^1(\bbR) = \bigg\{h' \in C_{ub}(\bbR) \, \bigg| \, 
\sup_{t\in\bbR}|h'(t)|+\int_\bbR\frac{\sup_{s \in \bbR}\|(\Delta^2_t h)(s)\|}{|t|^{2}}\, dt<\infty\bigg\}, 
\lb{BesovR}
\end{equation}
Here the notation $C_{ub}(\bbR)$ stands for bounded uniformly continuous functions on $\bbR$.

For every $h\colon\bbR\to\bbR$ such that $h\circ g^{-1}\in B_{\infty 1}^1(\bbR)$ (in particular, if $h\circ g^{-1}\in C^2[-1,1]$) it follows from \cite[Theorem~4]{Pe90} (see also  \cite[Lemma~8.11.1]{Ya10}) that
$[h(A_+)-h(A_-)] \in\cB_1\big(L^2(\bbR)\otimes \bbC^m\big)$, and
\begin{align}\lb{trace formula}
{\tr}_{L^2(\bbR) \otimes \bbC^m}(h(A_+)-h(A_-))&={\tr}_{L^2(\bbR) \otimes \bbC^m}(h\circ g^{-1}(g(A_+))-h\circ g^{-1}(g(A_-)))   \no \\
&=\int_{\bbR}\xi(s;g(A_+),g(A_-))\,d(h\circ g^{-1}(s))   \no \\ 
& \hspace*{-1.5mm} 
\stackrel{\eqref{def_xiA}}{=}\int_{\bbR}\xi(\nu;A_+,A_-)h'(\nu)\,d\nu.
\end{align}
We note that the last integral above is finite. Indeed, the definition formula \eqref{def_xiA} tells us that $\xi(\cdot;A_+,A_-)$ is
integrable with the weight $\nu^{-3}$; on the other hand, since the
derivative $(h \circ g^{-1})'$ is bounded, it also implies that the
derivative $h'$ decays as $\nu^{-3}$ at $\pm\infty$. That is, the last
integral in \eqref{trace formula} is convergent. 

The following proposition presents a proof of the fact that the spectral shift function
 for the pair $(A_+,A_-)$ is constant. 
 
\begin{proposition}\lb{first lemma} Let $\Phi\in\mnn{W^{1,1}(\bbR)\cap C_b(\bbR)}$, 
$A_-=D \otimes I_m$, $A_+=A_-+ \Phi  $. Then 
$
\xi(\nu; A_+,A_-)= C \, \text{ for a.e.\ $\nu\in\bbR$.}
$
\end{proposition}
\begin{proof} Let $h$ be such that $h'$ is a Schwartz function. Then
$h\circ g^{-1}\in C^2[-1,1]$, and hence by \eqref{trace formula}, 
$[h(A_+)-h(A_-)] \in\cB_1\big(L^2(\bbR)\otimes \bbC^m\big)$. We claim that
\begin{equation}
{\tr}_{L^2(\bbR) \otimes \bbC^m}(h(A_+)-h(A_-))={\tr}_{L^2(\bbR) \otimes \bbC^m}(h(A_++\alpha)-h(A_-+\alpha)), \quad \alpha\in\bbR.
\lb{cl}
\end{equation}
As $D= - i d/dx$ on $\dom(D) = W^{1,2}(\bbR)$ is the generator of translations in $L^2(\bbR)$, introducing $\Psi_0(x)=e^{-i\alpha x}\otimes I_m$, $\alpha \in \bbR$, yields 
$A_-+\alpha= \Psi_0 A_- \ol{\Psi_0}$ and hence,
\begin{equation} 
h(A_-+\alpha)= \Psi_0^{} h (A_-) \Psi_0^{*},\quad h(A_++\alpha)
= \Psi_0^{} h (A_+) \Psi_0^{*}. 
\end{equation} 
Consequently,
\begin{align}
{\tr}_{L^2(\bbR) \otimes \bbC^m}(h(A_++\alpha)-h(A_-+\alpha))
&={\tr}_{L^2(\bbR) \otimes \bbC^m}(\Psi_0^{}[h(A_+)-h(A_-)] \Psi_0^{*})  \no \\ 
&= {\tr}_{L^2(\bbR) \otimes \bbC^m}(h(A_+)-h(A_-)),  
\end{align} 
by the unitary invariance of ${\tr}_{L^2(\bbR)\otimes\bbC^m}(\cdot)$, proves \eqref{cl}.
Since by \eqref{trace formula} the equality
\begin{align}
\begin{split}  
{\tr}_{L^2(\bbR) \otimes \bbC^m}(h(A_++\alpha)-h(A_-+\alpha)) 
&=\int_{\bbR}h'(\nu + \alpha)\xi(\nu;A_+,A_-)\,d\nu   \\
&=\int_{\bbR}h'(\nu)\xi(\nu - \alpha; A_+,A_-)\,d\nu
\end{split} 
\end{align} 
holds, one obtains 
$
\int_{\bbR}h'(\nu)[\xi(\nu - \alpha; A_+,A_-) - \xi(\nu; A_+,A_-)]\,d\nu=0. 
$
Since $h'$ is an arbitrary Schwartz function, it follows by the Lemma of Du Bois-Reymond that 
$
\xi(\nu - \alpha; A_+,A_-)-\xi(\nu; A_+,A_-)=0 \, \text{ for a.e.~$\nu\in\bbR$.}  
$
Since $\alpha \in \bbR$ was arbitrary, $\xi(\,\cdot\,; A_+,A_-)$ is constant a.e.\ on $\bbR$.
\end{proof}

Next, we present the main result of this paper. It is an analog of
\cite[Theorem~2.2]{GLMST11} and \cite[Proposition~1.3]{Pu08} for the examples studied 
here. We remark again that the hypotheses in these theorems in the cited two papers do
not apply to our examples, nevertheless, we obtain the same conclusion.

\begin{theorem}\lb{thm_PTF}
Assume Hypothesis \ref{h3.1} and let $z\in\bbC\backslash [0,\infty)$. 
Then $[g_z(A_+)-g_z(A_-)]\in\cB_1(L^2(\bbR)\otimes\bbC^m)$ and the following trace formula holds, 
\begin{align}
\begin{split} 
 \tr_{L^2(\bbR^2)\otimes\bbC^m}((\bsH_2 -z \, \bsI)^{-1}-(\bsH_1-z, 
\bsI)^{-1})    
  =\frac{1}{2z}\tr_{L^2(\bbR)\otimes\bbC^m} (g_z(A_+)-g_z(A_-)). 
\lb{TR}
\end{split} 
\end{align} 
\end{theorem}
\begin{proof} Using the result of Proposition \ref{g_complex} and following the proof of \cite[Lemma~7.3]{GLMST11}, one can prove that $[g_z(A_+)-g_z(A_-)]\in\cB_1(L^2(\bbR)\otimes\bbC^m)$ for all $z\in\C\backslash [0,\infty)$ and that the function $z\mapsto \frac{1}{2z}\tr(g_z(A_+)-g_z(A_-))$ is analytic on $\C\backslash [0,\infty)$.~Writing (see, e.g., \cite[Exercise~7.8]{We80})
\begin{align}
\begin{split} 
& \big((\bsH_2 -z \, \bsI)^{-1}- (\bsH_1-z \, 
\bsI)^{-1}\big)=(\bsH_2 -z_0 \, \bsI)(\bsH_2-z \, 
\bsI)^{-1}   \\
&\quad \times\big((\bsH_2 -z_0 \, \bsI)^{-1}-(\bsH_1-z_0 \, 
\bsI)^{-1}\big)(\bsH_1 -z \, \bsI)^{-1}(\bsH_1-z_0\, \bsI)^{-1}, 
\end{split} 
\end{align}
one obtains that the function on $z\mapsto\tr\big((\bsH_2 -z \, \bsI)^{-1}-(\bsH_1-z \, 
\bsI)^{-1}\big)$ is also analytic on $\C\backslash [0,\infty)$. By analytic continuation it is sufficient to prove the principle trace formula \eqref{principle} for $z<0$. 
For $z<0$, Proposition \ref{propAppr} implies  
\begin{align}
\begin{split} 
& \tr_{L^2(\bbR^2)\otimes\bbC^m}\big((\bsH_{2,n} - z \, \bsI)^{-1}-(\bsH_{1,n} - z \, \bsI)^{-1}\big) \\
& \quad = \frac{1}{2z}\tr_{L^2(\bbR)\otimes\bbC^m} (g_z(A_{+,n})-g_z(A_-)), \quad n \in \bbN. 
\end{split} 
\end{align} 
An application of Theorem \ref{conv_rhs} then implies that the right-hand side converges to $\frac{1}{2z}\tr_{L^2(\bbR)\otimes\bbC^m}\big(g_z(A_{+})-g_z(A_{-})\big)$, whereas by Proposition \ref{t3.7} the left-hand side converges to 
$
\tr_{L^2(\bbR^2)\otimes\bbC^m}\big((\bsH_{2} - z \, \bsI)^{-1} 
- (\bsH_{1} - z \, \bsI)^{-1}\big), 
$ 
implying \eqref{TR}.
\end{proof}

Proposition \ref{fritz2}  implies that the spectral shift function 
$\xi(\, \cdot \, ; \bsH_2, \bsH_1)$ for the pair $(\bsH_2, \bsH_1)$ is well-defined and satisfies
\begin{equation}
\xi(\, \cdot \, ; \bsH_2, \bsH_1) \in L^1\big(\bbR; (\lambda^2 + 1)^{-1} \, d\lambda\big). 
\end{equation} 
Since $\bsH_j\geq 0$, $j=1,2$, one uniquely introduces $\xi(\,\cdot\,; \bsH_2,\bsH_1)$ 
by requiring that
\begin{equation}
\xi(\lambda; \bsH_2,\bsH_1) = 0, \quad \lambda < 0.    \lb{2.46c}
\end{equation}

Having established the principle trace formula \eqref{principle}, we now prove that the 
spectral shift functions $\xi(\,\cdot\,;A_+,A_-)$ and $\xi(\,\cdot\,;\bsH_2,\bsH_1)$ coincide.

\begin{theorem}\lb{equ_ssf}
Assume Hypothesis \ref{h3.1}.
Then, for $($Lebesgue\,$)$ a.e.~$\lambda > 0$ and a.e.~$\nu \in \bbR$,  
\begin{equation}
\xi(\lambda; \bsH_2, \bsH_1) = \xi(\nu; A_+, A_-) + C,    \lb{4.1}
\end{equation}
with $C$ the constant in Proposition \ref{first lemma}. 
\end{theorem}
\begin{proof}
Following the arguments in \cite[Sections~ 7, 8]{GLMST11} one can obtain the 
Pushnitski-type formula
\begin{align}
\begin{split}  
\xi(\lambda; \bsH_2, \bsH_1)&=\frac{1}{\pi}\int_{-\lambda^{1/2}}^{\lambda^{1/2}}
\frac{\xi(\nu; A_+,A_-)\, d\nu}{(\lambda-\nu^2)^{1/2}}   \\ 
& = \frac{C}{\pi}\int_{-\lambda^{1/2}}^{\lambda^{1/2}}\frac{d\nu}{(\lambda-\nu^2)^{1/2}}  
= C \, \text{  for a.e.~$\lambda>0$,}    \lb{SSFP}
\end{split} 
\end{align} 
employing the a.e.\ constancy of $\xi(\nu; A_+,A_-)$ established in 
Proposition \ref{first lemma} in the second line of \eqref{SSFP}. 
To obtain some degree of completeness we sketch the principal steps in the derivation of the first equality in \eqref{SSFP}: One starts with 
\begin{equation}
\tr_{\cH}\big(g_{z}(A_+) - g_{z}(A_-)\big)
  = - z \int_{\bbR} \frac{\xi(\nu; A_+, A_-) \, d\nu}{(\nu^2 - z)^{3/2}},
\quad z\in\bbC\backslash [0,\infty).       \lb{2.38}
\end{equation}
The trace identity \eqref{TR} then yields
\begin{equation}
\int_{[0, \infty)}  \frac{\xi(\lambda; \bsH_2, \bsH_1) \, 
d\lambda}{(\lambda -z)^{2}}
= \frac{1}{2} \int_{\bbR} \frac{\xi(\nu; A_+, A_-) \, d\nu}{(\nu^2 - z)^{3/2}},
\quad z\in\bbC\backslash [0,\infty),
\end{equation}
and hence,
\begin{align}
& \int_{[0, \infty)}  \xi(\lambda; \bsH_2, \bsH_1) 
\bigg(\frac{d}{dz}(\lambda -z)^{-1}\bigg)
  d\lambda
= \int_{\bbR} \xi(\nu; A_+, A_-)\bigg(\frac{d}{dz} (\nu^2 - z)^{-1/2}\bigg) d\nu,  \no  \\
& \hspace*{9cm} z\in\bbC\backslash [0,\infty).     \lb{2.40}
\end{align}
Integrating \eqref{2.40} with respect to $z$ from a fixed point $z_0 
\in (-\infty,0)$ to
$z\in\bbC\backslash\bbR$ along a straight line connecting $z_0$ and 
$z$ then results in
\begin{align}
\begin{split}
& \int_{[0, \infty)}  \xi(\lambda; \bsH_2, \bsH_1)
\bigg(\frac{1}{\lambda - z} - \frac{1}{\lambda - z_0}\bigg) d\lambda      \\
& \quad = \int_{\bbR} \xi(\nu; A_+, A_-)\big[(\nu^2 - z)^{-1/2} - (\nu^2 - 
z_0)^{-1/2}\big] \, d\nu,
\quad z\in\bbC\backslash [0,\infty).    \lb{2.42}
\end{split} 
\end{align}

Applying the Stieltjes inversion formula (cf., e.g., \cite{AD56}, 
\cite[Theorem\ B.3]{We80}) to \eqref{2.42} then yields
\begin{align}
\xi (\lambda; \bsH_2, \bsH_1)
&= \lim_{\varepsilon\downarrow 0} \frac{1}{\pi} \int_{[0,\infty)}
\xi (\lambda'; \bsH_2, \bsH_1) \Im\big((\lambda'  - 
\lambda) - i \varepsilon)^{-1}\big) d\lambda'
\no  \\
& =  \lim_{\varepsilon\downarrow 0} \frac{1}{\pi}  \int_{\bbR} \xi(\nu; A_+, A_-)
\Im\big((\nu^2 - \lambda - i \varepsilon)^{-1/2}\big) d\nu     \no  \\
& = \frac{1}{\pi}  \int_{- \lambda^{1/2}}^{\lambda^{1/2}}
\frac{\xi(\nu; A_+, A_-) \, d\nu}{(\lambda - \nu^2)^{1/2}}   \no \\
& = \frac{C}{\pi}  \int_{- \lambda^{1/2}}^{\lambda^{1/2}}
\frac{d\nu}{(\lambda - \nu^2)^{1/2}}   \no \\
& = C  \, \text{ for a.e.\ $\lambda > 0$.}   \lb{2.43}
\end{align} 
\end{proof}

\section{Computation of the Spectral Shift Function for the Pair $(A_-,A_+)$} \lb{s6} 

By Proposition \ref{first lemma}, the spectral shift function $\xi(\,\cdot\,; A_+,A_-)$ is constant 
a.e.\ on $\bbR$. The principal goal in this section is to compute this constant. By Theorem \ref{equ_ssf} this also yields the value of $\xi(\, \cdot \,; \bsH_2, \bsH_1)$ a.e.\ on $(0,\infty)$.

In order to calculate the precise value of the constant $\xi(\,\cdot\,; A_+,A_-)$ we consider the  auxiliary function $\arctan (\cdot)$. Since $\arctan (\cdot)\circ g^{-1}\in C^2([-1,1])$,  it follows 
from the arguments preceding  \eqref{trace formula} that 
\begin{equation} \lb{arc_in_L_1}
[\arctan(A_+)-\arctan(A_-)] \in \cB_1\big(L^2(\bbR)\otimes\bbC^m\big), 
\end{equation} 
and that for a.e.~$\nu\in\bbR$, 
\begin{equation}\lb{via_arctan}
\tr_{L^2(\bbR)\otimes\bbC^m}(\arctan(A_+)-\arctan(A_-))
=\int_\bbR\frac{\xi(\nu; A_+,A_-)}{\nu^2 + 1}\,d\nu=\pi \xi(\nu; A_+,A_-).   
\end{equation}
Thus, our task is computing the value of the right-hand side in \eqref{via_arctan}.

We temporarily assume that $m=1$. Given $\phi \in W^{1,1}(\bbR)\cap C_b(\bbR)$, our aim 
is to represent the operator $[\arctan(A_+)-\arctan(A_-)]$ as an integral operator on 
$L^2(\bbR)$ (cf.\ \eqref{as_intop}). The unitary equivalence in \eqref{unitary} implies 
\begin{align} 
\begin{split} 
\arctan(A_+)-\arctan(A_-)&=\psi\arctan(A_-) \ol{\psi}-\arctan{A_-}\\
&=\psi\cF^{-1} \arctan(\cdot)\cF \ol{\psi}-\cF^{-1} \arctan(\cdot)\cF, 
\end{split} 
\end{align}
where $\cF$ denotes  the Fourier transform on $L^2(\bbR)$,
\begin{equation} 
(\cF \eta)(s)=(2\pi)^{-1/2} \slim_{T \to \infty}\int_{[-T,T]} e^{- i s x}\eta(x)\,dx, 
\quad \eta\in L^2(\bbR). 
\end{equation} 
Fix $\eta\in L^2(\bbR)\cap L^1(\bbR)$, then 
\begin{equation}\lb{Fourier}
(\cF^{-1} \arctan(\cdot)\cF \eta)(x)
= (2\pi)^{-1} \int_{\bbR^2}\eta(x_1)\arctan(s_0)e^{-is_0(x_1-x)}\,ds_0dx_1.
\end{equation}
We would like to identify the quantity on the right-hand side of \eqref{Fourier} with the
integral 
\begin{equation}\lb{formal_int}
(2\pi)^{-1/2}\int_\bbR \eta(s_1)(\cF \arctan)(s_1-s)\,ds_1.
\end{equation}
However, this identification is not possible due to the fact that 
\begin{equation} 
(\cF\arctan)(s)=\frac1{is}\cF\big(\frac{1}{1+x^2}\big)(s)
=\big(\frac{\pi}{2}\big)^{1/2}\frac{1}{is}e^{-|s|}, 
\end{equation} 
that is, the function $(\cF\arctan)(s_1-s)$ is discontinuous at the point $s_1-s=0$. Thus, we have to replace \eqref{formal_int} by the principal value 
\begin{equation}\lb{Fourier_pv}
\frac1{2i}\lim_{\varepsilon\rightarrow 0}\int_{|s_1-s|>\varepsilon}\frac{ e^{-|s_1-s|}\eta(s_1)}{s_1-s}ds_1.
\end{equation}
The identification of the right-hand sides of \eqref{Fourier} and \eqref{Fourier_pv} will be done 
in Lemma \ref{pv_arctan} below.

The next lemma is crucial for our representation of the operator $\arctan(A_+)-\arctan(A_-)$ as an integral operator. 
\begin{lemma}\lb{pv_arctan}
Define $D$ as in \eqref{defD} and write $\pv \frac{e^{-|x|}}{x}$ for the principle value integral as a tempered distribution (see appendix B). Then, 
\begin{equation} 
(\arctan{D})(\eta) = - \frac{1}{2i}\pv \frac{e^{-|x|}}{x}*\eta,\quad \eta\in S(\bbR). 
\end{equation} 
\end{lemma}
\begin{proof} For every $t>0$, consider the function $Q_t:\bbR\to\bbR$ defined by 
\begin{equation} 
Q_t(x)=\frac{x}{t^2+x^2} e^{-|x|}, \quad x\in\bbR. 
\end{equation}  
It is clear that $Q_t \in L^2(\bbR)$, $t > 0$, and hence the Fourier transform of $Q_t$, 
$t > 0$, is also square-integrable.
One can consider the function $Q_t$ as a tempered distribution \cite[Section~3.3]{SW71}. 
Next, we claim that 
$\lim_{t\downarrow 0}Q_t=\pv \frac{e^{-|x|}}{x}$
in the sense of tempered distributions, that is, 
$
\lim_{t\downarrow 0} Q_t(\eta)=\pv \frac{e^{-|x|}}{x}(\eta), \quad \eta\in S(\bbR) 
$
(see also a similar, but slightly different result in \cite[Proposition~3.1]{Du01}).
Indeed, one can write $Q_t=\frac12\big(\frac1{x+it}+\frac1{x-it}\big)e^{-|x|}$, and by the 
Sokhotski--Plemelj formulas (see, e.g., \cite[p.~33--34]{GS64}) obtain for every 
$\eta\in S(\bbR)$, 
\begin{equation} 
\lim_{t\downarrow 0}\int_\bbR\frac1{x+it}e^{-|x|}\eta(x) \, dx
=-i\pi\eta(0)+ \pv \int_\bbR \frac{e^{-|x|}\eta(x)}{x}\,dx, 
\end{equation} 
and 
\begin{equation} 
\lim_{t\downarrow 0}\int_\bbR\frac1{x-it}e^{-|x|}\eta(x)\,dx
=+i\pi\eta(0)+ \pv \int_\bbR \frac{e^{-|x|}\eta(x)}{x}\,dx, 
\end{equation} 
that is,  
\begin{align} 
\begin{split} 
\lim_{t\downarrow 0} Q_t(\eta)&=\lim_{t\downarrow 0}\int_\bbR\frac1{x+it}e^{-|x|}\eta(x)\,dx
+\lim_{t\downarrow 0}\int_\bbR\frac1{x-it}e^{-|x|}\eta(x)\,dx\\
&=\pv \frac{e^{-|x|}}{x}(\eta), \quad \eta\in S(\bbR). 
\end{split} 
\end{align}

Next, standard properties of the Fourier transform imply 
\begin{align}
\cF(Q_t)(s)&=\cF\bigg(\frac{x}{t^2+x^2} e^{-|x|}\bigg)(s)=\frac1{(2\pi)^{1/2}}\bigg(\cF\bigg(\frac{x}{t^2+x^2}\bigg)*\cF( e^{-|x|})\bigg)(s)\no\\
&=i\frac1{(2\pi)^{1/2}}\bigg(\bigg(\cF(\frac{1}{t^2+x^2})\bigg)'*\cF( e^{-|x|})\bigg)(s)\no\\
&=-i\frac1{(2\pi)^{1/2}}\bigg((e^{-t|x|}\sgn(x))*\frac{1}{1+x^2}\bigg)(s).
\end{align}
Lebesgue's dominated convergence theorem implies 
\begin{align}
\lim_{t\downarrow 0}\cF(Q_t)(s)&=-i\frac1{(2\pi)^{1/2}}\lim_{t\downarrow 0}\int_\bbR 
e^{-t|x|}\sgn(x)\frac{1}{1+(x-s)^2}\,dx      \lb{lim_fQt} \\
&=-i\frac1{(2\pi)^{1/2}}\int_\bbR\sgn(x)\frac{1}{1+(x-s)^2}\,dx 
=-\frac{2i}{(2\pi)^{1/2}}\arctan{(s)}.    \no
\end{align}
In addition, since the Fourier transform is a continuous map of $S'(\bbR)$ onto itself (see, e.g., \cite[Theorem~7.15]{Ru91}), 
$\cF\big(\pv\frac{e^{-|x|}}{x}\big)=\cF(\lim_{t\downarrow 0}Q_t)=\lim_{t\rightarrow 0} \cF(Q_t),$ 
 in $S'(\bbR)$, or equivalently, 
\begin{equation} 
\cF\bigg(\pv\frac{e^{-|x|}}{x}\bigg)(\eta) = \frac1{(2\pi i)^{1/2}}\int_\bbR \big((e^{-t|x|}\sgn(x))* 
(1+x^2)^{-1}\big)(s)\eta(s)\,ds,
\  \eta\in S(\bbR).
\end{equation}
Since 
\begin{equation} 
\big\|(e^{-t|\cdot|}\sgn(\cdot))* (1+|\cdot|^2)^{-1}\big\|_\infty 
\leq \big\|(1+|\cdot|^2)^{-1}\big\|_1 
\big\|e^{- t |\cdot|}\sgn(\cdot)\big\|_\infty\leq \pi
\end{equation}  
(see, e.g., \cite[Section~1.1, Theorem~1.3]{SW71}), and $\eta\in S(\bbR)$, one infers that the integrand $((e^{-t|x|}\sgn(x))* (1+x^2)^{-1})(\cdot)\eta(\cdot)$ is dominated by the integrable function $\pi\eta(\cdot)$. Hence, by \eqref{lim_fQt}, applying once again Lebesgue's dominated convergence theorem, one arrives at 
\begin{equation} 
\cF\bigg(\pv\frac{e^{-|x|}}{x}\bigg) (\eta) = - \frac{2i}{(2\pi)^{1/2}}\int_\bbR \arctan(s)\eta(s)\,ds, 
\end{equation} 
that is, the distribution $\cF\big(\pv\frac{e^{-|x|}}{x}\big)$ is, in fact, the function $-\frac{2i}{(2\pi)^{1/2}}\arctan(\cdot)$. Thus,
\begin{equation}\lb{arctan}
\cF^{-1}\arctan (\cdot)=-\frac{(2\pi)^{1/2}}{2i}\pv\frac{e^{-|x|}}{x}.
\end{equation}
Finally, for an arbitrary $\eta\in S(\bbR)$ by \cite[Theorem~7.19]{Ru91} one obtains  
\begin{align}
(\arctan(D)\eta)(s)&=(\cF^{-1} \arctan(\cdot)\cF\eta)(s)=\cF^{-1}(\arctan\cdot 
\cF\eta)(s)    \\
&=\frac{1}{(2\pi)^{1/2}}\big(\eta*\cF^{-1}\arctan\big)(s)\stackrel{\eqref{arctan}}{=}-\frac{1}{2i}\big(\eta*\pv\frac{e^{-|x|}}{x}\big)(s),   \no 
\end{align}
completing the proof. 
\end{proof}

For the special case where the operator $D$ is perturbed by a Schwartz function 
$\phi \in S(\bbR)$, we also state the following result: 

\begin{corollary}\lb{arctan_cor}
Let $\phi \in S(\bbR)$. Then the operator $A_+=D+ \phi$,  
$\dom(A_+) = \dom(D) = W^{1,2}(\bbR)$, in $L^2(\bbR)$ satisfies  
\begin{equation} 
(\arctan{A_+})\eta=-\frac{1}{2i}\psi\;\pv \frac{e^{-|x|}}{x}*(\ol{\psi}\eta),\quad \eta\in S(\bbR). 
\end{equation} 
\end{corollary}
\begin{proof} Since $\phi$ is a Schwartz test function, 
$\psi(x)=\exp(- i \int_0^x \phi(x') \, dx')$ is infinitely differentiable and 
$\ol{\psi}\eta\in S(\bbR)$ for every $\eta\in S(\bbR)$. Hence, one can write
\begin{equation} 
(\arctan{A_+})\eta=\psi\arctan(D) \psi \, \eta=\psi \, \arctan(D)(\ol{\psi}\eta), 
\end{equation} 
and Lemma \ref{pv_arctan} completes the proof.
\end{proof}

\begin{proposition}\lb{trace_for_S(R)}
Let $\phi \in S(\bbR)$ and introduce $A_-=D$, $A_+=A_-+ \phi$, 
$\dom(A_\pm) = W^{1,2}(\bbR)$, in $L^2(\bbR)$. Then,  
\begin{equation} 
[\arctan(A_+)-\arctan(A_-)] \in \cB_1\big(L^2(\bbR) \otimes \bbC^m\big),   \lb{arctanB1}
\end{equation} 
and 
\begin{equation} 
\tr_{L^2(\bbR) \otimes \bbC^m}(\arctan(A_+)-\arctan(A_-))=\frac12\int_\bbR \phi(x) \, dx. 
\lb{trarctan}
\end{equation} 
\end{proposition}
\begin{proof}
The claim \eqref{arctanB1} has been discussed already in the context of \eqref{arc_in_L_1}. 
To prove \eqref{trarctan}, let $\eta\in S(\bbR)$. Combining Lemma \ref{pv_arctan} and 
Corollary \ref{arctan_cor} one infers, 
\begin{align}
& ((\arctan(A_+) - \arctan(A_-))\eta)(y)   \no \\
& \quad =-\frac{1}{2i}\big(\psi\;\pv \frac{e^{-|x|}}{x}*(\ol{\psi}\eta)(y) 
- \pv\frac{e^{-|x|}}{x}*(\eta)(y)\big)   \no \\
& \quad =-\frac{1}{2i}\lim_{\varepsilon\downarrow 0}\int_{|x|>\varepsilon}
\big(\psi(y)\ol{\psi(y-x)}-1\big)\frac{e^{-|x|}}{x}\eta(y-x)\,dx   \no \\
& \quad =-\frac{1}{2i}\lim_{\varepsilon\downarrow 0}\int_{|y-x|>\varepsilon}\big(\psi(y)-\psi(x)\big)\ol{\psi(x)} \, \frac{e^{-|y - x|}}{y - x}\eta(x)\,dx   \no \\
& \quad =-\frac{1}{2i}\int_{\bbR}\ol{\psi(x)} \, \frac{\psi(y)-\psi(x)}{y - x}e^{-|y - x|}\eta(x)\,dx,
\end{align}
where the last equality is due to continuity of 
$\ol{\psi(x)} \, \frac{\psi(y) - \psi(x)}{y - x}e^{-|y - x|}\eta(x)$ for all $x \in \bbR$ (given 
$y \in \bbR$).

Next, we will show that the preceding equality can be extended to arbitrary 
$\eta\in L^2(\bbR)$ and thus  
\begin{equation}\lb{as_intop}
\big((\arctan(A_+)-\arctan(A_-))\eta\big)(y)
=-\frac{1}{2i}\int_{\bbR}\ol{\psi(x)} \, \frac{\psi(y) - \psi(x)}{y - x}e^{-|y - x|}\eta(x)\,dx 
\end{equation} 
holds. Since $S(\bbR)$ is dense in $L^2(\bbR)$, for every $\eta\in L^2(\bbR)$ there exists 
a sequence $\{\eta_n\}_{n=1}^\infty \subset S(\bbR)$, such that $\|\eta_n-\eta\|_2\underset{n\to\infty}{\longrightarrow}  0$. 
On one hand, 
\begin{equation} 
\|(\arctan(A_+)-\arctan(A_-))(\eta_n-\eta)\|_2 \underset{n\to\infty}{\longrightarrow} 0, 
\end{equation}  
since $[\arctan(A_+)-\arctan(A_-)] \in \cB\big(L^2(\bbR)\big)$. 
On the other hand, we claim that the integral operator $K$ in $L^2(\bbR)$ with integral kernel 
\begin{equation} 
K(x,y)=-\frac{1}{2i}\ol{\psi(x)} \, \frac{\psi(y) - \psi(x)}{y - x}e^{-|x - y|} 
\end{equation}  
is a bounded operator on $L^2(\bbR)$. By 
\cite[Equation~(2.2)]{BS77} this will follow from the estimates 
\begin{equation}\lb{estimate_kernel}
\|K(\,\cdot\,,\,\cdot\,)\|_{L^\infty(\bbR;dx;L^1(\bbR;dy))} < \infty, \quad 
\|K(\,\cdot\,,\,\cdot\,)\|_{L^\infty(\bbR;dy,L^1(\bbR;dx))}<\infty 
\end{equation}
(Bochner norms are used in this context).
Since $|K(x,y)|=\frac12\big|\frac{\psi(y) - \psi(x)}{y - x}\big|e^{-|y - x|}$, it is sufficient to estimate 
one of the two norms in \eqref{estimate_kernel}. We estimate the norm of 
$\|K(\,\cdot\,,\,\cdot\,)\|_{L^\infty(\bbR;dx;L^1(\bbR;dy))}$ next:  
\begin{align}
& \|K(\,\cdot\,,\,\cdot\,)\|_{L^\infty(\bbR;dx;L^1(\bbR;dy))}=\frac12\sup_{x\in\bbR}\bigg(\int_\bbR
\bigg|\frac{\psi(y) - \psi(x)}{y - x}\bigg|e^{-|y - x|}\,dy\bigg)    \no \\
& \quad \leq\frac12\sup_{x \in \bbR} \bigg(\int_\bbR\sup_{(x,y)\in\bbR^2}
\bigg|\frac{\psi(y) - \psi(x)}{y -x}\bigg| e^{-|y  - x|}\,dy\bigg)   \no \\
& \quad \leq\frac12\|\psi'\|_\infty \sup_{x\in\bbR} \bigg(\int_\bbR e^{-|y - x|}\,dy\bigg)  
\leq \|\psi'\|_\infty<\infty.
\end{align}

Hence indeed, $K \in \cB\big(L^2(\bbR)\big)$ and 
$K \eta_n \underset{n\to\infty}{\longrightarrow} K \eta$ in $L^2(\bbR)$.
Thus, equality \eqref{as_intop} holds for all $\eta\in L^2(\bbR)$. 
Moreover, since the integral kernel $K(\,\cdot\,,\,\cdot\,)$ is continuous, invoking 
\eqref{arctanB1} implies  
\begin{align} 
\begin{split} 
& {\tr}_{L^2(\bbR) \otimes \bbC^m}(\arctan(A_+) - \arctan(A_-))=\int_\bbR K(x,x)\,dx\\
& \quad =-\frac{1}{2i}\int_\bbR \psi'(x) \ol{\psi(x)} \,dx = \frac12\int_\bbR \phi(x) \, dx, 
\end{split} 
\end{align}
completing the proof. 
\end{proof}

Next, we proceed to the case of arbitrary $m \in \bbN$. 

\begin{proposition}\lb{matrix_trace}
Let $\Phi \in \mnn{S(\bbR)}$ and consider $A_-=D\otimes I_m$, $A_+=A_-+ \Phi$, 
$\dom(A_{\pm}) = W^{1,2}(\bbR) \otimes \bbC^m$, in $L^2(\bbR) \otimes \bbC^m$, 
$m \in \bbN$. Then,  
\begin{equation} 
\tr_{L^2(\bbR)\otimes \bbC^m}(\arctan(A_+)-\arctan(A_-))
=\frac12\int_\bbR \tr_{\bbC^m}(\Phi(x))\,dx.    \lb{trarctana} 
\end{equation} 
\end{proposition}
\begin{proof}
Employing the unitary equivalence in \eqref{fc_D,DMf}, one writes  
\begin{align}
& \Psi(\,\cdot\,, x_0) \arctan(A_-) \Psi(\,\cdot\,,x_0)^*    \no \\
& \quad =\sum_{j,k=1}^m P_{j,k}\otimes \Psi_{j,k}(\,\cdot\,, x_0) 
\arctan(D)\sum_{\ell,r =1}^mP_{\ell,r}\otimes \overline{\Psi_{r,\ell}(\,\cdot\,,x_0)}   \no \\
& \quad =\sum_{j,k,\ell,r=1}^mP_{j,k}P_{\ell,r} \otimes \Psi_{j,k}(\,\cdot\,, x_0) \arctan(D)
\overline{\Psi_{r,\ell}(\,\cdot\,,x_0)}   \no \\
& \quad =\sum_{j,\ell,r=1}^mP_{j,r}\otimes \Psi_{j,\ell}(\,\cdot\,, x_0)
\arctan(D)\overline{\Psi_{r,\ell}(\,\cdot\,,x_0)},
\end{align}
utilizing $P_{j,k} P_{\ell,r} = P_{j,r} \delta_{k,\ell}$, $j,k,\ell,r = 1,\dots,m$ (cf.\ \eqref{Pjk}). 
Thus,  
\begin{equation} 
\Psi(\\,\cdot\,, x_0) \arctan(A_-) \overline{\Psi(\,\cdot\,,x_0)}
=\sum_{k, \ell=1}^mP_{k,\ell}\otimes\sum_{j=1}^m \Psi_{k,j}(\,\cdot\,, x_0)\arctan(D) 
\overline{\Psi_{\ell,j}(\,\cdot\,,x_0)}.
\end{equation}  
Hence, the operator 
$[\Psi(\,\cdot\,, x_0) \arctan(A_-) \Psi(\,\cdot\,,x_0)^*-\arctan(A_-)]$ is an 
operator-valued block matrix with $(k,k)$-th entry given by 
\begin{align} 
\begin{split} 
& \sum_{j=1}^m \Psi_{k,j}(\,\cdot\,, x_0) \arctan(D) \overline{\Psi_{k,j}(\,\cdot\,,x_0)}-\arctan(D) \\
& \quad =\sum_{j=1}^m \Psi_{k,j}(\,\cdot\,, x_0) \cF^{-1} \arctan(\cdot) \cF 
\overline{\Psi_{k,j}(\,\cdot\,,x_0)}-\cF^{-1} \arctan (\cdot) \cF. 
\end{split} 
\end{align}
Thus, 
\begin{align}\lb{trace_aux}
& \tr_{L^2(\bbR)\otimes\bbC^m}(\arctan(A_+)-\arctan(A_-))   \\
& \quad =\sum_{k=1}^m\tr_{L^2(\bbR)}\bigg(\sum_{j=1}^m \Psi_{k,j}(\,\cdot\,, x_0) \cF^{-1} \arctan(\cdot)\cF \overline{\Psi_{k,j}(\,\cdot\,,x_0)}-\cF^{-1} \arctan(\cdot)\cF\bigg).   \no 
\end{align}
Fix $k=1,\dots,m$. Since $\Psi(\,\cdot\,, x_0)$ is infinitely differentiable  (cf.\  
Remark \ref{Psi}\,$(i)$), it follows that $\Psi_{j,k}(\,\cdot\,, x_0) \eta \in S(\bbR)$ for every 
$\eta\in S(\bbR)$, $j,k = 1, \dots,m$, and hence by Lemma \ref{pv_arctan} and Corollary \ref{arctan_cor}, 
\begin{align}
& \bigg(\sum_{j=1}^m \Psi_{k,j}(\,\cdot\,, x_0) \cF \arctan(\cdot) \cF^{-1} 
\overline{\Psi_{k,j}(\,\cdot\,,x_0)}-\cF \arctan(\cdot) \cF^{-1}\bigg)\eta(x)     \no \\
& \quad =\lim_{\varepsilon\downarrow 0}\bigg(\sum_{j=1}^m \frac{(-1)}{2i}\int_{|x-x_1|>\varepsilon}  \Psi_{k,j}(x,x_0)\overline{\Psi_{k,j}(x_1,x_0)}\eta(x_1)(x-x_1)^{-1}e^{-|x-x_1|}\,dx_1    \no \\
&\qquad +\frac{1}{2i}\int_{|x-x_1|>\varepsilon} (x-x_1)^{-1}e^{-|x-x_1|}\eta(x_1)\,dx_1\bigg), 
\quad \eta \in S(\bbR).  
\end{align}
Since the matrix $\Psi$ is unitary, $\sum_{j=1}^m \Psi_{k,j}(\,\cdot\,,x_0)
\overline{\Psi_{\ell,j}(\,\cdot\,,x_0)} = \delta_{k,\ell}$, and hence 
\begin{align}
& \bigg(\sum_{j=1}^m \Psi_{k,j}(\,\cdot\,, x_0) \cF \arctan(\cdot) \cF^{-1} 
\overline{\Psi_{k,j}(\,\cdot\,,x_0)}-\cF \arctan(\cdot) \cF^{-1}\bigg)\eta(x)    \no \\
& \quad =\lim_{\varepsilon\downarrow 0}\bigg(
\sum_{j=1}^m \frac{(-1)}{2i}\int_{|x-x_1|>\varepsilon}\Psi_{k,j}(x,x_0)\overline{\Psi_{k,j}(x_1,x_0)} \eta(x_1)(x-x_1)^{-1}e^{-|x-x_1|}\,dx_1   \no \\
&\qquad +\sum_{j=1}^m\frac{1}{2i}\int_{|x-x_1|>\varepsilon}\Psi_{k,j}(x_1,x_0)
\overline{\Psi_{k,j}(x_1,x_0)} (x-x_1)^{-1}e^{-|x-x_1|}\eta(x_1)\,dx_1\bigg)    \no \\
& \quad =-
\sum_{j=1}^m \frac{1}{2i}\int_\bbR \frac{\Psi_{k,j}(x,x_0)-\Psi_{k,j}(x_1,x_0)}{x-x_1}
\overline{\Psi_{k,j}(x_1,x_0)} \eta(x_1)e^{-|x-x_1|}\,dx_1.
\end{align}
Arguing as in the proof of equality \eqref{as_intop} one concludes 
\begin{align} 
\begin{split} 
& \bigg(\sum_{j=1}^m \Psi_{k,j}(\,\cdot\,, x_0) \cF \arctan (\cdot) \cF^{-1} 
\overline{\Psi_{k,j}(\,\cdot\,,x_0)}-\cF \arctan(\cdot) \cF^{-1} \bigg)\eta(x)\\ 
& \quad =-\sum_{j=1}^m \frac{1}{2i}\int_\bbR \frac{\Psi_{k,j}(x,x_0)-\Psi_{k,j}(x_1,x_0)}{x-x_1}
\overline{\Psi_{k,j}(x_1,x_0)}\eta(x_1)e^{-|x-x_1|}\,dx_1
\end{split} 
\end{align}
for arbitrary $\eta\in L^2(\bbR)$. Since the function $\Psi_{j,k}(\,\cdot\,, x_0) \in C^{\infty}(\bbR)$, 
it follows that the integral kernel 
\begin{equation} 
K_{k}(x,x_1)=\frac1{2i}\sum_{j=1}^m\frac{\Psi_{k,j}(x,x_0)-\Psi_{k,j}(x_1,x_0)}{x-x_1}
\overline{\Psi_{k,j}(x_1,x_0)} e^{-|x-x_1|}
\end{equation} 
is continuous and 
$
K_k(x,x)=\frac1{2i}\sum_{j=1}^m\Psi'_{k,j}(x,x_0)\overline{\Psi_{k,j}(x,x_0)}.
$
Hence, 
\begin{align}
& \tr_{L^2(\bbR)}\bigg(\sum_{j=1}^m \Psi_{k,j}(\,\cdot\,, x_0) \cF \arctan(\cdot) \cF^{-1} 
\overline{\Psi_{k,j}(\,\cdot\,,x_0)}-\cF \arctan(\cdot) \cF^{-1}\bigg)    \\ 
& \quad =\int_\bbR K_k(x,x)\,dx =-\frac1{2i}\sum_{j=1}^m\int_\bbR \Psi'_{k,j}(x,x_0)
\overline{\Psi_{k,j}(x,x_0)}\,dx \stackrel{\eqref{5.10}}{=}\frac1{2}\int_\bbR \Phi_{k,k}(x)\,dx. \no
\end{align}
Finally, by \eqref{trace_aux} one obtains 
\begin{equation}
\tr_{L^2(\bbR)\otimes\bbC^m}(\arctan(A_+)-\arctan(A_-))
=\frac1{2}\sum_{k=1}^m\int_\bbR \Phi_{k,k}(x)dx
=\frac12\int_\bbR\tr_{\bbC^m}(\Phi(x))dx.
\end{equation}
\end{proof}

By Proposition \ref{matrix_trace} and equality \eqref{via_arctan}, 
\begin{equation}\lb{xi_for_S(R)}
\xi(\nu; A_+,A_-)=\frac{1}{2\pi}\int_\bbR\tr_{\bbC^m}(\Phi(x))\,dx 
 \, \text{ for a.e.\ $\nu \in \bbR$,}
\end{equation} 
as soon as $\Phi\in \mnn{S(\bbR)}$. The following theorem extends this result to an 
arbitrary $\Phi\in\mnn{W^{1,1}(\bbR)\cap C_b(\bbR)}$.

\begin{theorem}\lb{value_via_arctan} 
Assume that $\Phi\in\mnn{W^{1,1}(\bbR)\cap C_b(\bbR)}$, $m\in \bbN$. Then, 
\begin{equation} 
\xi(\nu; A_+,A_-)=\frac{1}{2\pi}\int_\bbR\tr_{\bbC^m}(\Phi(x))\,dx 
\, \text{ for a.e.\ $\nu \in \bbR$.}  
\end{equation} 
\end{theorem}
\begin{proof} Since $\Phi\in\mnn{W^{1,1}(\bbR)}$, one concludes the existence of a sequence $\{\Phi_n\}_{n=1}^\infty\subset\mnn{S(\bbR)}$ such that $\|\Phi_n-\Phi\|_{1,1} \underset{n\to\infty}{\longrightarrow}  0$. By Lemma \ref{unit_equiv}, $A_- + \Phi_n = 
\Psi_n(\,\cdot\,,x_0) A_- \Psi_n(\,\cdot\,,x_0)^*$, and 
\begin{align}
\begin{split}  
& A_-+ \Phi  =A_- + \Phi_n + (\Phi-\Phi_n)   \\
& \quad = \Psi_n(\,\cdot\,,x_0) \big(A_- + \Psi_n(\,\cdot\,,x_0)^* (\Phi-\Phi_n)  
\Psi_n(\,\cdot\,,x_0) \big) \Psi_n(\,\cdot\,,x_0)^*,  
\end{split} 
\end{align} 
that is, $A_-+ \Phi  $ is unitarily equivalent to 
$A_- + \Psi_n(\,\cdot\,,x_0)^* (\Phi-\Phi_n)\Psi_n(\,\cdot\,,x_0)$. Hence, applying 
Proposition \ref{g_complex}, 
\begin{align}
& |{\tr}_{L^2(\bbR) \otimes \bbC^m} \big(g(A_-+ \Phi  )- g(A_-)\big) 
- {\tr}_{L^2(\bbR) \otimes \bbC^m}\big(g(A_- + \Phi_n)-g(A_-)\big)|   \no \\
& \quad \leq \|g(A_-+ \Phi  )-g(A_- + \Phi_n)\|_{\cB_1(L^2(\bbR)  \otimes \bbC^m)}  \no \\
& \quad = \|g(A_- + \Psi_n(\,\cdot\,,x_0)^* (\Phi-\Phi_n) 
\Psi_n(\,\cdot\,,x_0)) - g(A_-)\|_{\cB_1(L^2(\bbR) \otimes \bbC^m)}  \no \\
& \quad \leq\ C \, \| \Psi_n(\,\cdot\,,x_0)^* (\Phi-\Phi_n) \Psi_n(\,\cdot\,,x_0)\|  \no \\
& \quad \leq\ C \, \|\Phi-\Phi_n\|_{1,1} \underset{n\to\infty}{\longrightarrow} 0,
\end{align}
that is,  
\begin{equation} 
{\tr}_{L^2(\bbR) \otimes \bbC^m}\big(g(A_-+ \Phi  )-g(A_-)\big)=\lim_{n\rightarrow\infty}
{\tr}_{L^2(\bbR) \otimes \bbC^m}\big(g(A_- + \Phi_n)-g(A_-)\big).
\end{equation} 
Since $\Phi_n\in\mnn{S(\bbR)}$, $n \in \bbN$, equalities \eqref{tr for g} and 
\eqref{xi_for_S(R)} imply that 
\begin{align} 
{\tr}_{L^2(\bbR) \otimes \bbC^m}\big(g(A_-+ \Phi  )-g(A_-)\big) 
&=\lim_{n\rightarrow\infty}\xi(\,\cdot\,;A_+,A_-)\big(g(+\infty)-g(-\infty)\big)   \no \\
&=\lim_{n\rightarrow\infty} \frac{1}{\pi}\int_\bbR \tr_{\bbC^m}(\Phi_n(x))\,dx,
\end{align} 
Moreover, the convergence $\|\Phi_n-\Phi\|_{1,1} \underset{n\to\infty}\longrightarrow 0$ 
implies that $\|\Phi_n-\Phi\|_1 \underset{n\to\infty}\longrightarrow 0$, and thus, 
$\int_\bbR\tr_{\bbC^m}(\Phi_n(x))\,dx \underset{n\to\infty}\longrightarrow \int_\bbR\tr_{\bbC^m}(\Phi(x))\,dx$,  
that is,  
\begin{equation} 
{\tr}_{L^2(\bbR) \otimes \bbC^m}\big(g(A_-+ \Phi  )-g(A_-)\big) 
= \frac{1}{\pi}\int_\bbR \tr_{\bbC^m}(\Phi(x))\,dx. 
\end{equation} 
Using once more equality \eqref{tr for g}, one concludes  
\begin{equation} 
\xi(\nu;A_+,A_-)=\frac{1}{2\pi}\int_\bbR\tr_{\bbC^m}(\Phi(x))\,dx 
 \, \text{ for a.e.\ $\nu \in \bbR$,}  
\end{equation} 
finishing the proof. 
\end{proof}

Combining Theorems \ref{value_via_arctan} and \ref{equ_ssf}, one arrives at the following  fact.

\begin{corollary} Assume that $\Phi\in\mnn{W^{1,1}(\bbR)\cap C_b(\bbR)}$, $m\in \bbN$. Then,  
\begin{equation} 
\xi(\lambda;\bsH_2,\bsH_1)=\frac1{2\pi}\int_{\bbR}\tr_{\bbC^m}(\Phi(x))\,dx  
\, \text{ for a.e.~ $\lambda>0$.} 
\end{equation} 
\end{corollary}

In addition Theorem \ref{value_via_arctan} and equality \eqref{trace formula} implies the 
following result. 

\begin{corollary}\lb{Trace_g_z}
Assume that $\Phi\in\mnn{W^{1,1}(\bbR)\cap C_b(\bbR)}$, $m\in \bbN$. Then,  
\begin{equation}  
\tr_{L^2(\bbR)\otimes\bbC^m}(g_z(A_+)-g_z(A_-)) 
=\frac1{\pi}\int_{\bbR}\tr_{\bbC^m}(\Phi(x))\,dx, \quad z\in\bbC\backslash [0,\infty).
\end{equation}  
\end{corollary}

\begin{remark}
By unitary equivalence (see Lemma \ref{unit_equiv}) and unitary invariance of the trace 
$\tr_{L^2(\bbR)\otimes\bbC^m}(\cdot)$, one concludes 
\begin{equation}  
\tr_{L^2(\bbR)\otimes\bbC^m}([\Psi(\,\cdot\,,x_0),g(A_-)])
=\frac1{\pi}\int_{\bbR}\tr_{\bbC^m}(\Phi(x))\,dx. 
\end{equation} 
So, if the matrix $\Phi\in \mnn{W^{1,1}(\bbR)\cap C_b(\bbR)}$ is such that 
$
\int_{\bbR}\tr_{\bbC^m}(\Phi(x))\,dx\neq 0, 
$
$[\Psi(\,\cdot\,,x_0),g(A_-)]$ represents an explicit example of a commutator with nonzero trace. \hfill $\diamond$
\end{remark}

\section{The Witten Index} \lb{WI_section}

In this section we briefly discuss the Witten index for the model operator $\bsD_\bsA^{}$ introduced in \eqref{def_D_A} following the detailed treatment in \cite{CGPST15}.

Firstly we recall the definition of the resolvent regularized Witten index (see \cite{BGGSS87})
\begin{definition} \lb{d8.1} 
Let $T$ be a closed, linear, densely defined operator in $\cH$ and  
suppose that for some $($and hence for all\,$)$ 
$z \in \bbC \backslash [0,\infty)$,  
\begin{equation} 
\big[(T^* T - z I_{\cH})^{-1} - (TT^* - z I_{\cH})^{-1}\big] \in \cB_1(\cH).   \lb{8.1} 
\end{equation}  
Then introducing the resolvent regularization 
\begin{equation}
\Delta_r(T, \lambda) = (- \lambda) \tr_{\cH}\big((T^* T - \lambda I_{\cH})^{-1}
- (T T^* - \lambda I_{\cH})^{-1}\big), \quad \lambda < 0,        \lb{8.2} 
\end{equation} 
the resolvent regularized Witten index $W_r (T)$ of $T$ is defined by  
\begin{equation} 
W_r(T) = \lim_{\lambda \uparrow 0} \Delta_r(T, \lambda),      \lb{8.3}
\end{equation}
whenever this limit exists. 
\end{definition} 

Here, in obvious notation, the subscript ``$r$'' indicates the use of the resolvent 
regularization (for a semigroup or heat kernel regularization we refer to \cite{CGPST15}).
Before proceeding to compute the Witten index for the model operator $\bsD_\bsA^{}$, 
we recall the known consistency between the Fredholm and Witten index 
whenever $T$ is Fredholm:

\begin{theorem} $($\cite{BGGSS87}, \cite{GS88}.$)$  \lb{t8.2} 
Suppose that $T$ is a Fredholm operator in $\cH$. If \eqref{8.1} holds, then the 
resolvent regularized Witten index $W_r(T)$ exists, equals the Fredholm index, 
$\ind (T)$, of $T$, and
\begin{equation} 
W_r(T) =  \ind (T) = \xi(0_+; T T^*, T^* T).    \lb{8.4}
\end{equation}
\end{theorem}

\begin{remark}
Following the proof of \cite[Theorem 2.6]{CGPST15} one can show that the operator $\bsD_\bsA^{}$ is Fredholm if and only if the operators $A_\pm$ are boundedly invertible. Therefore, the fact that $\sigma(A_-)=\sigma(A_+)=\bbR$ implies that the operator $\bsD_\bsA^{}$ is not Fredholm. Furthermore, $\sigma_{\rm ess}(\bsD_\bsA^{})=\bbC$ (cf. \cite[Corollary 2.8]{CGPST15}).
 \hfill$\diamond$
\end{remark}

Although $\bsD_\bsA^{}$ is not a Fredholm operator in $L^2(\bbR^2)$, we can determine 
the resolvent regularized Witten index of $\bsD_\bsA^{}$ (generalizing \cite{CGLPSZ14}) as follows:

\begin{theorem} 
Assume Hypothesis \ref{h3.1}. Then $W_r(\bsD_\bsA^{})$ exists and equals 
\begin{equation}
W_r(\bsD_\bsA^{}) = \xi(0_+; \bsH_2, \bsH_1) = 
\xi(0; A_+, A_-) =\frac1{2\pi}\int_{\bbR}\tr_{\bbC^m}(\Phi(x))\,dx       \lb{8.5}
\end{equation}
\end{theorem}
\begin{proof}
The equality $ \xi(0_+; \bsH_2, \bsH_1) = 
\xi(0; A_+, A_-) $ follows immediately from Theorem \ref{equ_ssf}. 
By Theorem \ref{value_via_arctan} we know that $\xi(\cdot; A_+, A_-)$ is constant  and for a.e.~$\nu \in \bbR$, 
\begin{equation} 
\xi(\nu; A_+, A_-) = \f{1}{2 \pi} \int_{\bbR} \tr_{\bbC^m}(\Phi(x))\,dx. 
\end{equation} 
Thus, combining the trace formula \eqref{principle} and \eqref{trace formula} one obtains 
the equality 
\begin{align}
& z\tr_{L^2(\bbR^2) \otimes \bbC^m} \big((\bsH_2 - z \, \bsI)^{-1}-(\bsH_1 - z \, \bsI)^{-1}\big) \no \\
& \quad = \frac{1}{2}\tr_{L^2(\bbR) \otimes \bbC^m} (g_z(A_+)-g_z(A_-))
=-\frac{z}{2}\int_\bbR\frac{\xi(\nu;A_+,A_-)}{(\nu^2-z)^{3/2}}\, d\nu  \no \\
& \quad = \xi(0;A_+,A_-).
\end{align}
Hence, by definition of the regularized Witten index one infers that 
\begin{align}
W_r(\bsD_\bsA^{}) &= \lim_{\lambda \uparrow 0}z\tr_{L^2(\bbR^2) \otimes \bbC^m}
\big((\bsH_2 - z \, \bsI)^{-1}-(\bsH_1 - z \, \bsI)^{-1}\big)  \no \\
& =\xi(0;A_+,A_-)=\f{1}{2 \pi} \int_{\bbR} \tr_{\bbC^m}(\Phi(x))\,dx, 
\end{align}
finishing the proof. 
\end{proof}

\appendix
\section{Connections to the Theory of Fredholm Modules}      \lb{sA} 

In this appendix we use the terminology of \cite{CGRS2}.
For the function $\psi$ on $\bbR$ the study of the commutator
$[\sgn(D),\psi]$ acting in the Hilbert space $L^2(\bbR)$, where 
\begin{equation} 
D= - i d/dx, \quad \dom(D) = W^{1,2}(\bbR),
\end{equation}  
is relevant to the discussion of Fredholm modules in \cite{Co94}. 
The function $g(t)=t(1+t^2)^{-1/2}$ is 
a ``smoothed" sign function and it is shown in \cite[Section~2.7]{CGRS2} that the operators $\sgn(D)$ and $g(D)$ define Fredholm modules for the spectral triple $(C_0^\infty(\bbR), L^2(\bbR), D)$, which lie in the same Kasparov class.
However, our results show that the trace class properties of the commutator $[\sgn(D),\psi]$  differ dramatically from the trace class properties of the commutator $[g(D),\psi]$, that is, the following result holds:

\begin{proposition}\lb{g_versus_sgn}
Let $f\in W^{1,1}(\bbR)\cap C_b(\bbR)$ and assume that  
\begin{equation} 
\bigg(\int_0^\infty f(x)\,dx - \int_0^{-\infty} f(x)\,dx\bigg) \notin 2\pi \bbZ. 
\end{equation} 
Then for the function 
$\psi(x)=\exp(- i \int_0^x f(y)\,dy)$, the commutator of $\psi$ and $g(D)$ is trace class, 
$[\psi,g(D)] \in \cB_1\big(L^2(\bbR)\big)$, while that of $\psi$ and $\sgn(D)$ is not, 
$[\psi,\sgn(D)] \notin \cB_1\big(L^2(\bbR)\big)$.
\end{proposition}

Our proof is based on Peller's theorem \cite{Pe80}  (see also \cite[Theorem IV.3.4]{Co94}). The required results in \cite{Pe80} and \cite[Theorem IV.3.4]{Co94} are stated for the circle, and so a preparatory lemma, connecting the commutators $[\sgn(D),\psi]$ and 
$[\sgn(D_0),\psi\circ\gamma^{-1}]$, is required. Here $D_0$ in $L^2(\bbT)$ is defined by 
\begin{equation} 
D_0 = - i d/dx, \quad \dom(D_0) = \big\{g \in L^2(\bbT) \, \big| \, g \in AC([0, 2 \pi]); 
\, g(0_+) = g(2 \pi_-)\big\},  
\end{equation}  
and $\gamma$ denotes the Cayley transform
$
\gamma(x)=({x-i})({x+i})^{-1},\quad x\in\bbR.
$

\begin{lemma} \lb{lA.2} 
The commutator $[\psi,\sgn(D)]$ belongs to a symmetric ideal $\cE$ (not necessarily proper) in $\cB(L^2(\bbR))$ if and only if $[\psi\circ\gamma^{-1},\sgn(D_0)]$ belongs to the symmetric ideal 
$\cE$ in $\cB(L^2(\bbT))$. In particular, the operator $[\psi,\sgn(D)]$ is bounded in $L^2(\bbR)$ if and only if $[\psi\circ\gamma^{-1},\sgn(D_0)]$ is bounded in $L^2(\bbT)$.
\end{lemma}
\begin{proof}
Following \cite[p.~252]{Ni86}, one considers the operator 
$U:L^2(\mathbb{T})\to L^2(\bbR)$ defined by setting
$
(Uf)(x)=\pi^{-1/2}(x+i)^{-1}f(\gamma(x)),
\ \  x\in\bbR. 
$
This operator is an isometry from $L^2(\mathbb{T})$ onto $L^2(\bbR).$ It follows from 
the \lq\lq Lemma about the image of $H^2$\rq\rq in \cite[p.~253]{Ni86} and the Paley--Wiener Theorem in \cite[p.~254]{Ni86} that the isometry $U$ maps the Hardy space $H^2(\mathbb{T})$ onto the set $H^2_\Pi$, where $H^2_\Pi=\mathcal{F}^{-1}(\chi_{[0,\infty)}L^2(\mathbb{R)})$. 
For the operator $D_0$ in $L^2(\mathbb{T})$ one notes the equality 
$\sgn(D_0)=2P_{H^2(\bbT)}-1$ \cite[p.~317]{Co94}. In addition, $\sgn(D)=2P_{H^2_\Pi}-1$ (see 
\cite[Lemma~3.1]{CHOB82}), where $P_{H^2(\bbT)}$ and $P_{H^2_\Pi}$ are orthogonal projections on $H^2(\bbT)$ and $H^2_\Pi$ respectively. Thus, one obtains 
\begin{equation}\lb{UsgnD}
U{\rm sgn}(D_0)=U(2P_{H^2(\bbT)}-1)=(2P_{H^2_\Pi}-1)U = \sgn(D)U.
\end{equation}

It follows directly from the definition of $U$ that
\begin{align} 
\begin{split}  
& (U\circ (\psi\circ \gamma^{-1}) f)(x) = (U(\psi\circ\gamma^{-1}(\cdot)f)(x))    \\
& \quad = \pi^{-1/2}(x+i)^{-1} \psi(x)f(\gamma(x))=\psi (Uf)(x), \quad 
f \in L^2(\bbR), 
\end{split} 
\end{align} 
and hence,
$
U^{-1} \psi U= \psi\circ\gamma^{-1}. 
$
Consequently, 
\begin{align}\lb{nikolskii pass}
\begin{split}
[\psi,{\rm sgn}(D)]&=UU^{-1}\psi\sgn(D)UU^{-1}-\sgn(D)UU^{-1}\psi UU^{-1}\\
&\stackrel{\eqref{UsgnD}}{=}
U[U^{-1} \psi U,{\rm sgn}(D_0)]U^{-1}=U^{-1}[ \psi\circ\gamma^{-1},{\rm sgn}(D_0)]U.
\end{split}
\end{align}
Since $U$ is an isometry from $L^2(\bbR)$ onto $L^2(\bbR)$, the claim follows. 
\end{proof}

We note that $\psi\circ\gamma^{-1}\in BMO(\bbT),$ or, equivalently, $\psi\in BMO(\bbR)$ 
(see, e.g., \cite[Ch.~6, Corollary~1.3]{Garnett}) is a sufficient and necessary condition for 
$[\psi\circ\gamma^{-1},\sgn(D_0)] \in \cB\big(L^2(\bbT)\big)$ and 
$[\psi,\sgn(D)] \in \cB\big(L^2(\bbR)\big)$ (see, e.g., \cite{Pe82}, 
\cite[Theorem~1.1.10]{Pe03}). 

Before proceeding to the proof of Proposition \ref{g_versus_sgn}, we recall the 
definition of the Besov space $B_1^1(\bbT)$ (see e.g. \cite{Pe05}), 
\begin{equation} 
B_1^1(\bbT) = \big\{h \in L^1(\bbT) \, \big| \, 
\int_\bbT\frac{\|h(\tau t)-h(t)\|_1}{|1-\tau|^2}\, d{\bf m}(\tau)<\infty\bigg\}, 
\end{equation} 
where ${\bf m}$ denotes the (normalized) Lebesgue measure on $\bbT$. 

\begin{proof}[Proof of Proposition \ref{g_versus_sgn}]
By Proposition \ref{g_complex}, 
$[g(A_+)-g(A_-)] \in\cB_1\big(L^2(\bbR)\big)$. Unitary equivalence of the operator $D+ \phi$ and $D$ (see Remark \ref{Psi}\,$(ii)$) implies that 
$g(A_+)-g(A_-)=[\psi,g(D)] \ol{\psi}$, and hence 
$[\psi,g(D)] \in \cB_1\big(L^2(\bbR)\big)$.

On the other hand, by Peller's theorem, \cite{Pe80},
$[\psi\circ\gamma^{-1},\sgn(D_0)] \in \cB_1\big(L^2(\bbT)\big)$ if and only if 
$\psi\circ\gamma^{-1}\in B_1^1(\bbT)$. However, since 
\begin{equation} 
\psi(+\infty)=\exp\bigg(-i\int_0^\infty \phi (x) \, dx\bigg) 
\neq \exp\bigg(-i\int_0^{-\infty} \phi(x) \, dx\bigg)
=\psi(-\infty), 
\end{equation} 
and since the Cayley transform maps $\pm\infty$ to the same point $1\in\bbT$, the function $\psi\circ\gamma^{-1}$ is discontinuous at the point $1\in\bbT$. Since functions in the Besov class $B_1^1(\bbT)$ are necessarily continuous, one concludes that 
$[\psi\circ\gamma^{-1},\sgn(D_0)]\notin\cB_1\big(L^2(\bbT)\big)$.
\end{proof}

\section{Some Useful Technicalities} \lb{sB}

The following lemma states a well-known fact, for which we have not been able to find a convenient reference. The 
symbol $B_p((a,b),\mnn{\bbR})$ denotes the Bochner $L^p$-space of all 
$\mnn{\bbC}$-valued functions on $(a,b) \subseteq \bbR$, that is,
\begin{align} 
&B_p((a,b),M^{m \times m}(\mathbb{C}))    \\
& \quad = \bigg\{F\colon (a,b) \to M^{m \times m}(\mathbb{C}) \, 
\text{measurable} \, \bigg|  
\int_{(a,b)} \|F(x)\|^p_{M^{m \times m}(\mathbb{C})} \, dx < \infty\bigg\}.   \no 
\end{align} 
In the following we denote by $P_{j,k}$, $j,k =1,\dots,m$, the projection onto the $j,k$-th element 
in a (block) matrix, such that
\begin{equation}
P_{j,k} P_{\ell,r} = \delta_{k, \ell} P_{j,r}, \quad j,k,\ell,r = 1,\dots,m.   \lb{Pjk} 
\end{equation}

\begin{lemma} \lb{matrixlemma}
Let $m \in \bbN$, $p \in [1,\infty) \cup \{\infty\}$, $(a,b) \subseteq \bbR$, and suppose that 
$F$ is an $m \times m$ matrix-valued function with complex-valued entries. Then, 
\begin{equation} 
F \in B_p((a,b),\mnn{\bbC}) \, \text{ if and only if } \, F\in \mnn{L^p(a,b)}.
\end{equation}  
In this case, the corresponding norms are equivalent.
\end{lemma}
\begin{proof}
We note that $F$ is a Bochner measurable $\mnn{\bbC}$-valued function if and only if each entry $F_{j,k}$, $j,k=1,\dots,m$, of the matrix $F$ is a measurable function on $(a,b)$. Indeed, since by the definition each Bochner measurable function $F$ can be approximated (in $\mnn{\bbC}$-norm) by step functions,  the inequality 
$|F_{j,k}(x)|\leq \|F(x)\|_{\mnn{\bbR}}$ for a.e.~$x \in (a,b)$, 
shows that Bochner measurability 
of $F:(a,b)\to \mnn{(a,b)}$ implies measurability of each function 
$F_{j,k}(\cdot)$, $j,k = 1,\dots,m$. On the other hand, by the fact 
$F =\sum_{j,k,\ell, r=1}^m P_{k,\ell} F_{\ell,j}P_{j,r}$, with $P_{j,k}$ the projections 
introduced in \eqref{Pjk}, one concludes Bochner measurability of $F$ from measurability 
of all $F_{j,k}$, $j,k=1,\dots,m$.

Next, let $F\in\mnn{L^p(a,b)}$. Then 
\begin{equation} 
\int_{(a,b)} |F_{j,k}(x)|^p\,dx \leq \int_{(a,b)} \|F(x)\|_{\mnn{\bbC}}^p\,dx < \infty,
\end{equation}   
that is, $F_{j,k}\in L^p(a,b)$ for all $j,k =1,\dots,m$. 
Conversely, 
\begin{align}
& \bigg(\int_{(a,b)} \|F(x)\|_{\mnn{\bbC}}^p\,dx\bigg)^{1/p}   \no \\
& \quad =\bigg(\int_{(a,b)} \bigg\|\sum_{j,k,\ell,r=1}^m P_{k,\ell}F_{\ell,j}(x)P_{j,r}\bigg\|_{\mnn{\bbC}}^p\,dx\bigg)^{1/p}   \no \\
& \quad \leq \sum_{j,k,\ell,r=1}^m\bigg(\int_{(a,b)} |F_{\ell,j}(x)|^p\,dx\bigg)^{1/p}<\infty.
\end{align}
\end{proof}

The next result guarantees that $F \in \mnn{L^{1}(\bbR)}$ 
is equivalent to $|F|^{1/2} \in \mnn{L^{2}(\bbR)}$. Here, as usual, we 
write $|F|=(F^*F)^{1/2}$ for matrices and bounded operators. The following result immediately follows from 
Lemma \ref{square_root} and standard properties of Bochner $L^p$-spaces.

\begin{corollary}\lb{square_root}
Let $m \in \bbN$, $p \in [1,\infty) \cup \{\infty\}$, $(a,b) \subseteq \bbR$.  
$F \in \mnn{L^p(\bbR)}$ if and only if  
$|F|^{1/2} \in \mnn{L^{2p}(\bbR)}$. In addition, 
$\big\|F^{1/2}\big\|_2^2\leq\ C \, \|F\|_1$. 
\end{corollary}
Both Lemma \ref{matrixlemma} and Corollary \ref{square_root} extend to arbitrary measure spaces. 

\begin{remark}\lb{viamax}
We recall that any bounded operator $A$ on $L^2(\bbR)\otimes \bbC^m$ can be represented as 
a matrix $\{A_{j,k}\}_{j,k=1}^m$, where the entries $A_{j,k}$ are bounded operators on 
$L^2(\bbR)$. In particular, for arbitrary $p\geq 1$, 
\begin{align}
\|A\|_{\cB_p(L^2(\bbR)\otimes\bbC^m)} 
&= \bigg\|\sum_{j,k,\ell,r =1}^mP_{k,l}A_{lj}P_{j,r}\bigg\|_{\cB_p(L^2(\bbR)\otimes\bbC^m)} 
\leq \sum_{\ell,j=1}^m\|A_{\ell,j}\|_{\cB_p(L^2(\bbR))}   \no \\
& \leq\ C \, \max_{\ell,j=1,\dots,m}\|A_{\ell,j}\|_{\cB_p(L^2(\bbR))}. 
\end{align} 
${}$ \hfill $\diamond$
\end{remark}

For various trace class estimates we need a simple sufficient condition for a function on $\bbR$  to be contained in the space $\ell^p(L^2(\bbR))$, $1\leq p\leq 2$, introduced in \cite{BS69}. We recall that the space $\ell^p(L^2(\bbR))$, $1\leq p\leq 2$, consists of all functions such that 
\begin{equation} 
\|f\|_{\ell^p(L^2(\bbR))}=\bigg(\sum_{n\in\bbZ}\|f\chi_{[n-1,n]}\|_2^p\bigg)^{1/p}<\infty. 
\end{equation} 

\begin{lemma}\lb{Cwikel_par} If $f\in W^{1,p}(\bbR)\cap C_b(\bbR), 1\leq p<\infty$,  then $f\in \ell^p(L^2(\bbR))$ and
\begin{equation} 
\|f\|_{\ell^p(L^2(\bbR))}\leq\ C \, \|f\|_{1,p}. 
\end{equation} 
\end{lemma}
\begin{proof}
Since $f$ is continuous, for every $n\in\mathbb{Z}$, there exists  $x_n\in[n,n+1]$, such that
\begin{equation} 
|f(x_n)|=\bigg(\int_n^{n+1}|f(x)|^p\,dx\bigg)^{1/p}. 
\end{equation} 
For every $x\in[n,n+1]$, one has  
\begin{align}
|f(x)|&= \bigg|f(x_n)+\int_{x_n}^xf'(s)\,ds\bigg|\leq |f(x_n)|+\int_{x_n}^x|f'(s)|\,ds  \no \\
&\leq\bigg(\int_n^{n+1}|f(s)|^p \,ds\bigg)^{1/p}+\int_n^{n+1}|f'(s)|\,ds   \no \\
&\leq\bigg(\int_n^{n+1}|f(s)|^p \,ds\bigg)^{1/p}+\bigg(\int_n^{n+1}|f'(s)|^p \,ds\bigg)^{1/p}  \no \\
&\leq 2^{1-1/p}\bigg(\int_n^{n+1}|f(s)|^p \,ds+\int_n^{n+1}|f'(s)|^p \,ds\bigg)^{1/p}.
\end{align}
Thus,  
\begin{align}
& \|f\|_{\ell^p(L^2(\bbR))} \leq \|f\|_{\ell^p(L^\infty(\bbR))} 
=\sum_{n\in\mathbb{Z}}\sup_{x\in[n,n+1)}|f(x)|^p    \no \\
& \quad \leq \ C \, \sum_{n\in\mathbb{Z}}  \bigg(\int_n^{n+1}|f(x)|^p \,dx
+\int_n^{n+1}|f'(x)|^p \,dx\bigg)^{1/p}   \\
& \quad \leq \ C \, ( \|f\|_p+\|f'\|_p).  \qedhere
\end{align}
\end{proof}

We will also recall the following lemma, which, together with pertinent hints to the literature, 
can be found in \cite{GLMST11}.

\begin{lemma}\lb{so_L_p}
Let $p\in[1,\infty)$ and assume that $R,R_n,T,T_n\in\cB(\cH)$, $n\in\bbN$, satisfy
$\slim_{n\to\infty}R_n = R$  and $\slim_{n\to \infty}T_n = T$ and that
$S,S_n\in\cB_p(\cH)$, $n\in\bbN$, satisfy 
$\lim_{n\to\infty}\|S_n-S\|_{\cB_p(\cH)}=0$.
Then $\lim_{n\to\infty}\|R_n S_n T_n^\ast - R S T^\ast\|_{\cB_p(\cH)}=0$.
\end{lemma}
Finally we note the following (known) fact.
The  principle value of $\frac{e^{-|x|}}{x}$ (in the sense of distributions), abbreviated by $\pv\frac{e^{-|x|}}{x}$, is a tempered distribution, and hence the convolution on right-hand side of the equality in Lemma \ref{pv_arctan} is well-defined. 
The result follows from using the fact that  $\arctan(\cdot)$ is bounded,  so we may regard it as a tempered distribution 
(see, e.g., \cite[Section I.3]{SW71}). We consider the principle value of $\frac{e^{-|x|}}{x}$ introduced by the equality 
\begin{equation}\lb{def_pv}
\pv \frac{e^{-|x|}}{x}(\eta)
=\lim_{\varepsilon\downarrow 0}\int_{|x|>\varepsilon}\frac{e^{-|x|}\eta(x)}{x}\, dx, \quad \eta\in S(\bbR).
\end{equation}
This is a tempered distribution since for arbitrary $\eta\in S(\bbR)$, 
\begin{align}
\pv \frac{e^{-|x|}}{x}(\eta)&=\lim_{\varepsilon\downarrow 0}\int_{\varepsilon<|x|<1}\frac{e^{-|x|}\eta(x)}{x} \, dx+\int_{|x|>1}\frac{e^{-|x|}\eta(x)}{x} \, dx   \no \\
&=\lim_{\varepsilon\downarrow 0}\int_{\varepsilon<|x|<1}\frac{e^{-|x|}(\eta(x)-\eta(0))}{x} \, dx+\eta(0)\lim_{\varepsilon\downarrow 0}\int_{\varepsilon<|x|<1}\frac{e^{-|x|}}{x} \, dx   \no \\
&\quad + \int_{|x|>1}\frac{e^{-|x|}\eta(x)}{x} \, dx, 
\end{align}
and since the next to last integral equals zero, 
\begin{equation} 
\bigg|\pv \frac{e^{-|x|}}{x}(\eta)\bigg|\leq\ C \, \big[\|\eta'\|_\infty+\|\eta\|_\infty\big]. 
\end{equation} 
Thus, by \cite[Section~1.3, Theorem~3.11]{SW71}, $\pv \frac{e^{-|x|}}{x}$ is a tempered distribution.

\section{The Gaiotto--Moore--Witten Model: A Counterexample}      \lb{sC} 

We briefly review a Dirac-type operator that arose recently in \cite{GMW15} in connection with 
Ginzburg--Landau models. We will consider the very simplest instance of the general theory 
developed in \cite{GMW15} and 
show that it represents a model that violates even the most basic of our trace class hypotheses 
and hence we dubbed it a ``counterexample''. 

First, consider a family of model operators of the form
\begin{equation} \lb{model2}
\wti A(t)= i \frac{d}{dx} \sigma_3 + V(t, \, \cdot \,) 
\begin{pmatrix}  0 & 0\\ 1 & 0 \end{pmatrix} 
 + \ol{V(t,\, \cdot \,)} \begin{pmatrix}  0 & 1\\ 0 & 0 \end{pmatrix}, \quad t \in \bbR,
\end{equation} 
acting in $L^2(\bbR) \otimes \bbC^2$. (Here $\sigma_3= \big(\begin{smallmatrix}  1 & 0 \\ 0& -1 \end{smallmatrix}\big)$ denotes a standard Pauli matrix.) 
Provided the complex-valued function $V$ on $\bbR^2$ is
chosen so that the asymptotes $\wti A_\pm$ exist in norm resolvent sense as $t \to \pm\infty$
then the theory in \cite{CGLS14} applies. 
Examples of operators of this form appear in studies of Andreev reflection in superconductivity
\cite{St96}.
While this class is quite interesting in its own right we will not further pursue the study of this general family 
$\big\{\wti A(t)\big\}_{t \in \bbR}$, but rather focus on a far simpler situation that illustrates
why strong assumptions of the kind described earlier in this paper are needed.

Returning to \cite{GMW15}, we show by using an  explicit particular case of the general situation from \cite{GMW15} that the  hypotheses of \cite{CGLS14} cannot hold for the examples considered there. The simplest form of the Ginzburg--Landau equation in one dimension is, 
\begin{equation} 
\xi^2 \psi'' = \psi^3 - \psi, 
\end{equation} 
where $\xi$ is a constant and we seek real-valued solutions on $\bbR$. 
A summary with references to further literature is in \cite{BY98}. The solutions
are given by elliptic functions, specifically, 
\begin{equation} 
\psi(x)=\big[2k^2(1+k^2)^{-1}\big]^{1/2} \, \mbox{sn}\big((x-x_0)(\xi(1+k^2)^{1/2})^{-1};k\big), 
\quad x \in \bbR,     \lb{sn} 
\end{equation} 
where $x_0 \in \bbR$, $k \in [0,1]$, are integration constants with $k$ being the elliptic modulus of the Jacobi elliptic function $\mbox{sn}(\, \cdot \,; \, \cdot \,)$, see\ \cite[Ch.\ 22]{OLBC10}. 

In \cite{GMW15} only solutions with well defined  asymptotic values at $\pm\infty$ are considered.
As 
$\mbox{sn}$ is periodic, it does not have well-defined asymptotics at $\pm\infty$, except in the limiting  case $k=1$, when the solution \eqref{sn} degenerates into (cf.\ \cite[p.~555]{OLBC10}), 
\begin{equation} \lb{tanh}
\psi(x)=\tanh\big((x-x_0)/\big(2^{1/2} \xi\big)\big), \quad x \in \bbR.
\end{equation} 
We will only consider this special case in the following discussion and for simplicity employ the normalization $x_0 = 0$ and $\xi = 2^{-1/2}$.

In \cite{GMW15} the solutions to the Ginzburg--Landau equation are used to define
an associated Dirac-type operator. 
Following \cite{GMW15} for this very simple situation, introduce a polynomial $W$ in one variable  (termed the `super-potential').
One also has a `boosted soliton' $\phi$, constructed as a function on $\bbR^2$
from a function $\psi$ on $\bbR$ that is  a solution of the Ginzburg--Landau
equation, by defining $\phi$ on $\bbR^2$ by 
\begin{equation} \lb{boost}
\phi(t, x)=\psi(x\cos(\theta) + t \sin(\theta)), \quad (t,x) \in \bbR^2, 
\end{equation} 
where $\theta \in (0,\pi) \backslash \{\pi/2\}$ is a parameter (called the `boost variable' in 
\cite{GMW15}). (We chose $\theta \neq \pi/2$ to avoid the case that $\phi(t,x)$ becomes 
$x$-independent.) Next we  form  the composition $W''\circ \phi \equiv W''(\phi)$. 

The existence of well-defined pointwise limits 
$\psi_\pm$ for the solution (\ref{tanh}) at spatial infinity and the definition of $\phi$ in 
 (\ref{boost}) imply that
 as $t\to\pm\infty$, pointwise limits $\phi_\pm$, exist 
and that these are independent of the variable $x$.
In turn these fix $W_\pm= W(\phi_\pm)$. 

In \cite{GMW15} the notation $\zeta$ is used for the  additional 
parameter  which is the product of $e^{i\theta}$ with the phase of $W_+-W_-$
and the authors study a special case of the operator (\ref{model2})
by  introducing the family of Dirac-type operators
\begin{align} 
\begin{split} 
& \, A(t)= i \frac{d}{dx} \sigma_3 +  2^{-1} \ol{\zeta}W''(\phi(t, \, \cdot \, )) 
\begin{pmatrix}  0 & 0\\ 1 & 0 \end{pmatrix}  
+ 2^{-1} \zeta \ol{W''(\phi(t, \, \cdot \,))} \begin{pmatrix}  0 & 1\\ 0 & 0 \end{pmatrix},     \\[1mm] 
& \dom(A(t)) = W^{1,2}(\bbR) \otimes \bbC^2, \quad  t \in \bbR, 
\end{split} 
\end{align} 
is self-adjoint in $L^2(\bbR) \otimes \bbC^2$.  

We emphasize that we specialise here to the $\phi$-constructed from the solution $\psi$ in 
(\ref{tanh}) in order to illustrate the general case.

If one treats $t \in \bbR$ as a flow parameter (for the spectral flow for example) then for 
fixed $(x, \theta) \in \bbR \times [(0,\pi)\backslash\{\pi/2\}]$, one infers the pointwise limits
\begin{equation} 
\lim_{t \to \pm \infty}\phi(t, x) = \lim_{t \to \pm \infty}\psi (x \cos(\theta) + t \sin(\theta)) = \pm 1. 
\lb{C.7} 
\end{equation} 
These limits are not uniform in $(x, \theta) \in \bbR \times [(0,\pi)\backslash\{\pi/2\}]$ (e.g.,  
choose $x = - \tanh(\theta) t$). 

The next step is, following \cite{CGLS14}, to introduce the model operator 
$\bsD_\bsA= (d/dt) + \bsA$ acting in $L^2(\bbR; L^2(\bbR) \otimes \bbC^2)$: 
If one chooses the function $W$ to be a cubic (say, $W(x) =x^3/3$) then for this example the parameter $\zeta=e^{i\theta}$ and the operator  $A(t)$, $t \in \bbR$, in $L^2(\bbR) \otimes \bbC^2$ 
becomes (for $\psi$ as in \eqref{tanh}),
\begin{equation}  
A(t)= i \frac{d}{dx} \sigma_3 
+  \phi(t,\, \cdot \,) \begin{pmatrix}  0 & e^{i\theta} \\ e^{-i\theta} & 0 \end{pmatrix}, 
\quad \dom(A(t)) = W^{1,2}(\bbR) \otimes \bbC^2, \;  t \in \bbR.
\end{equation} 
The asymptotic operators $A_\pm$ in $L^2(\bbR) \otimes \bbC^2$ are then given by
\begin{equation} 
A_{\pm} = i \frac{d}{dx} \sigma_3 \pm  
\begin{pmatrix} 0 & e^{i\theta}\\ e^{-i\theta} & 0 \end{pmatrix},   \quad 
\dom(A_{\pm}) = W^{1,2}(\bbR) \otimes \bbC^2.  
\end{equation} 
They are unitarily equivalent as $A_-=\sigma_3A_+\sigma_3$. Employing the identity,
\begin{align} 
& (A(t) - z I)^{-1} - (A_{\pm} - z I)^{-1}   \\ 
& \quad = - (A(t) - z I)^{-1} [\phi(t, \, \cdot \,) \mp 1] \begin{pmatrix} 
0 & e^{i \theta} \\ e^{- i \theta} & 0 \end{pmatrix} 
(A_{\pm} - z I)^{-1}, \quad t \in \bbR, \; z \in \bbC \backslash \bbR,    \no 
\end{align} 
and the elementary estimate 
$\big\|(S - z I)^{-1}\big\|_{\cB(\cH)} \leq |\Im(z)|^{-1}$, $z \in \bbC \backslash \bbR$, for any 
self-adjoint operator $S$ in $\cH$, proves strong resolvent convergence of  $A(t)$ to 
$A_{\pm}$ as $t \to \pm \infty$. (However, since $\sup_{(x,t) \in \bbR^2} |\phi(t,x) -1| = 1$, 
one does not obtain norm resolvent convergence in this example as $t \to \pm \infty$.) 

Since, 
\begin{equation} 
A_\pm^2= \bigg(- \frac{d^2}{dx^2} + I\bigg) I_2, \quad 
\dom \big(A_{\pm}^2\big) = W^{2,2}(\bbR) \otimes \bbC^2,  
\end{equation}
one concludes that 
\begin{equation} 
A_\pm^{-1} \in \cB\big(L^2(\bbR) \otimes \bbC^2\big),    \lb{C.12}
\end{equation}  
and hence $A_{\pm}$ are Fredholm in $L^2(\bbR) \otimes \bbC^2$. Moreover,
the fact,
\begin{align} 
\begin{split} 
(A_+ - z I)^{-1} - (A_- - z I)^{-1} = - 2 (A_+ - z I)^{-1}  
\begin{pmatrix} 0 & e^{i \theta} \\ e^{- i \theta} & 0 \end{pmatrix} (A_- - z I)^{-1},&    \\
z \in \bbC \backslash \bbR,&  
\end{split} 
\end{align}
and differentiating this identity with respect to $z$ implies that 
\begin{equation}
\big[(A_+ - z I)^{-r} - (A_- - z I)^{-r}\big] \notin \cB_{\infty}\big(L^2(\bbR) \otimes \bbC^2\big), 
\quad r \in \bbN, \; z \in \bbC \backslash \bbR.
\end{equation} 

Noticing that 
\begin{equation}
A(t)^2 = \bigg(- \frac{d^2}{dx^2} + I\bigg) I_2 + \big[\phi(t, \, \cdot \,)^2 - 1\big] I_2 
+  i \phi_x(t, \, \cdot \,) \begin{pmatrix} 0 & e^{i\theta}\\ - e^{-i\theta}  & 0 \end{pmatrix}, 
\quad t\in \bbR,    \lb{C.15} 
\end{equation} 
and using the fact that the functions in the zeroth-order terms of $A(t)^2$ vanish rapidly at spatial infinity, one concludes that 
\begin{align} 
\begin{split} 
& \bigg(-\frac{d^2}{dx^2} +I\bigg)^{-1} I_2  
\bigg[\big[\phi(t, \, \cdot \, )^2-1\big] I_2 +  i \phi_x(t,\,\cdot\,) 
\begin{pmatrix}  0 & e^{i\theta}\\ - e^{-i\theta}  & 0 \end{pmatrix}\bigg]   \\
& \quad \in \cB_{\infty}\big(L^2(\bbR) \otimes \bbC^2\big), \quad t \in \bbR, \; 
\theta \in (0,\pi) \backslash \{\pi/2\}, 
\end{split}
\end{align} 
is compact. Hence, $\big(-\frac{d^2}{dx^2} +I \big)^{-1} I_2$ is a parametrix for $A(t)^2$, $t\in\bbR$. 
It follows that both $A(t)^2$ and hence $A(t)$, $t \in \bbR$, are Fredholm, and hence, the spectral flow along the path $\{A(t)\}_{t \in \bbR}$ is well-defined. (We emphasize that 
$\theta \in (0,\pi) \backslash \{\pi/2\}$ throughout this appendix.)  

 Relation \eqref{C.15} and the asymptotic behavior \eqref{C.7} of $\phi(t,x)$ as $|x|\to\infty$ 
also proves that 
\begin{equation}
\sigma_{\rm ess} \big(A(t)^2\big) = [1,\infty), \quad t \in \bbR,    \lb{C.17} 
\end{equation}
and hence 
\begin{equation}
\sigma_{\rm ess}(A(t)) = (-\infty,-1] \cup [1,\infty), \quad t \in \bbR.   \lb{C.18}
\end{equation}

Next we prove that $A(t)$ has a one dimensional kernel for any real $t \in \bbR$. Indeed, 
\begin{equation} 
A(t)\psi=0, \quad \psi (t,\,\cdot\,) = (\psi_1(t,\,\cdot\,), \, \psi_2(t,\,\cdot\,))^\top 
\in W^{1,2}(\bbR) \otimes \bbC^2, \;  t \in \bbR,    \lb{C.19} 
\end{equation} 
is equivalent to
\begin{align} 
\begin{split} 
 i \psi_{1,x}(t,x) + e^{i\theta} \tanh(x\cos\theta +t\sin\theta) \psi_2(t,x)&= 0,   \\
 -i \psi_{2,x}(t,x) + e^{-i\theta}\tanh(x\cos\theta +t\sin\theta) \psi_1(t,x) &= 0.   \lb{C.20}
\end{split} 
\end{align}
This yields 
\begin{align} 
\psi_1(t,x) &= \begin{cases} 
[\cosh(x\cos(\theta) + t \sin(\theta))]^{-1/\cos(\theta)}, & \theta \in (0, \pi/2), \\
[\cosh(x\cos(\theta) + t \sin(\theta))]^{1/\cos(\theta)}, & \theta \in (\pi/2,\pi),    \lb{C.21} 
\end{cases}   \\ 
\psi_2(t,x) &= \begin{cases} ie^{-i\theta}\psi_1(t,x), & \theta \in (0, \pi/2), \\ 
- ie^{-i\theta}\psi_1(t,x), & \theta \in (\pi/2,\pi),       \lb{C.22} 
\end{cases} 
\quad (t,x) \in \bbR^2. 
\end{align}
Since $A(t)$, $t \in \bbR$ is in the limit point case at $x = \pm \infty$, there can be no second, 
linearly independent $L^2(\bbR) \otimes \bbC^2$-solution of $A(t) \psi = 0$, and hence $0$ is a 
simple eigenvalue of $A(t)$ for all $t \in \bbR$. In particular, the kernel of $A(t)$ is given by
\begin{equation}
\ker(A(t)) = {\rm lin.span} \big\{\psi (t,\,\cdot\,) = (\psi_1(t,\,\cdot\,), \, \psi_2(t,\,\cdot\,))^\top\big\}, 
\quad t \in \bbR,   \lb{C.23} 
\end{equation}
with $\psi_j$ given by \eqref{C.19}, \eqref{C.20}. 

Since by \eqref{C.18}, $A(t)$ has the essential spectral gap $(-1,1)$ and an isolated eigenvalue $0$ for all $t \in \bbR$, the spectral flow for the family $\{A(t)\}_{t \in \bbR}$ is actually zero.  

One observes that the existence of an eigenvalue zero of $A(t)$ for all $t \in \bbR$, yet the fact 
that $\ker(A_{\pm}) = \{0\}$ by \eqref{C.12}, underscores that $A(t)$ cannot converge to $A_{\pm}$ 
in norm resolvent case as $t \to \pm \infty$. 

It is straightforward to check that 
\begin{align} 
\begin{split} 
& \, \bsH_j= \bsH_0+2\bsB\bsA_- +[\bsA_-,\bsB] +\bsB^2+(-1)^j\bsB',     \\ 
&\dom(\bsH_j) = W^{2,2}(\bbR^2) \otimes \bbC^2, \; j=1,2, 
\end{split} 
\end{align} 
is self-adjoint in $L^2(\bbR^2) \otimes \bbC^2$, where 
\begin{align} 
& \bsH_0 = - \bigg(\frac{\partial^2}{\partial t^2} + \frac{\partial^2}{\partial x^2} +I\bigg) I_2, \quad 
\dom(\bsH_0) = W^{2,2}(\bbR^2) \otimes \bbC^2, \\ 
& B(t) = [\phi(t,\, \cdot \,) +1] \begin{pmatrix} 0 & e^{i\theta}\\ e^{-i\theta}  & 0 \end{pmatrix},  
\quad t \in \bbR,   \\
& B(t)^2= [\phi(t,\, \cdot \,)+1]^2{I_2}, \quad
B'(t) = \phi_t(t,\, \cdot \,)
\begin{pmatrix} 0 & e^{i\theta} \\ e^{-i\theta} & 0 \end{pmatrix}, \quad t \in \bbR,   \lb{C.27} \\
& [A_-, B(t)] = i \phi_x(t,\, \cdot \,) \begin{pmatrix} 0 &- e^{i\theta}\\ e^{-i\theta} & 0 \end{pmatrix} 
+ 2i [\phi(t,\, \cdot \,)+1] \begin{pmatrix} 0 &- e^{i\theta}\\ e^{-i\theta} & 0 \end{pmatrix}\frac
{\partial}{\partial x}.      \lb{C.28}
\end{align} 

Thus, 
\begin{align}
& \|B'(t)\|_{\cB(L^2(\bbR) \otimes \bbC^2)} = \|\phi_t(t,\, \cdot \,)]\|_{L^{\infty}(\bbR;dx)}   \no \\
& \quad = {\rm ess.sup}_{x \in \bbR} \big|\sin(\theta) [\cosh(x \cos(\theta) + t \sin(\theta))]^{-2}\big| 
\no \\
& \quad = \sin(\theta), \quad \theta \in (0,\pi/2) \cup (\pi/2,\pi),  
\end{align}
is independent of $t \in \bbR$, and hence hypothesis \eqref{intB'a} is clearly violated. In 
addition, since $B_+$ is  the operator of multiplication by a constant matrix, 
Hypothesis \ref{h3.4}\,$(ii)$ cannot hold as well. 

However, strictly speaking, in order to test the applicability of \cite[Theorem~2.6]{CGPST15} (or
\cite[Corollary~8.4]{GLMST11}) to decide the Fredholm property of $\bsD^{}_{\bsA}$,  we should 
compare with the relative trace class hypotheses employed in \cite{CGPST15} and 
\cite{GLMST11} which reads, 
\begin{equation}
\big\|B'(\cdot) (|A_-| + I)^{-1}\big\|_{\cB_1(L^2(\bbR) \otimes \bbC^2)} \in L^1(\bbR; dt).     \lb{C.30}
\end{equation} 
From the outset it is clear from \cite[Remark~(a) on p.~39]{Si05} that  
\begin{equation}
B'(t) (|A_-| + I)^{-1} \notin \cB_1\big(L^2(\bbR) \otimes \bbC^2\big), \quad t \in \bbR,
\end{equation}
since $(| \cdot| + 1)^{-1} \notin L^1(\bbR)$, rendering condition \eqref{C.30} a moot point. (Here we ignored the $2 \times 2$ matrix structure in $B'(\cdot) (|A_-| + I)^{-1}$ as the latter represents a unitary matrix in $\bbC^2$.) Actually, \eqref{C.30} does not work in the Hilbert--Schmidt context either as an application of the Fourier transform and Plancherel's identity yield
\begin{equation}
\|f(X) g(- i \nabla)\|_{\cB(L^2(\bbR^n)} = (2 \pi)^{- n/2} \|f\|_{L^2(\bbR^n)} \|g\|_{L^2(\bbR^n)}, 
\quad f, g \in L^2(\bbR^n),  
\end{equation}
where $f(X)$ denotes the operator of multiplication by $f(\cdot)$ in $L^2(\bbR^n)$. Thus, one obtains 
\begin{align}
& \big\|B'(\cdot) (|A_-| + I)^{-1}\big\|_{\cB_2(L^2(\bbR) \otimes \bbC^2)}^2 = (2\pi)^{-1} 
\|\phi_t(t, \, \cdot \,)\|_{L^2(\bbR)}^2 \, \big\|(|\cdot| + 1)^{-1}\big\|_{L^2(\bbR)}^2   \no \\
& \quad = \pi^{-1} [\sin(\theta)]^2 \int_{\bbR} dx \, [\cosh(x \cos(\theta) + t \sin(\theta))]^{-4}   \no \\ 
& \quad = \f{4}{3 \pi} \f{[\sin(\theta)]^2}{|\cos(\theta)|},  \quad \theta \in (0,\pi/2) \cup (\pi/2,\pi), 
\end{align}
is again independent of $t \in \bbR$. Thus, we cannot decide at this point in time whether or not 
$\bsD^{}_{\bsA}$ is a Fredholm operator. (If $\bsD^{}_{\bsA}$ were Fredholm, a simple application of the homotopy invariance of the Fredholm index would prove that its index equals zero. Indeeed, upon 
replacing $B(\cdot)$ by $s B(\cdot)$ with $ s \in (0,1]$, observing that the index is clearly zero for 
$0 < s$ sufficiently small, would prove this claim.)

Next, one notes that   
\begin{equation} 
(\bsH_2-z\bsI)^{-1} - (\bsH_1-z\bsI)^{-1} = (\bsH_2-z\bsI)^{-1}2 \bsB'(\bsH_1-z\bsI)^{-1}, 
\quad z \in \bbC \backslash \bbR,    \lb{2.18}
\end{equation}  
with $\bsB'$ the operator of multiplication operator on $L^2(\bbR^2) \otimes \bbC^2$ by a 
$2 \times 2$ matrix whose off-diagonal entries are essentially equal to the  function 
$\phi_t(t, \, \cdot \,)$ (cf.\ \eqref{C.27}). However, since the latter is a function of $(t,x)$ only via 
the combination $x\cos(\theta)+t\sin(\theta)$, where $\theta \in (0,\pi) \backslash \{\pi/2\}$, this 
means that the resolvent difference \eqref{2.18} cannot be trace class. In fact, choosing 
$|z| > 0$ sufficiently large, applying a Neumann series of $(\bsH_j - z\bsI)^{-1}$, $j=1,2$, in 
terms of $(\bsH_0 - z\bsI)^{-1}$, and investigating 
the integral kernel of $(\bsH_0 - z\bsI)^{-1}2 \bsB' (\bsH_0 - z\bsI)^{-1}$, readily proves that the 
latter is not even compact. Once more differentiation with respect to $z$ then implies  
\begin{equation} 
\big[(\bsH_2-z\bsI)^{-r} - (\bsH_1-z\bsI)^{-r}\big] \notin 
\cB_{\infty}\big(L^2(\bbR^2) \otimes \bbC^2\big), 
\quad r \in \bbN, \; z \in \bbC \backslash \bbR. 
\end{equation}  

So the Witten index of Section \ref{WI_section} is not defined, and neither are the spectral shift functions for the pairs $(A_+, A_-)$ and $(\bsH_2, \bsH_1)$. This demonstrates the existence of an elementary example which is 
not amenable to the techniques discussed in the bulk of this paper. Of course, this is not 
really surprising as no single technique based on scattering theoretic concepts can be expected 
to handle all Fredholm, let alone, non-Fredholm, situations.  

\medskip 
 
\noindent 
{\bf Acknowledgments.}  We are indebted to the anonymous referee for a careful reading 
of our manuscript and for his numerous and helpful suggestions. We thank Greg Moore for 
showing us a preliminary version of the manuscript \cite{GMW15}.
A.C., F.G., D.P., and F.S.\ thank the Erwin Schr\"odinger International 
Institute for Mathematical Physics (ESI), Vienna, Austria, for funding support 
for this collaboration in form of a Research-in Teams project, ``Scattering Theory 
and Non-Commutative Analysis'' for the duration of June 22 to July 22, 2014. 
F.G. and G.L.\ are indebted to Gerald Teschl and the ESI for a kind invitation to visit the 
University of Vienna, Austria, for a period of two weeks in June/July of 2014. 
A.C., G.L., and F.S.\ gratefully acknowledge financial support from the Australian Research 
Council. A.C.\ also thanks the Alexander von Humboldt 
Stiftung and colleagues at the University of M\"unster. 

 
\end{document}